\documentclass[a4paper]{amsart}
\usepackage[style=alphabetic,backend=biber, sorting=nyt]{biblatex}
\bibliography{references}

\usepackage{amssymb}
\usepackage{bbm}
\usepackage{amsmath}
\usepackage{enumitem}
\usepackage{tikz-cd}
\usepackage{faktor}

\usepackage{amsthm}
\usepackage{thmtools}

\usepackage[colorlinks=true,allcolors=blue!80!black, bookmarksdepth=3]{hyperref}
\usepackage[capitalize,nameinlink,noabbrev]{cleveref}

\usepackage{amscd}

\usepackage{tikz, scalerel}
\def\myDA{\scalerel*{\downarrow}{X}}

\usetikzlibrary{shapes.geometric}
\tikzset{
v/.style={draw, fill, circle, minimum size=1.5mm, inner sep=0},
b/.style={draw , regular polygon,regular polygon sides=4, minimum size=1.5mm, inner sep=.5mm},
e/.style={very thick},
vs/.style={draw, fill, circle, minimum size=1mm, inner sep=0},
bs/.style={draw,  regular polygon,regular polygon sides=4, minimum size=2mm, inner sep=0mm},
es/.style={thick}
}
\newlength{\nodeheight}
\newlength{\nodewidth}
\newlength{\nodehelp}
\newlength{\negnodeheight}
\usetikzlibrary{arrows,matrix,positioning}

\usepackage{amsthm}
\theoremstyle{plain}
\newtheorem{theorem}{Theorem}[section]
\theoremstyle{definition} \newtheorem{definition}[theorem]{Definition}
\theoremstyle{plain}
\newtheorem{lemma}[theorem]{Lemma}
\theoremstyle{plain} \newtheorem{proposition}[theorem]{Proposition}
\theoremstyle{plain} 
\theoremstyle{plain} \newtheorem{corollary}[theorem]{Corollary}
\theoremstyle{remark} 
\theoremstyle{remark} \newtheorem{remark}[theorem]{Remark}
\theoremstyle{remark} 
\theoremstyle{remark} 
\theoremstyle{remark} 
\theoremstyle{definition} 

\renewcommand{\t}{\mathbbm{1}}
\newcommand{\rook}{\mathcal{R}}
\newcommand{\Tor}{\mathrm{Tor}}
\newcommand{\Ext}{\mathrm{Ext}}
\newcommand{\Hom}{\mathrm{Hom}}
\newcommand{\Br}{\mathrm{Br}}
\newcommand{\TL}{\mathrm{TL}}
\newcommand{\norm}[1]{\lvert{#1}\rvert}
\newcommand{\Def}{\mathrm{Def}}
\newcommand{\ch}{\mathrm{char}}
\newcommand{\Span}{\mathrm{Span}}
\newcommand{\Tot}{\mathrm{Tot}}

\DeclareMathOperator{\OutermostCupComplex}{S}
\newcommand{\CPL}{L}
\newcommand{\CPLr}{\overline{L}}

\newcommand{\standardWidth}{1}
\newcommand{\standardHeight}{0.8}

 \subjclass[2020]{
        20J06,
        16E40,
        16E45
    }
    \keywords{Homology, Temperley--Lieb algebras, differential graded algebras} 

\begin{document}

\title{The dga of planar loops when $2n=4$}
\author{Guy Boyde}
\address{Department of Mathematics, Vrije Universiteit Amsterdam, De Boelelaan 1111, 1081 HV Amsterdam, The Netherlands}

\begin{abstract} The dga of planar loops is a pictorial chain complex introduced in a recent paper of the author, Boyd, Randal-Williams, and Sroka, where a minimal model for it was given. In the first nontrivial case, $2n=4$, we give a new model which incorporates two natural `reflection' involutions, as well as a more explicit description of the existing model.
\end{abstract}

\maketitle

\section{Introduction}

In \cite{BH}, Boyd and Hepworth-Young proved vanishing of certain homology groups of Temperley--Lieb algebras in low dimensions. Homology here is defined using $\Tor$, in analogy with the usual notion of group homology, and their result demonstrated (along with \cite{Hepworth}, which treated Iwahori--Hecke algebras) that the methods of `homological stability', hitherto used to study families of groups, can also be fruitfully applied to families of algebras. Many other families of algebras have since been studied (for an overview, see \cite[Section 1.3]{BBRWS}). For Temperley--Lieb algebras, the original vanishing results were extended by Sroka \cite{Sroka}, and then more recently the remaining unknown homology groups were shown to have a rich and nontrivial structure governed by the homology of a certain chain complex: the \emph{dga of planar loops} \cite{BBRWS}. This allowed complete computations in many cases. This chain complex is a new object which is easy to define and interesting in its own right, and here we will study it further.

The definition is pictorial. Fix an even integer $2n \geq 0$. Draw $p$ vertical lines (\emph{bars}) in the plane. Mark $2n$ points (\emph{nodes}) on each bar. A \emph{system of planar loops (pinned by $p$ bars)} is an isotopy class of a number of disjoint loops which hits each node precisely once, and such that each loop hits at least two nodes.

Given a chosen element $a$ of a commutative ring $R$, we may fit these systems of loops together as $p$ varies to form a chain complex of $R$-modules $\CPL(2n) = \CPL(2n;R,a)$, whose component $\CPL(2n)_p$ in degree $p$ has $R$-basis systems of planar loops pinned by $p$ bars: this is the \emph{dga of planar loops}.

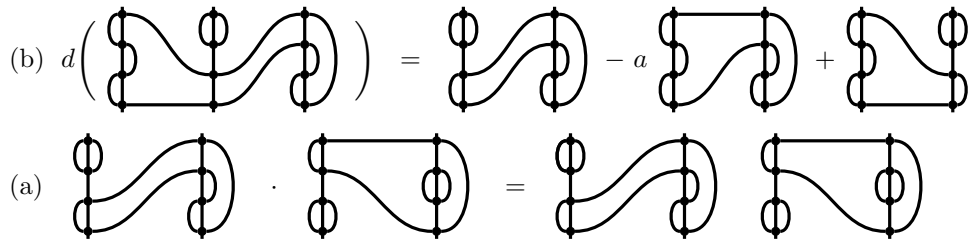
\begin{figure}
\begin{flushleft}
$\mathrm{(a)}$
\
$d \Biggl($
\begin{tikzpicture}[x=1.5cm,y=-.5cm,baseline=-.7cm]
\def\wid{0.8*\standardWidth}
\def\hei{\standardHeight}
\def\nodesize{3}
\def\ang{90}

\node[v, minimum size=\nodesize] (11) at (0* \wid,0*\hei) {};
\node[v, minimum size=\nodesize] (12) at (0* \wid,1*\hei) {};
\node[v, minimum size=\nodesize] (13) at (0* \wid,2*\hei) {};
\node[v, minimum size=\nodesize] (14) at (0* \wid,3*\hei) {};

\draw[e] (11) to[in =180, out =180] (12);
\draw[e] (13) to[in =180, out =180] (14);

\node[v, minimum size=\nodesize] (21) at (1* \wid,0*\hei) {};
\node[v, minimum size=\nodesize] (22) at (1* \wid,1*\hei) {};
\node[v, minimum size=\nodesize] (23) at (1* \wid,2*\hei) {};
\node[v, minimum size=\nodesize] (24) at (1* \wid,3*\hei) {};

\draw[e] (12) to[in =0, out =0] (13);
\draw[e] (11) to[out = 0, in = 180] (23);
\draw[e] (14) to[out = 0, in = 180] (24);
\draw[e] (21) to[out = 180, in = 180] (22);

\node[v, minimum size=\nodesize] (31) at (2* \wid,0*\hei) {};
\node[v, minimum size=\nodesize] (32) at (2* \wid,1*\hei) {};
\node[v, minimum size=\nodesize] (33) at (2* \wid,2*\hei) {};
\node[v, minimum size=\nodesize] (34) at (2* \wid,3*\hei) {};

\draw[e] (21) to[in =0, out =0] (22);
\draw[e] (23) to[out = 0, in = 180] (31);
\draw[e] (24) to[out = 0, in = 180] (32);
\draw[e] (33) to[out = 180, in = 180] (34);

\draw[e] (31) to[in =0, out =0] (34);
\draw[e] (32) to[in =0, out =0] (33);

\draw[very thick] (0* \wid,-0.25*\hei) -- (0* \wid,3.25*\hei);
\draw[very thick] (1* \wid,-0.25*\hei) -- (1* \wid,3.25*\hei);
\draw[very thick] (2* \wid,-0.25*\hei) -- (2* \wid,3.25*\hei);
\end{tikzpicture}
$\Biggr)$
\ \
$=$
\ \
\begin{tikzpicture}[x=1.5cm,y=-.5cm,baseline=-.7cm]
\def\wid{0.8*\standardWidth}
\def\hei{\standardHeight}
\def\nodesize{3}
\def\ang{90}

\node[v, minimum size=\nodesize] (21) at (1* \wid,0*\hei) {};
\node[v, minimum size=\nodesize] (22) at (1* \wid,1*\hei) {};
\node[v, minimum size=\nodesize] (23) at (1* \wid,2*\hei) {};
\node[v, minimum size=\nodesize] (24) at (1* \wid,3*\hei) {};

\draw[e] (24) to[out = 180, in = 180] (23);
\draw[e] (21) to[out = 180, in = 180] (22);

\node[v, minimum size=\nodesize] (31) at (2* \wid,0*\hei) {};
\node[v, minimum size=\nodesize] (32) at (2* \wid,1*\hei) {};
\node[v, minimum size=\nodesize] (33) at (2* \wid,2*\hei) {};
\node[v, minimum size=\nodesize] (34) at (2* \wid,3*\hei) {};

\draw[e] (21) to[in =0, out =0] (22);
\draw[e] (23) to[out = 0, in = 180] (31);
\draw[e] (24) to[out = 0, in = 180] (32);
\draw[e] (33) to[out = 180, in = 180] (34);

\draw[e] (31) to[in =0, out =0] (34);
\draw[e] (32) to[in =0, out =0] (33);

\draw[very thick] (1* \wid,-0.25*\hei) -- (1* \wid,3.25*\hei);
\draw[very thick] (2* \wid,-0.25*\hei) -- (2* \wid,3.25*\hei);
\end{tikzpicture}
$- \ a $
\begin{tikzpicture}[x=1.5cm,y=-.5cm,baseline=-.7cm]
\def\wid{0.8*\standardWidth}
\def\hei{\standardHeight}
\def\nodesize{3}
\def\ang{90}

\node[v, minimum size=\nodesize] (11) at (0* \wid,0*\hei) {};
\node[v, minimum size=\nodesize] (12) at (0* \wid,1*\hei) {};
\node[v, minimum size=\nodesize] (13) at (0* \wid,2*\hei) {};
\node[v, minimum size=\nodesize] (14) at (0* \wid,3*\hei) {};

\draw[e] (11) to[in =180, out =180] (12);
\draw[e] (13) to[in =180, out =180] (14);

\node[v, minimum size=\nodesize] (21) at (1* \wid,0*\hei) {};
\node[v, minimum size=\nodesize] (22) at (1* \wid,1*\hei) {};
\node[v, minimum size=\nodesize] (23) at (1* \wid,2*\hei) {};
\node[v, minimum size=\nodesize] (24) at (1* \wid,3*\hei) {};

\draw[e] (12) to[in =0, out =0] (13);
\draw[e] (11) to[out = 0, in = 180] (21);
\draw[e] (14) to[out = 0, in = 180] (22);
\draw[e] (23) to[out = 180, in = 180] (24);

\draw[e] (21) to[in =0, out =0] (24);
\draw[e] (22) to[out = 0, in = 0] (23);

\draw[very thick] (0* \wid,-0.25*\hei) -- (0* \wid,3.25*\hei);
\draw[very thick] (1* \wid,-0.25*\hei) -- (1* \wid,3.25*\hei);
\end{tikzpicture}
$+$
\begin{tikzpicture}[x=1.5cm,y=-.5cm,baseline=-.7cm]
\def\wid{0.8*\standardWidth}
\def\hei{\standardHeight}
\def\nodesize{3}
\def\ang{90}

\node[v, minimum size=\nodesize] (11) at (0* \wid,0*\hei) {};
\node[v, minimum size=\nodesize] (12) at (0* \wid,1*\hei) {};
\node[v, minimum size=\nodesize] (13) at (0* \wid,2*\hei) {};
\node[v, minimum size=\nodesize] (14) at (0* \wid,3*\hei) {};

\draw[e] (11) to[in =180, out =180] (12);
\draw[e] (13) to[in =180, out =180] (14);

\node[v, minimum size=\nodesize] (21) at (1* \wid,0*\hei) {};
\node[v, minimum size=\nodesize] (22) at (1* \wid,1*\hei) {};
\node[v, minimum size=\nodesize] (23) at (1* \wid,2*\hei) {};
\node[v, minimum size=\nodesize] (24) at (1* \wid,3*\hei) {};

\draw[e] (12) to[in =0, out =0] (13);
\draw[e] (11) to[out = 0, in = 180] (23);
\draw[e] (14) to[out = 0, in = 180] (24);
\draw[e] (21) to[out = 180, in = 180] (22);

\draw[e] (21) to[in =0, out =0] (22);
\draw[e] (23) to[out = 0, in = 0] (24);

\draw[very thick] (0* \wid,-0.25*\hei) -- (0* \wid,3.25*\hei);
\draw[very thick] (1* \wid,-0.25*\hei) -- (1* \wid,3.25*\hei);
\end{tikzpicture}
\\
\vspace{0.2cm}
$\mathrm{(b)}$
\ \
\begin{tikzpicture}[x=1.5cm,y=-.5cm,baseline=-.7cm]

\def\wid{\standardWidth}
\def\hei{\standardHeight}
\def\nodesize{3}
\def\ang{90}

\draw[very thick] (1* \wid,-0.25*\hei) -- (1* \wid,3.25*\hei);
\draw[very thick] (2* \wid,-0.25*\hei) -- (2* \wid,3.25*\hei);

\node[v, minimum size=\nodesize] (21) at (1* \wid,0*\hei) {};
\node[v, minimum size=\nodesize] (22) at (1* \wid,1*\hei) {};
\node[v, minimum size=\nodesize] (23) at (1* \wid,2*\hei) {};
\node[v, minimum size=\nodesize] (24) at (1* \wid,3*\hei) {};

\draw[e] (24) to[out = 180, in = 180] (23);
\draw[e] (21) to[out = 180, in = 180] (22);

\node[v, minimum size=\nodesize] (31) at (2* \wid,0*\hei) {};
\node[v, minimum size=\nodesize] (32) at (2* \wid,1*\hei) {};
\node[v, minimum size=\nodesize] (33) at (2* \wid,2*\hei) {};
\node[v, minimum size=\nodesize] (34) at (2* \wid,3*\hei) {};

\draw[e] (21) to[in =0, out =0] (22);
\draw[e] (23) to[out = 0, in = 180] (31);
\draw[e] (24) to[out = 0, in = 180] (32);
\draw[e] (33) to[out = 180, in = 180] (34);

\draw[e] (31) to[in =0, out =0] (34);
\draw[e] (32) to[in =0, out =0] (33);

\end{tikzpicture}
\ \
$\cdot $
\ \
\begin{tikzpicture}[x=1.5cm,y=-.5cm,baseline=-.7cm]

\def\wid{\standardWidth}
\def\hei{\standardHeight}
\def\nodesize{3}
\def\ang{90}

\node[v, minimum size=\nodesize] (41) at (3* \wid,0*\hei) {};
\node[v, minimum size=\nodesize] (42) at (3* \wid,1*\hei) {};
\node[v, minimum size=\nodesize] (43) at (3* \wid,2*\hei) {};
\node[v, minimum size=\nodesize] (44) at (3* \wid,3*\hei) {};

\draw[e] (43) to[in =0, out =0] (44);
\draw[e] (43) to[in =180, out =180] (44);

\draw[e] (41) to[in =180, out =180] (42);

\node[v, minimum size=\nodesize] (51) at (4* \wid,0*\hei) {};
\node[v, minimum size=\nodesize] (52) at (4* \wid,1*\hei) {};
\node[v, minimum size=\nodesize] (53) at (4* \wid,2*\hei) {};
\node[v, minimum size=\nodesize] (54) at (4* \wid,3*\hei) {};

\draw[e] (41) to[in =180, out =0] (51);
\draw[e] (42) to[in =180, out =0] (54);

\draw[e] (52) to[in =0, out =0] (53);
\draw[e] (52) to[in =180, out =180] (53);
\draw[e] (51) to[in =0, out =0] (54);

\draw[very thick] (3* \wid,-0.25*\hei) -- (3* \wid,3.25*\hei);
\draw[very thick] (4* \wid,-0.25*\hei) -- (4* \wid,3.25*\hei);
\end{tikzpicture}
\ \
$=$
\ \
\begin{tikzpicture}[x=1.5cm,y=-.5cm,baseline=-.7cm]

\def\wid{\standardWidth}
\def\hei{\standardHeight}
\def\nodesize{3}
\def\ang{90}

\draw[very thick] (1* \wid,-0.25*\hei) -- (1* \wid,3.25*\hei);
\draw[very thick] (2* \wid,-0.25*\hei) -- (2* \wid,3.25*\hei);

\node[v, minimum size=\nodesize] (21) at (1* \wid,0*\hei) {};
\node[v, minimum size=\nodesize] (22) at (1* \wid,1*\hei) {};
\node[v, minimum size=\nodesize] (23) at (1* \wid,2*\hei) {};
\node[v, minimum size=\nodesize] (24) at (1* \wid,3*\hei) {};

\draw[e] (24) to[out = 180, in = 180] (23);
\draw[e] (21) to[out = 180, in = 180] (22);
\draw[e] (21) to[out = 180, in = 180] (22);

\node[v, minimum size=\nodesize] (31) at (2* \wid,0*\hei) {};
\node[v, minimum size=\nodesize] (32) at (2* \wid,1*\hei) {};
\node[v, minimum size=\nodesize] (33) at (2* \wid,2*\hei) {};
\node[v, minimum size=\nodesize] (34) at (2* \wid,3*\hei) {};

\draw[e] (21) to[in =0, out =0] (22);
\draw[e] (23) to[out = 0, in = 180] (31);
\draw[e] (24) to[out = 0, in = 180] (32);
\draw[e] (33) to[out = 180, in = 180] (34);

\draw[e] (31) to[in =0, out =0] (34);
\draw[e] (32) to[in =0, out =0] (33);

\begin{scope}[shift={(-0.2*\wid,0)}]
\node[v, minimum size=\nodesize] (41) at (3* \wid,0*\hei) {};
\node[v, minimum size=\nodesize] (42) at (3* \wid,1*\hei) {};
\node[v, minimum size=\nodesize] (43) at (3* \wid,2*\hei) {};
\node[v, minimum size=\nodesize] (44) at (3* \wid,3*\hei) {};

\draw[e] (43) to[in =0, out =0] (44);
\draw[e] (43) to[in =180, out =180] (44);

\draw[e] (41) to[in =180, out =180] (42);

\node[v, minimum size=\nodesize] (51) at (4* \wid,0*\hei) {};
\node[v, minimum size=\nodesize] (52) at (4* \wid,1*\hei) {};
\node[v, minimum size=\nodesize] (53) at (4* \wid,2*\hei) {};
\node[v, minimum size=\nodesize] (54) at (4* \wid,3*\hei) {};

\draw[e] (41) to[in =180, out =0] (51);
\draw[e] (42) to[in =180, out =0] (54);

\draw[e] (52) to[in =0, out =0] (53);
\draw[e] (52) to[in =180, out =180] (53);
\draw[e] (51) to[in =0, out =0] (54);
\draw[very thick] (3* \wid,-0.25*\hei) -- (3* \wid,3.25*\hei);
\draw[very thick] (4* \wid,-0.25*\hei) -- (4* \wid,3.25*\hei);
\end{scope}
\end{tikzpicture}
\end{flushleft}

    \caption{$\mathrm{(a)}$ the action of the differential on a system of loops in $\CPL(4)_3$, and $\mathrm{(b)}$ The juxtaposition product.}
    \label{fig:first diagram}
\end{figure}

The differential $d$ is given by $d = \sum_{i=0}^{p-1} (-1)^{p} d_i$, where $d_i$ deletes the $i$-th bar to regard the system of loops as pinned by $p-1$ bars. Any \emph{unpinned} loops (those not intersecting any bar) resulting from this are removed, and for each unpinned loop removed we multiply by $a \in R$. The complex $\CPL(2n)$ enjoys a product, given by juxtaposition of pictures, which gives it the structure of a dga. See \cref{fig:first diagram} (b) for an example of the product $\CPL(4)_2 \times \CPL(4)_2 \to \CPL(4)_4$.

In \cite{BBRWS}, a minimal model for $\CPL(2n)$ is given: there is a weak equivalence
$$(T_R(x_1, x_3, \ldots, x_{2 n-1}), d) \overset{\sim}\to \CPL(2n; R,a)$$
of dg-$R$-algebras, where $x_i$ has degree $i$ and the differential is given by:
$$d(x_1)=a \text{, and } d(x_{2i-1}) = \sum_{\substack{j+k=i \\ j,k>0}} \binom{i}{j} x_{2j-1}x_{2k-1} \text{ for } 2i-1>1.$$

This map must carry each generator $x_{2i-1}$ to an $R$-linear combination of diagrams. It is shown in \cite{BBRWS} that $x_1$ corresponds to any diagram on a single bar consisting of only one loop (any two of these are homologous), but the other generators seem quite mysterious. The dga $\CPL(2n)$ carries two natural actions of the cyclic group of order 2: an involution, $\sigma_{\updownarrow}$, given by top-to-bottom reflection of diagrams and an anti-involution, $\sigma_{\leftrightarrow}$, given by left-to-right reflection of diagrams. The effect of these involutions (beyond that they fix the homology class of $x_1$) is not clear. They respect the differential in the respective senses that $d \sigma_{\updownarrow} = \sigma_{\updownarrow} d$, and $d \sigma_{\leftrightarrow} (x) = (-1)^{\deg(x)+1} \sigma_{\leftrightarrow} d(x)$.

The effect of these actions on homology is one of several open questions about the dga of planar loops \cite{BBRWS}. In this paper we give a model for $\CPL(4)$ which incorporates the involutions, and explain how it relates to the existing model. We will give the data $(A,d,\sigma_{\updownarrow},\sigma_{\leftrightarrow})$ of a dga $(A,d)$, together with an involution $\sigma_{\updownarrow}$ and an anti-involution $\sigma_{\leftrightarrow}$, which respect the differential as above. Said this way, our model is:
$$(T_R[x, \hat{x}, r, y], d, \sigma_{\updownarrow}, \sigma_{\leftrightarrow}),$$
where the generators have degrees $\deg(x) = \deg(\hat{x})=1$, $\deg(r)=2$, $\deg(y)=3$, the differential is given by $$d(x)=d(\widehat{x}) = a \text{, } d(r)=\hat{x}-x \text{, and } d(y) = 2 x \hat{x} - 2 a r,$$
the involution $\sigma_{\updownarrow}$ acts as the identity, and the anti-involution $\sigma_{\leftrightarrow}$ interchanges $x$ and $\hat{x}$ and acts by the identity on the other generators.

\begin{theorem} \label{thm: main}
    There is a weak equivalence of dgas
    $$\varphi \colon (T_R[x, \hat{x}, r, y], d) \overset{\sim} \to \CPL(4; R,a),$$
    commuting with both involutions, which is given on generators by
\begin{center}
$\varphi(x)=$
\begin{tikzpicture}[x=1.5cm,y=-.5cm,baseline=-0.7cm]
\def\wid{\standardWidth}
\def\hei{\standardHeight}
\def\nodesize{3}
\def\ang{90}

\node[v, minimum size=\nodesize] (11) at (0* \wid,0*\hei) {};
\node[v, minimum size=\nodesize] (12) at (0* \wid,1*\hei) {};
\node[v, minimum size=\nodesize] (13) at (0* \wid,2*\hei) {};
\node[v, minimum size=\nodesize] (14) at (0* \wid,3*\hei) {};

\draw[e] (11) to[in =180, out =180] (12);
\draw[e] (13) to[in =180, out =180] (14);
\draw[e] (12) to[in =0, out =0] (13);
\draw[e] (11) to[out = 0, in = 0] (14);

\draw[very thick] (0* \wid,-0.25*\hei) -- (0* \wid,3.25*\hei);
\end{tikzpicture}
,
$\varphi(\hat{x})=$
\begin{tikzpicture}[x=1.5cm,y=-.5cm,baseline=-0.7cm]
\def\wid{\standardWidth}
\def\hei{\standardHeight}
\def\nodesize{3}
\def\ang{90}

\node[v, minimum size=\nodesize] (11) at (0* \wid,0*\hei) {};
\node[v, minimum size=\nodesize] (12) at (0* \wid,1*\hei) {};
\node[v, minimum size=\nodesize] (13) at (0* \wid,2*\hei) {};
\node[v, minimum size=\nodesize] (14) at (0* \wid,3*\hei) {};

\draw[e] (11) to[in =0, out =0] (12);
\draw[e] (13) to[in =0, out =0] (14);
\draw[e] (12) to[in =180, out =180] (13);
\draw[e] (11) to[out = 180, in = 180] (14);

\draw[very thick] (0* \wid,-0.25*\hei) -- (0* \wid,3.25*\hei);
\end{tikzpicture}
, 
$\varphi(r)=$
\begin{tikzpicture}[x=1.5cm,y=-.5cm,baseline=-.7cm]

\def\wid{\standardWidth}
\def\hei{\standardHeight}
\def\nodesize{3}
\def\ang{90}

\draw[very thick] (1* \wid,-0.25*\hei) -- (1* \wid,3.25*\hei);
\draw[very thick] (2* \wid,-0.25*\hei) -- (2* \wid,3.25*\hei);

\node[v, minimum size=\nodesize] (21) at (1* \wid,0*\hei) {};
\node[v, minimum size=\nodesize] (22) at (1* \wid,1*\hei) {};
\node[v, minimum size=\nodesize] (23) at (1* \wid,2*\hei) {};
\node[v, minimum size=\nodesize] (24) at (1* \wid,3*\hei) {};

\draw[e] (24) to[out = 180, in = 180] (23);
\draw[e] (21) to[out = 180, in = 180] (22);

\node[v, minimum size=\nodesize] (31) at (2* \wid,0*\hei) {};
\node[v, minimum size=\nodesize] (32) at (2* \wid,1*\hei) {};
\node[v, minimum size=\nodesize] (33) at (2* \wid,2*\hei) {};
\node[v, minimum size=\nodesize] (34) at (2* \wid,3*\hei) {};

\draw[e] (21) to[out =0, in =180] (31);
\draw[e] (24) to[out = 0, in = 180] (34);
\draw[e] (22) to[out = 0, in = 0] (23);
\draw[e] (32) to[out = 180, in = 180] (33);

\draw[e] (31) to[in =0, out =0] (32);
\draw[e] (34) to[in =0, out =0] (33);

\end{tikzpicture}
$, \textrm{ and }$
\end{center}

\begin{center}
$\varphi(y)=$
    \begin{tikzpicture}[x=1.5cm,y=-.5cm,baseline=-0.7cm]
        \def\wid{\standardWidth}
            \def\hei{\standardHeight}
            \def\nodesize{3}

            \node[v, minimum size=\nodesize] (01) at (-1* \wid,0*\hei) {};
            \node[v, minimum size=\nodesize] (02) at (-1* \wid,1*\hei) {};
            \node[v, minimum size=\nodesize] (03) at (-1* \wid,2*\hei) {};
            \node[v, minimum size=\nodesize] (04) at (-1* \wid,3*\hei) {};

            \node[v, minimum size=\nodesize] (11) at (0* \wid,0*\hei) {};
            \node[v, minimum size=\nodesize] (12) at (0* \wid,1*\hei) {};
            \node[v, minimum size=\nodesize] (13) at (0* \wid,2*\hei) {};
            \node[v, minimum size=\nodesize] (14) at (0* \wid,3*\hei) {};

            \node[v, minimum size=\nodesize] (21) at (1* \wid,0*\hei) {};
            \node[v, minimum size=\nodesize] (22) at (1* \wid,1*\hei) {};
            \node[v, minimum size=\nodesize] (23) at (1* \wid,2*\hei) {};
            \node[v, minimum size=\nodesize] (24) at (1* \wid,3*\hei) {};

            \draw[e] (02) to[in =0, out =0] (03);
            \draw[e] (01) to[in =180, out =180] (02);
            \draw[e] (03) to[in =180, out =180] (04);
            \draw[e] (11) to[in =180, out =180] (12);
            \draw[e] (13) to[out = 180, in = 0] (01);
            \draw[e] (14) to[out = 180, in = 0] (04);

            \draw[e] (11) to[in =0, out =0] (12);
            \draw[e] (23) to[in =180, out =180] (22);
            \draw[e] (21) to[in =0, out =0] (22);
            \draw[e] (23) to[in =0, out =0] (24);
            \draw[e] (13) to[out = 0, in = 180] (21);
            \draw[e] (14) to[out = 0, in = 180] (24);

            \draw[very thick] (0* \wid,-0.25*\hei) -- (0* \wid,3.25*\hei);
            \draw[very thick] (-1* \wid,-0.25*\hei) -- (-1* \wid,3.25*\hei);
            \draw[very thick] (1* \wid,-0.25*\hei) -- (1* \wid,3.25*\hei);
    \end{tikzpicture}
    $+$ 
    \begin{tikzpicture}[x=1.5cm,y=-.5cm,baseline=-0.7cm]
       \def\wid{\standardWidth}
            \def\hei{\standardHeight}
            \def\nodesize{3}

            \node[v, minimum size=\nodesize] (01) at (-1* \wid,0*\hei) {};
            \node[v, minimum size=\nodesize] (02) at (-1* \wid,1*\hei) {};
            \node[v, minimum size=\nodesize] (03) at (-1* \wid,2*\hei) {};
            \node[v, minimum size=\nodesize] (04) at (-1* \wid,3*\hei) {};

            \node[v, minimum size=\nodesize] (11) at (0* \wid,0*\hei) {};
            \node[v, minimum size=\nodesize] (12) at (0* \wid,1*\hei) {};
            \node[v, minimum size=\nodesize] (13) at (0* \wid,2*\hei) {};
            \node[v, minimum size=\nodesize] (14) at (0* \wid,3*\hei) {};

            \node[v, minimum size=\nodesize] (21) at (1* \wid,0*\hei) {};
            \node[v, minimum size=\nodesize] (22) at (1* \wid,1*\hei) {};
            \node[v, minimum size=\nodesize] (23) at (1* \wid,2*\hei) {};
            \node[v, minimum size=\nodesize] (24) at (1* \wid,3*\hei) {};

            \draw[e] (02) to[in =0, out =0] (03);
            \draw[e] (01) to[in =180, out =180] (02);
            \draw[e] (03) to[in =180, out =180] (04);
            \draw[e] (13) to[in =180, out =180] (14);
            \draw[e] (11) to[out = 180, in = 0] (01);
            \draw[e] (12) to[out = 180, in = 0] (04);

            \draw[e] (13) to[in =0, out =0] (14);
            \draw[e] (23) to[in =180, out =180] (22);
            \draw[e] (21) to[in =0, out =0] (22);
            \draw[e] (23) to[in =0, out =0] (24);
            \draw[e] (11) to[out = 0, in = 180] (21);
            \draw[e] (12) to[out = 0, in = 180] (24);

            \draw[very thick] (0* \wid,-0.25*\hei) -- (0* \wid,3.25*\hei);
            \draw[very thick] (-1* \wid,-0.25*\hei) -- (-1* \wid,3.25*\hei);
            \draw[very thick] (1* \wid,-0.25*\hei) -- (1* \wid,3.25*\hei); 
    \end{tikzpicture}
    \\
    $-$
    \begin{tikzpicture}[x=1.5cm,y=-.5cm,baseline=-0.7cm]
       \def\wid{\standardWidth}
            \def\hei{\standardHeight}
            \def\nodesize{3}

            \node[v, minimum size=\nodesize] (01) at (-1* \wid,0*\hei) {};
            \node[v, minimum size=\nodesize] (02) at (-1* \wid,1*\hei) {};
            \node[v, minimum size=\nodesize] (03) at (-1* \wid,2*\hei) {};
            \node[v, minimum size=\nodesize] (04) at (-1* \wid,3*\hei) {};

            \node[v, minimum size=\nodesize] (11) at (0* \wid,0*\hei) {};
            \node[v, minimum size=\nodesize] (12) at (0* \wid,1*\hei) {};
            \node[v, minimum size=\nodesize] (13) at (0* \wid,2*\hei) {};
            \node[v, minimum size=\nodesize] (14) at (0* \wid,3*\hei) {};

            \node[v, minimum size=\nodesize] (21) at (1* \wid,0*\hei) {};
            \node[v, minimum size=\nodesize] (22) at (1* \wid,1*\hei) {};
            \node[v, minimum size=\nodesize] (23) at (1* \wid,2*\hei) {};
            \node[v, minimum size=\nodesize] (24) at (1* \wid,3*\hei) {};

            \draw[e] (02) to[in =0, out =0] (03);
            \draw[e] (01) to[in =180, out =180] (02);
            \draw[e] (03) to[in =180, out =180] (04);
            \draw[e] (11) to[in =180, out =180] (12);
            \draw[e] (13) to[out = 180, in = 0] (01);
            \draw[e] (14) to[out = 180, in = 0] (04);

            \draw[e] (13) to[in =0, out =0] (14);
            \draw[e] (23) to[in =180, out =180] (22);
            \draw[e] (21) to[in =0, out =0] (22);
            \draw[e] (23) to[in =0, out =0] (24);
            \draw[e] (11) to[out = 0, in = 180] (21);
            \draw[e] (12) to[out = 0, in = 180] (24);

            \draw[very thick] (0* \wid,-0.25*\hei) -- (0* \wid,3.25*\hei);
            \draw[very thick] (-1* \wid,-0.25*\hei) -- (-1* \wid,3.25*\hei);
            \draw[very thick] (1* \wid,-0.25*\hei) -- (1* \wid,3.25*\hei); 
    \end{tikzpicture}
    $-$
    \begin{tikzpicture}[x=1.5cm,y=-.5cm,baseline=-0.7cm]
       \def\wid{\standardWidth}
            \def\hei{\standardHeight}
            \def\nodesize{3}

            \node[v, minimum size=\nodesize] (01) at (-1* \wid,0*\hei) {};
            \node[v, minimum size=\nodesize] (02) at (-1* \wid,1*\hei) {};
            \node[v, minimum size=\nodesize] (03) at (-1* \wid,2*\hei) {};
            \node[v, minimum size=\nodesize] (04) at (-1* \wid,3*\hei) {};

            \node[v, minimum size=\nodesize] (11) at (0* \wid,0*\hei) {};
            \node[v, minimum size=\nodesize] (12) at (0* \wid,1*\hei) {};
            \node[v, minimum size=\nodesize] (13) at (0* \wid,2*\hei) {};
            \node[v, minimum size=\nodesize] (14) at (0* \wid,3*\hei) {};

            \node[v, minimum size=\nodesize] (21) at (1* \wid,0*\hei) {};
            \node[v, minimum size=\nodesize] (22) at (1* \wid,1*\hei) {};
            \node[v, minimum size=\nodesize] (23) at (1* \wid,2*\hei) {};
            \node[v, minimum size=\nodesize] (24) at (1* \wid,3*\hei) {};

            \draw[e] (02) to[in =0, out =0] (03);
            \draw[e] (01) to[in =180, out =180] (02);
            \draw[e] (03) to[in =180, out =180] (04);
            \draw[e] (13) to[in =180, out =180] (14);
            \draw[e] (11) to[out = 180, in = 0] (01);
            \draw[e] (12) to[out = 180, in = 0] (04);

            \draw[e] (11) to[in =0, out =0] (12);
            \draw[e] (23) to[in =180, out =180] (22);
            \draw[e] (21) to[in =0, out =0] (22);
            \draw[e] (23) to[in =0, out =0] (24);
            \draw[e] (13) to[out = 0, in = 180] (21);
            \draw[e] (14) to[out = 0, in = 180] (24);

            \draw[very thick] (0* \wid,-0.25*\hei) -- (0* \wid,3.25*\hei);
            \draw[very thick] (-1* \wid,-0.25*\hei) -- (-1* \wid,3.25*\hei);
            \draw[very thick] (1* \wid,-0.25*\hei) -- (1* \wid,3.25*\hei); 
    \end{tikzpicture}
    .
    \end{center}
\end{theorem}

We also provide the following comparison with the existing model.

\begin{proposition} \label{prop: model comparison} The map of dgas $\psi \colon (T_R[x_1,x_3],d) \to (T_R[x,\hat{x},r,y],d)$ given by $x_1 \mapsto x$, $x_3 \mapsto y + 2 x r$ is a weak equivalence. \end{proposition}

It follows that the composite $\varphi \circ \psi$ is (one possible choice of) the model established in \cite{BBRWS}, so diagram generators can be read off. The homology of $(T_R[x_1,x_3],d)$ was computed there for certain $(R,a)$. For example, for $(\mathbb{F}_2,0)$ the differential is trivial, so the homology is just $T_{\mathbb{F}_2}[x_1,x_3]$, but for $(\mathbb{Q},0)$ the homology is $\mathbb{Q}[\alpha] \otimes \Lambda[x_1]$, where $\alpha=x_1 x_3 + x_3 x_1$ \cite[Corollary 1.2(c)]{BBRWS}.

The involution $\sigma_{\updownarrow}$ acts trivially on our model, hence also on homology with any coefficients. Over $\mathbb{F}_2$, the anti-involution $\sigma_{\leftrightarrow}$ fixes the homology generators $x_1$ and $x_3$, but acts nontrivially because it reverses word order. Over $\mathbb{Q}$, direct calculation shows that $\sigma_{\leftrightarrow}\psi(\alpha)$ and $\psi(\alpha)$ differ by the boundary $d(ry-yr + 2r^2 \hat{x} - 2x r^2)$, so the homology action is trivial (the algebra being commutative in this case).

\section*{Acknowledgements}

Thanks to Richard Hepworth-Young for productive early discussions, and to the anonymous referee for many thoughtful comments. This work was supported by the ERC (grant no. 950048), and by the NWO (VI.Veni.242.230).

\section{Temperley--Lieb diagrams, cell modules, ideals} \label{section:basicDefinitions}

Fix a commutative ring $R$ and an element $a \in R$. The \emph{Temperley--Lieb category} $\TL = \TL(R,a)$ is the $R$-linear category where the objects are non-negative integers $n$ and the morphism modules $\TL(n,m)$ are free $R$-modules on \emph{Temperley--Lieb $(n,m)$ diagrams}. Place $n$ vertices on the left edge of the unit square and $m$ on the right. A \emph{Temperley--Lieb $(n,m)$ diagram} consists of $\frac{n+m}{2}$ nonintersecting arcs which connect the vertices in pairs. The arcs are not oriented, and are considered up to isotopy, so two diagrams are regarded as the same if they give the same pairing of the vertices.

Composition of morphisms is given by concatenation, replacing loops by factors of $a$. An example composition $\TL(4,2) \otimes_R \TL(2,0) \to \TL(4,0)$ is shown below.

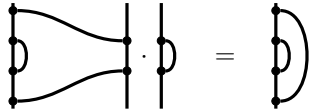
\begin{center}
\begin{tikzpicture}[x=1.5cm,y=-.5cm,baseline=-.7cm]

\def\wid{\standardWidth}
\def\hei{\standardHeight}
\def\nodesize{3}

\node[v, minimum size=\nodesize] (11) at (0* \wid,0*\hei) {};
\node[v, minimum size=\nodesize] (12) at (0* \wid,1*\hei) {};
\node[v, minimum size=\nodesize] (13) at (0* \wid,2*\hei) {};
\node[v, minimum size=\nodesize] (14) at (0* \wid,3*\hei) {};

\node[v, minimum size=\nodesize] (22) at (1* \wid,1*\hei) {};
\node[v, minimum size=\nodesize] (23) at (1* \wid,2*\hei) {};

\draw[e] (12) to[in =0, out =0] (13);
\draw[e] (11) to[out = 0, in = 180] (22);
\draw[e] (14) to[out = 0, in = 180] (23);

\draw[very thick] (0* \wid,-0.25*\hei) -- (0* \wid,3.25*\hei);
\draw[very thick] (1* \wid,-0.25*\hei) -- (1* \wid,3.25*\hei);
\end{tikzpicture}
$\cdot$
\begin{tikzpicture}[x=1.5cm,y=-.5cm,baseline=-.7cm]

\def\wid{\standardWidth}
\def\hei{\standardHeight}
\def\nodesize{3}

\node[v, minimum size=\nodesize] (12) at (0* \wid,1*\hei) {};
\node[v, minimum size=\nodesize] (13) at (0* \wid,2*\hei) {};

\draw[e] (12) to[in =0, out =0] (13);

\draw[very thick] (0* \wid,-0.25*\hei) -- (0* \wid,3.25*\hei);
\end{tikzpicture}
\quad
$=$
\quad
\begin{tikzpicture}[x=1.5cm,y=-.5cm,baseline=-.7cm]

\def\wid{\standardWidth}
\def\hei{\standardHeight}
\def\nodesize{3}

\node[v, minimum size=\nodesize] (11) at (0* \wid,0*\hei) {};
\node[v, minimum size=\nodesize] (12) at (0* \wid,1*\hei) {};
\node[v, minimum size=\nodesize] (13) at (0* \wid,2*\hei) {};
\node[v, minimum size=\nodesize] (14) at (0* \wid,3*\hei) {};

\draw[e] (12) to[in =0, out =0] (13);
\draw[e] (11) to[out = 0, in = 0] (14);
\draw[very thick] (0* \wid,-0.25*\hei) -- (0* \wid,3.25*\hei);
\end{tikzpicture}
\end{center}

The \emph{Temperley--Lieb algebra} $\TL_n = \TL_n(R,a)$ is the endomorphism algebra $\mathrm{End}_{\TL}(n)$. By construction, $\TL(n,m)$ is a $\TL_n$-$\TL_m$ bimodule.

\subsection{Cell modules} The \emph{cell module} $S(n,k)$ is the left $\TL_n$-module obtained as the quotient of $\TL(n,k)$ by all diagrams having a right-to-right connection. In other words, $S(n,k)$ has a basis consisting of all $(n,k)$-diagrams where every vertex on the right is attached to one on the left. In the case $k = 0$, this condition is vacuous, so we get a canonical isomorphism $S(n,0) \cong \TL(n,0)$.

We also define the \emph{opposite cell module} $S^\vee(k,n)$ to be the \emph{right} $\TL_n$-module obtained as the quotient of $\TL(k,n)$ by all diagrams having a left-to-left connection.

\subsection{Ideals} Fix $n$, and write $I_{\leq k}$ for the two-sided ideal of $\TL_n$ consisting of diagrams with at most $k$ left-to-right connections. Write $I_k$ for the quotient module $\faktor{I_{\leq k}}{I_{\leq k-1}}$, which is spanned by diagrams with precisely $k$ left-to-right connections.

Slicing down the middle of a diagram in $I_k$ gives a `left half' in $S(n,k)$, called the \emph{left link state}, and a `right half' in $S^\vee(k,n)$, called the \emph{right link state}. Graphically, this looks as follows:
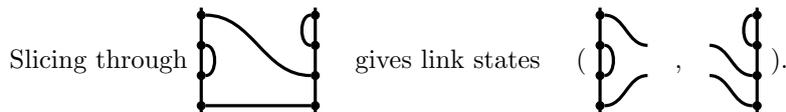
\begin{center}
Slicing through
\begin{tikzpicture}[x=1.5cm,y=-.5cm,baseline=-0.7cm]
\def\wid{\standardWidth}
\def\hei{\standardHeight}
\def\nodesize{3}
\def\ang{90}

\node[v, minimum size=\nodesize] (11) at (0* \wid,0*\hei) {};
\node[v, minimum size=\nodesize] (12) at (0* \wid,1*\hei) {};
\node[v, minimum size=\nodesize] (13) at (0* \wid,2*\hei) {};
\node[v, minimum size=\nodesize] (14) at (0* \wid,3*\hei) {};

\node[v, minimum size=\nodesize] (21) at (1* \wid,0*\hei) {};
\node[v, minimum size=\nodesize] (22) at (1* \wid,1*\hei) {};
\node[v, minimum size=\nodesize] (23) at (1* \wid,2*\hei) {};
\node[v, minimum size=\nodesize] (24) at (1* \wid,3*\hei) {};

\draw[e] (12) to[in =0, out =0] (13);
\draw[e] (11) to[out = 0, in = 180] (23);
\draw[e] (14) to[out = 0, in = 180] (24);
\draw[e] (21) to[out = 180, in = 180] (22);

\draw[very thick] (0* \wid,-0.25*\hei) -- (0* \wid,3.25*\hei);
\draw[very thick] (1* \wid,-0.25*\hei) -- (1* \wid,3.25*\hei);
\end{tikzpicture}
\quad
gives link states
\quad
\begin{tikzpicture}[x=1.5cm,y=-.5cm,baseline=-0.7cm]
\def\wid{\standardWidth}
\def\hei{\standardHeight}
\def\nodesize{3}
\def\ang{90}

\node[v, minimum size=\nodesize] (11) at (0* \wid,0*\hei) {};
\node[v, minimum size=\nodesize] (12) at (0* \wid,1*\hei) {};
\node[v, minimum size=\nodesize] (13) at (0* \wid,2*\hei) {};
\node[v, minimum size=\nodesize] (14) at (0* \wid,3*\hei) {};

\node[] (1h1) at (0.5* \wid,1*\hei) {};
\node[] (1h2) at (0.5* \wid,2*\hei) {};

\draw[e] (12) to[in =0, out =0] (13);
\draw[e] (11) to[out = 0, in = 180] (1h1);
\draw[e] (14) to[out = 0, in = 180] (1h2);

\draw[very thick] (0* \wid,-0.25*\hei) -- (0* \wid,3.25*\hei);
\end{tikzpicture}
 and 
\begin{tikzpicture}[x=1.5cm,y=-.5cm,baseline=-0.7cm]
\def\wid{\standardWidth}
\def\hei{\standardHeight}
\def\nodesize{3}
\def\ang{90}

\node[] (1h1) at (0.5* \wid,1*\hei) {};
\node[] (1h2) at (0.5* \wid,2*\hei) {};

\node[v, minimum size=\nodesize] (21) at (1* \wid,0*\hei) {};
\node[v, minimum size=\nodesize] (22) at (1* \wid,1*\hei) {};
\node[v, minimum size=\nodesize] (23) at (1* \wid,2*\hei) {};
\node[v, minimum size=\nodesize] (24) at (1* \wid,3*\hei) {};

\draw[e] (1h1) to[out = 0, in = 180] (23);
\draw[e] (1h2) to[out = 0, in = 180] (24);
\draw[e] (21) to[out = 180, in = 180] (22);

\draw[very thick] (1* \wid,-0.25*\hei) -- (1* \wid,3.25*\hei);
\end{tikzpicture}
.
\end{center}
We can reconstruct the original diagram from these two halves (since there is only one planar way to join the loose ends up) so we have the following.

\begin{lemma} \label{lem:RtensorIso}
    Slicing diagrams into their left and right link states gives an isomorphism of $\TL_n$-bimodules $I_k \cong S(n,k) \otimes_R S^{\vee}(k,n)$. \qed
\end{lemma}

We will use this lemma primarily in the case $k=0$. It is really an instance of Graham and Lehrer's \emph{cellular} philosophy \cite{GrahamLehrer}, but we will not think explicitly in these terms. For more detail, see their paper, as well as \cite{RidoutSaintAubin}.

Using \cref{lem:RtensorIso}, given a diagram $\beta \in I_k$, we may write $$\beta = \beta^L \otimes \beta^R,$$ decomposing $\beta$ into its link states. This isomorphism will be central in what follows: we will usually use it graphically, but the reader may sometimes wish to use \cref{lem:RtensorIso} to spell something out in more algebraic terms.

\section{The reduced complex of loops} \label{sectionBar}

As before, let $R$ be commutative, and let $a \in R$. Let $2n \geq 4$ be even. Pictures will be drawn in the case $2n=4$. Unadorned tensor products are over $R$. Note that there is exactly one Temperley--Lieb $(0,0)$-diagram, namely the empty diagram $\emptyset$.

By definition \cite[Definition 2.4]{BBRWS}, the complex of planar loops is obtained from the shifted bar construction $\mathrm{Bar}_{*-1}(\TL(0,2n), \TL_{2n}, \TL(2n,0))$ by adding an augmentation given by the composition map $\TL(0,2n) \otimes \TL(2n,0) \xrightarrow{c} \TL(0,0)  = R\{ \emptyset \}$. In this paper we will use the reduced version of this bar construction, without augmentation. Recall that $\TL(2n,0) \cong S(2n,0)$ and $\TL(0,2n) \cong S^\vee(0,2n)$.

\begin{definition} The \emph{reduced complex of planar loops} $\CPLr(2n)=\CPLr(2n;R,a)$ is the reduced bar construction $\overline{\mathrm{Bar}}_{*-1}(S^{\vee}(0,2n),\TL_{2n},S(2n,0))$. Explicitly, in degree $p \geq 1$, the underlying $R$-module is:
$$\CPLr(2n)_p = S^{\vee}(0,2n) \otimes \overline{\TL}_{2n}^{\otimes p-1} \otimes S(2n,0),$$
where $\overline{\TL}_{2n} = I_{\leq 2n-1} \subset \TL_{2n}$ has $R$-basis all diagrams other than the identity. The module $\CPLr(2n)_p$ therefore has an $R$-basis consisting of elements of the form
$$\beta_{0} \lvert \beta_1 \lvert \dots \lvert \beta_{p-1} \lvert \beta_{p} : = \beta_0 \otimes \beta_1 \otimes \dots \otimes \beta_{p-1} \otimes \beta_p,$$
where $\beta_{0} \in S^\vee(0,2n)$, $\beta_1, \dots ,\beta_{p-1} \in \overline{\TL}_{2n}$, and $\beta_{p} \in S(2n,0)$. The differential $\CPLr(2n)_p \xrightarrow{d} \CPLr(2n)_{p-1}$ is the alternating sum $d(x) = \sum_{i=1}^p (-1)^i d_i(x)$ of face maps $d_i \colon \CPLr(2n)_p \to \CPLr(2n)_{p-1},$ with
$$d_i(\beta_0 \lvert \beta_1 \lvert \dots \lvert \beta_{p-1} \lvert \beta_p) = 
    \beta_0 \lvert \beta_1 \lvert \dots \lvert \beta_{i-1} \beta_{i} \lvert \dots \lvert \beta_{p-1} \lvert \beta_p .$$ 
\end{definition}

We will review the multiplicative structure on this chain complex in \cref{section: product and grading}. The Temperley--Lieb algebra and the dga of planar loops both depend on a chosen ground ring $R$ and an element $a \in R$: in other words, on a pointed ring $(R,a)$. In fact, to prove our result, it will suffice to consider the case $(\mathbb{Z},0)$. We will make this precise at the start of \cref{section:basechange}, in particular in \cref{lemma: reduction techniques}. With this in mind, in \cref{section: product and grading,section: spectral sequence,section:filtrationByDividers,section: mainTechnical} it will be a standing assumption that $a=0$.

\begin{remark} \label{rmk: reduced is same} When $a=0$, the augmentation $c$ of $\CPL(2n)$ is zero, so $\CPLr(2n)$ is a subcomplex of $\CPL(2n)$. The inclusion of the reduced bar construction in the full bar construction is a weak equivalence, so in this case $H_*(\CPL(2n)) \cong H_*(\CPLr(2n)) \oplus R\{\emptyset\}$. We will take advantage of this to instead study the smaller complex $\CPLr(2n)$. \end{remark}

Graphically, we will think of the face maps $d_i$ as \emph{bar deletions}. For the diagram from \cref{fig:first diagram} (a), the action of the face map $d_0$ looks as follows:
\begin{center}
$d_0$ :
\begin{tikzpicture}[x=1.5cm,y=-.5cm,baseline=-.7cm]
\def\wid{\standardWidth}
\def\hei{\standardHeight}
\def\nodesize{3}
\def\ang{90}

\node[v, minimum size=\nodesize] (11) at (0* \wid,0*\hei) {};
\node[v, minimum size=\nodesize] (12) at (0* \wid,1*\hei) {};
\node[v, minimum size=\nodesize] (13) at (0* \wid,2*\hei) {};
\node[v, minimum size=\nodesize] (14) at (0* \wid,3*\hei) {};

\draw[e] (11) to[in =180, out =180] (12);
\draw[e] (13) to[in =180, out =180] (14);

\node[v, minimum size=\nodesize] (21) at (1* \wid,0*\hei) {};
\node[v, minimum size=\nodesize] (22) at (1* \wid,1*\hei) {};
\node[v, minimum size=\nodesize] (23) at (1* \wid,2*\hei) {};
\node[v, minimum size=\nodesize] (24) at (1* \wid,3*\hei) {};

\draw[e] (12) to[in =0, out =0] (13);
\draw[e] (11) to[out = 0, in = 180] (23);
\draw[e] (14) to[out = 0, in = 180] (24);
\draw[e] (21) to[out = 180, in = 180] (22);

\node[v, minimum size=\nodesize] (31) at (2* \wid,0*\hei) {};
\node[v, minimum size=\nodesize] (32) at (2* \wid,1*\hei) {};
\node[v, minimum size=\nodesize] (33) at (2* \wid,2*\hei) {};
\node[v, minimum size=\nodesize] (34) at (2* \wid,3*\hei) {};

\draw[e] (21) to[in =0, out =0] (22);
\draw[e] (23) to[out = 0, in = 180] (31);
\draw[e] (24) to[out = 0, in = 180] (32);
\draw[e] (33) to[out = 180, in = 180] (34);

\draw[e] (31) to[in =0, out =0] (34);
\draw[e] (32) to[in =0, out =0] (33);

\draw[very thick] (0* \wid,-0.25*\hei) -- (0* \wid,3.25*\hei);
\draw[very thick] (1* \wid,-0.25*\hei) -- (1* \wid,3.25*\hei);
\draw[very thick] (2* \wid,-0.25*\hei) -- (2* \wid,3.25*\hei);
\end{tikzpicture}
\quad
$\mapsto$
\quad
\begin{tikzpicture}[x=1.5cm,y=-.5cm,baseline=-.7cm]

\def\wid{\standardWidth}
\def\hei{\standardHeight}
\def\nodesize{3}
\def\ang{90}

\node (11) at (0* \wid,0*\hei) {};
\node (12) at (0* \wid,1*\hei) {};
\node (13) at (0* \wid,2*\hei) {};
\node (14) at (0* \wid,3*\hei) {};

\draw[e] (11.center) to[in =180, out =180] (12.center);
\draw[e] (13.center) to[in =180, out =180] (14.center);

\node[v, minimum size=\nodesize] (21) at (1* \wid,0*\hei) {};
\node[v, minimum size=\nodesize] (22) at (1* \wid,1*\hei) {};
\node[v, minimum size=\nodesize] (23) at (1* \wid,2*\hei) {};
\node[v, minimum size=\nodesize] (24) at (1* \wid,3*\hei) {};

\draw[e] (12.center) to[in =0, out =0] (13.center);
\draw[e] (11.center) to[out = 0, in = 180] (23);
\draw[e] (14.center) to[out = 0, in = 180] (24);
\draw[e] (21) to[out = 180, in = 180] (22);

\node[v, minimum size=\nodesize] (31) at (2* \wid,0*\hei) {};
\node[v, minimum size=\nodesize] (32) at (2* \wid,1*\hei) {};
\node[v, minimum size=\nodesize] (33) at (2* \wid,2*\hei) {};
\node[v, minimum size=\nodesize] (34) at (2* \wid,3*\hei) {};

\draw[e] (21) to[in =0, out =0] (22);
\draw[e] (23) to[out = 0, in = 180] (31);
\draw[e] (24) to[out = 0, in = 180] (32);
\draw[e] (33) to[out = 180, in = 180] (34);

\draw[e] (31) to[in =0, out =0] (34);
\draw[e] (32) to[in =0, out =0] (33);

\draw[very thick] (1* \wid,-0.25*\hei) -- (1* \wid,3.25*\hei);
\draw[very thick] (2* \wid,-0.25*\hei) -- (2* \wid,3.25*\hei);
\end{tikzpicture}
\quad \\
=
\begin{tikzpicture}[x=1.5cm,y=-.5cm,baseline=-.7cm]
\def\wid{\standardWidth}
\def\hei{\standardHeight}
\def\nodesize{3}
\def\ang{90}

\node[v, minimum size=\nodesize] (21) at (1* \wid,0*\hei) {};
\node[v, minimum size=\nodesize] (22) at (1* \wid,1*\hei) {};
\node[v, minimum size=\nodesize] (23) at (1* \wid,2*\hei) {};
\node[v, minimum size=\nodesize] (24) at (1* \wid,3*\hei) {};

\draw[e] (24) to[out = 180, in = 180] (23);
\draw[e] (21) to[out = 180, in = 180] (22);

\node[v, minimum size=\nodesize] (31) at (2* \wid,0*\hei) {};
\node[v, minimum size=\nodesize] (32) at (2* \wid,1*\hei) {};
\node[v, minimum size=\nodesize] (33) at (2* \wid,2*\hei) {};
\node[v, minimum size=\nodesize] (34) at (2* \wid,3*\hei) {};

\draw[e] (21) to[in =0, out =0] (22);
\draw[e] (23) to[out = 0, in = 180] (31);
\draw[e] (24) to[out = 0, in = 180] (32);
\draw[e] (33) to[out = 180, in = 180] (34);

\draw[e] (31) to[in =0, out =0] (34);
\draw[e] (32) to[in =0, out =0] (33);

\draw[very thick] (1* \wid,-0.25*\hei) -- (1* \wid,3.25*\hei);
\draw[very thick] (2* \wid,-0.25*\hei) -- (2* \wid,3.25*\hei);
\end{tikzpicture}
,
\end{center}
and the effect of the total boundary map $d$ on this element is shown in \cref{fig:first diagram}.

We will talk a lot about the basis elements of $\CPLr(2n)$. In the introduction, we called these `systems of planar loops (of height $2n$) pinned by $p$ bars'; it will be useful to have a shorter name.

\begin{definition} \label{def: graffiti}
We will call a basis element in $\CPLr(2n)$ a \emph{(closed) graffito}. In the case $2n=4$, we will say that a \emph{left-open graffito} is a diagram which has two hanging edges at the left end, such as:
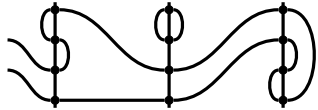
\begin{center}
\begin{tikzpicture}[x=1.5cm,y=-.5cm,baseline=-1.05cm]

\def\wid{\standardWidth}
\def\hei{\standardHeight}
\def\nodesize{3}
\def\ang{90}

\node[v, minimum size=\nodesize] (11) at (0* \wid,0*\hei) {};
\node[v, minimum size=\nodesize] (12) at (0* \wid,1*\hei) {};
\node[v, minimum size=\nodesize] (13) at (0* \wid,2*\hei) {};
\node[v, minimum size=\nodesize] (14) at (0* \wid,3*\hei) {};

\node (h1) at (-0.5* \wid,1*\hei) {};
\node (h2) at (-0.5* \wid,2*\hei) {};

\draw[e] (11) to[in =180, out =180] (12);
\draw[e] (13) to[in =0, out =180] (h1);
\draw[e] (14) to[in =0, out =180] (h2);

\node[v, minimum size=\nodesize] (21) at (1* \wid,0*\hei) {};
\node[v, minimum size=\nodesize] (22) at (1* \wid,1*\hei) {};
\node[v, minimum size=\nodesize] (23) at (1* \wid,2*\hei) {};
\node[v, minimum size=\nodesize] (24) at (1* \wid,3*\hei) {};

\draw[e] (12) to[in =0, out =0] (13);
\draw[e] (11) to[out = 0, in = 180] (23);
\draw[e] (14) to[out = 0, in = 180] (24);
\draw[e] (21) to[out = 180, in = 180] (22);

\node[v, minimum size=\nodesize] (31) at (2* \wid,0*\hei) {};
\node[v, minimum size=\nodesize] (32) at (2* \wid,1*\hei) {};
\node[v, minimum size=\nodesize] (33) at (2* \wid,2*\hei) {};
\node[v, minimum size=\nodesize] (34) at (2* \wid,3*\hei) {};

\draw[e] (21) to[in =0, out =0] (22);
\draw[e] (23) to[out = 0, in = 180] (31);
\draw[e] (24) to[out = 0, in = 180] (32);
\draw[e] (33) to[out = 180, in = 180] (34);

\draw[e] (31) to[in =0, out =0] (34);
\draw[e] (32) to[in =0, out =0] (33);

\draw[very thick] (0* \wid,-0.25*\hei) -- (0* \wid,3.25*\hei);
\draw[very thick] (1* \wid,-0.25*\hei) -- (1* \wid,3.25*\hei);
\draw[very thick] (2* \wid,-0.25*\hei) -- (2* \wid,3.25*\hei);
\end{tikzpicture}
\end{center}
Similarly, we define \emph{right} and \emph{two-sided} open graffiti. These correspond to the bar construction where one or both of the modules has been replaced with the cell module $S(4,2)$ (or its opposite) - in particular, the differential in these complexes gives a factor of $a$ for each loop that becomes unpinned, and gives 0 if a non-loop becomes unpinned. 
\end{definition}

We will refer to all of these bar constructions as \emph{complexes of loops}.

\begin{remark} \label{rmk:SES} Closing up ends, that is, the map $S(4,2) \to S(4,0)$ which sends a link state to the link state obtained by joining up the two hanging connections, is evidently a surjection, and we obtain a short exact sequence of $\TL_4$-modules
$$0 \to K \to S(4,2) \to S(4,0) \to 0.$$
 Reflecting diagrams left-to-right, we get a dual short exact sequence
$$0 \to K^\vee \to S^\vee(2,4) \to S^\vee(0,4) \to 0.$$
One can identify the kernel $K$ as (isomorphic to) the trivial module $\t$, showing that this is the case $2n=4$ of the \emph{complex of outermost cups} defined as \cite[Definition 5.9]{BBRWS}). We will not do this.
\end{remark}

\section{The juxtaposition product and counting loops} \label{section: product and grading}

In this section we assume that $a=0$. Juxtaposing diagrams, or, more formally, the identification $$S^{\vee}(0,2n) \otimes_R S(2n,0) \cong I_0 \subset I_{\leq 2n-1} = \overline{\TL}_{2n}$$ of \cref{lem:RtensorIso}, defines a graded product on $\CPLr(2n)$, which gives it the structure of a (nonunital) dga, and which looks as follows.

\begin{center}
If $u=$
\begin{tikzpicture}[x=1.5cm,y=-.5cm,baseline=-0.7cm]

\def\wid{\standardWidth}
\def\hei{\standardHeight}
\def\nodesize{3}
\def\ang{90}

\node[v, minimum size=\nodesize] (11) at (0* \wid,0*\hei) {};
\node[v, minimum size=\nodesize] (12) at (0* \wid,1*\hei) {};
\node[v, minimum size=\nodesize] (13) at (0* \wid,2*\hei) {};
\node[v, minimum size=\nodesize] (14) at (0* \wid,3*\hei) {};

\node[v, minimum size=\nodesize] (21) at (1* \wid,0*\hei) {};
\node[v, minimum size=\nodesize] (22) at (1* \wid,1*\hei) {};
\node[v, minimum size=\nodesize] (23) at (1* \wid,2*\hei) {};
\node[v, minimum size=\nodesize] (24) at (1* \wid,3*\hei) {};

\node[v, minimum size=\nodesize] (31) at (2* \wid,0*\hei) {};
\node[v, minimum size=\nodesize] (32) at (2* \wid,1*\hei) {};
\node[v, minimum size=\nodesize] (33) at (2* \wid,2*\hei) {};
\node[v, minimum size=\nodesize] (34) at (2* \wid,3*\hei) {};

\draw[e] (11) to[in =180, out =180] (12);
\draw[e] (13) to[in =180, out =180] (14);
\draw[e] (12) to[in =0, out =0] (13);
\draw[e] (11) to[out = 0, in = 180] (23);
\draw[e] (14) to[out = 0, in = 180] (24);
\draw[e] (21) to[out = 180, in = 180] (22);
\draw[e] (21) to[in =0, out =0] (22);
\draw[e] (23) to[out = 0, in = 180] (31);
\draw[e] (24) to[out = 0, in = 180] (32);
\draw[e] (33) to[out = 180, in = 180] (34);
\draw[e] (31) to[in =0, out =0] (34);
\draw[e] (32) to[in =0, out =0] (33);

\draw[very thick] (0* \wid,-0.25*\hei) -- (0* \wid,3.25*\hei);
\draw[very thick] (1* \wid,-0.25*\hei) -- (1* \wid,3.25*\hei);
\draw[very thick] (2* \wid,-0.25*\hei) -- (2* \wid,3.25*\hei);
\end{tikzpicture}
 $\in  \CPLr(4)_3$, \quad and $v = $
\begin{tikzpicture}[x=1.5cm,y=-.5cm,baseline=-0.7cm]

\def\wid{\standardWidth}
\def\hei{\standardHeight}
\def\nodesize{3}
\def\ang{90}

\node[v, minimum size=\nodesize] (41) at (3* \wid,0*\hei) {};
\node[v, minimum size=\nodesize] (42) at (3* \wid,1*\hei) {};
\node[v, minimum size=\nodesize] (43) at (3* \wid,2*\hei) {};
\node[v, minimum size=\nodesize] (44) at (3* \wid,3*\hei) {};

\draw[e] (41) to[in =180, out =180] (44);
\draw[e] (42) to[in =180, out =180] (43);
\draw[e] (41) to[in =0, out =0] (42);
\draw[e] (44) to[in =0, out =0] (43);

\draw[very thick] (3* \wid,-0.25*\hei) -- (3* \wid,3.25*\hei);
\end{tikzpicture}
$ \in \CPLr(4)_1$,
\end{center}

\begin{center}
then the product $uv$ is the graffito
\begin{tikzpicture}[x=1.5cm,y=-.5cm,baseline=-0.7cm]

\def\wid{\standardWidth}
\def\hei{\standardHeight}
\def\nodesize{3}
\def\ang{90}

\node[v, minimum size=\nodesize] (11) at (0* \wid,0*\hei) {};
\node[v, minimum size=\nodesize] (12) at (0* \wid,1*\hei) {};
\node[v, minimum size=\nodesize] (13) at (0* \wid,2*\hei) {};
\node[v, minimum size=\nodesize] (14) at (0* \wid,3*\hei) {};

\node[v, minimum size=\nodesize] (21) at (1* \wid,0*\hei) {};
\node[v, minimum size=\nodesize] (22) at (1* \wid,1*\hei) {};
\node[v, minimum size=\nodesize] (23) at (1* \wid,2*\hei) {};
\node[v, minimum size=\nodesize] (24) at (1* \wid,3*\hei) {};

\node[v, minimum size=\nodesize] (31) at (2* \wid,0*\hei) {};
\node[v, minimum size=\nodesize] (32) at (2* \wid,1*\hei) {};
\node[v, minimum size=\nodesize] (33) at (2* \wid,2*\hei) {};
\node[v, minimum size=\nodesize] (34) at (2* \wid,3*\hei) {};

\node[v, minimum size=\nodesize] (41) at (3* \wid,0*\hei) {};
\node[v, minimum size=\nodesize] (42) at (3* \wid,1*\hei) {};
\node[v, minimum size=\nodesize] (43) at (3* \wid,2*\hei) {};
\node[v, minimum size=\nodesize] (44) at (3* \wid,3*\hei) {};

\draw[e] (11) to[in =180, out =180] (12);
\draw[e] (13) to[in =180, out =180] (14);
\draw[e] (12) to[in =0, out =0] (13);
\draw[e] (11) to[out = 0, in = 180] (23);
\draw[e] (14) to[out = 0, in = 180] (24);
\draw[e] (21) to[out = 180, in = 180] (22);
\draw[e] (21) to[in =0, out =0] (22);
\draw[e] (23) to[out = 0, in = 180] (31);
\draw[e] (24) to[out = 0, in = 180] (32);
\draw[e] (33) to[out = 180, in = 180] (34);
\draw[e] (31) to[in =0, out =0] (34);
\draw[e] (32) to[in =0, out =0] (33);

\draw[e] (41) to[in =180, out =180] (44);
\draw[e] (42) to[in =180, out =180] (43);
\draw[e] (41) to[in =0, out =0] (42);
\draw[e] (44) to[in =0, out =0] (43);

\draw[very thick] (0* \wid,-0.25*\hei) -- (0* \wid,3.25*\hei);
\draw[very thick] (1* \wid,-0.25*\hei) -- (1* \wid,3.25*\hei);
\draw[very thick] (2* \wid,-0.25*\hei) -- (2* \wid,3.25*\hei);
\draw[very thick] (3* \wid,-0.25*\hei) -- (3* \wid,3.25*\hei);
\end{tikzpicture}
$\in \CPLr(4)_4$.
\end{center}

\begin{remark} In \cite[Definition 2.4]{BBRWS} it was shown that we always have such a product on the larger complex $\CPL(2n)$. When $a=0$ we saw in \cref{rmk: reduced is same} that $\CPLr(2n)$ is a subcomplex of $\CPL(2n)$, and the product evidently restricts.
\end{remark}

The number of loops which appear in a graffito cannot change under any face map (if one of them disappeared, we would get a factor of $a=0$). This tells us that the complex $\CPLr(2n)_*$ decomposes as a direct sum over the number of loops. Write $\#(x)$ for the number of loops in a graffito $x$, and define $\CPLr(2n)_{*,w}$ to be the subcomplex of $\CPLr(2n)_*$ on graffiti $x$ with $\#(x) = w$, so that
$\CPLr(2n)_* \cong \bigoplus_{w \geq 1} \CPLr(2n)_{*,w}.$
In \cite[Definition 2.11]{BBRWS} this additional grading is called the \emph{weight}. Juxtaposing graffiti adds the numbers of loops, so the product is additive in weight. We will write $\CPLr(2n)_{*,*}$ when we wish to emphasise both gradings.

\section{The filtration by dividers} \label{section:filtrationByDividers}

In this section we assume $a=0$. We will define a filtration on $\CPLr(2n)_{*,*}$ and study the resulting spectral sequence. We will think of this filtration graphically, but we are really calculating primitives.

For a graffito $x \in \CPLr(2n)_p$, let $\myDA(x)$ denote the number of \emph{vertical dividers which can be drawn through $x$}. The following picture shows a graffito $x$, together with three red vertical dashed lines (dividers) illustrating that $\myDA(x) = 3$.

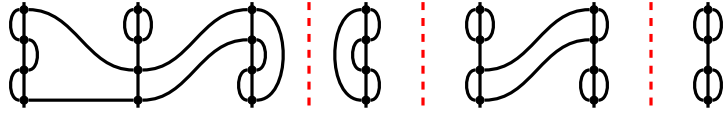
\begin{center}
\begin{tikzpicture}[x=1.5cm,y=-.5cm,baseline=-1.05cm]

\def\wid{\standardWidth}
\def\hei{\standardHeight}
\def\nodesize{3}
\def\ang{90}

\node[v, minimum size=\nodesize] (11) at (0* \wid,0*\hei) {};
\node[v, minimum size=\nodesize] (12) at (0* \wid,1*\hei) {};
\node[v, minimum size=\nodesize] (13) at (0* \wid,2*\hei) {};
\node[v, minimum size=\nodesize] (14) at (0* \wid,3*\hei) {};

\node[v, minimum size=\nodesize] (21) at (1* \wid,0*\hei) {};
\node[v, minimum size=\nodesize] (22) at (1* \wid,1*\hei) {};
\node[v, minimum size=\nodesize] (23) at (1* \wid,2*\hei) {};
\node[v, minimum size=\nodesize] (24) at (1* \wid,3*\hei) {};

\node[v, minimum size=\nodesize] (31) at (2* \wid,0*\hei) {};
\node[v, minimum size=\nodesize] (32) at (2* \wid,1*\hei) {};
\node[v, minimum size=\nodesize] (33) at (2* \wid,2*\hei) {};
\node[v, minimum size=\nodesize] (34) at (2* \wid,3*\hei) {};

\node[v, minimum size=\nodesize] (41) at (3* \wid,0*\hei) {};
\node[v, minimum size=\nodesize] (42) at (3* \wid,1*\hei) {};
\node[v, minimum size=\nodesize] (43) at (3* \wid,2*\hei) {};
\node[v, minimum size=\nodesize] (44) at (3* \wid,3*\hei) {};

\node[v, minimum size=\nodesize] (51) at (4* \wid,0*\hei) {};
\node[v, minimum size=\nodesize] (52) at (4* \wid,1*\hei) {};
\node[v, minimum size=\nodesize] (53) at (4* \wid,2*\hei) {};
\node[v, minimum size=\nodesize] (54) at (4* \wid,3*\hei) {};

\node[v, minimum size=\nodesize] (61) at (5* \wid,0*\hei) {};
\node[v, minimum size=\nodesize] (62) at (5* \wid,1*\hei) {};
\node[v, minimum size=\nodesize] (63) at (5* \wid,2*\hei) {};
\node[v, minimum size=\nodesize] (64) at (5* \wid,3*\hei) {};

\node[v, minimum size=\nodesize] (71) at (6* \wid,0*\hei) {};
\node[v, minimum size=\nodesize] (72) at (6* \wid,1*\hei) {};
\node[v, minimum size=\nodesize] (73) at (6* \wid,2*\hei) {};
\node[v, minimum size=\nodesize] (74) at (6* \wid,3*\hei) {};

\draw[e] (11) to[in =180, out =180] (12);
\draw[e] (13) to[in =180, out =180] (14);
\draw[e] (12) to[in =0, out =0] (13);
\draw[e] (11) to[out = 0, in = 180] (23);
\draw[e] (14) to[out = 0, in = 180] (24);
\draw[e] (21) to[out = 180, in = 180] (22);
\draw[e] (21) to[in =0, out =0] (22);
\draw[e] (23) to[out = 0, in = 180] (31);
\draw[e] (24) to[out = 0, in = 180] (32);
\draw[e] (33) to[out = 180, in = 180] (34);
\draw[e] (31) to[in =0, out =0] (34);
\draw[e] (32) to[in =0, out =0] (33);

\draw[e] (41) to[in =180, out =180] (44);
\draw[e] (42) to[in =180, out =180] (43);
\draw[e] (41) to[in =0, out =0] (42);
\draw[e] (44) to[in =0, out =0] (43);

\draw[e] (51) to[in =180, out =180] (52);
\draw[e] (53) to[in =180, out =180] (54);
\draw[e] (51) to[in =0, out =0] (52);
\draw[e] (53) to[out = 0, in = 180] (61);
\draw[e] (54) to[out = 0, in = 180] (62);
\draw[e] (63) to[out = 180, in = 180] (64);
\draw[e] (61) to[in =0, out =0] (62);
\draw[e] (63) to[in =0, out =0] (64);

\draw[e] (71) to[out = 180, in = 180] (72);
\draw[e] (73) to[out = 180, in = 180] (74);
\draw[e] (71) to[in =0, out =0] (72);
\draw[e] (74) to[in =0, out =0] (73);

\draw[very thick] (0* \wid,-0.25*\hei) -- (0* \wid,3.25*\hei);
\draw[very thick] (1* \wid,-0.25*\hei) -- (1* \wid,3.25*\hei);
\draw[very thick] (2* \wid,-0.25*\hei) -- (2* \wid,3.25*\hei);
\draw[very thick] (3* \wid,-0.25*\hei) -- (3* \wid,3.25*\hei);
\draw[very thick] (4* \wid,-0.25*\hei) -- (4* \wid,3.25*\hei);
\draw[very thick] (5* \wid,-0.25*\hei) -- (5* \wid,3.25*\hei);
\draw[very thick] (6* \wid,-0.25*\hei) -- (6* \wid,3.25*\hei);

\draw[very thick, red, dashed] (2.5* \wid,-0.25*\hei) -- (2.5* \wid,3.25*\hei);
\draw[very thick, red, dashed] (3.5* \wid,-0.25*\hei) -- (3.5* \wid,3.25*\hei);
\draw[very thick, red, dashed] (5.5* \wid,-0.25*\hei) -- (5.5* \wid,3.25*\hei);
\end{tikzpicture}
\end{center}

In symbols, writing $ x = \beta_0 \lvert \beta_1 \lvert \dots \lvert \beta_{p-1} \lvert \beta_p,$ where $\beta_0 \in S^{\vee}(0,2n)$, $\beta_i \in \overline{\TL}_{2n} = I_{\leq 2n-1}$, and $\beta_p \in S(2n,0)$, $\myDA(x)$ is the number of values of $i$ for which $\beta_i$ lies in $I_0$.

\begin{lemma} \label{lem:downArrowGEQ} For a graffito $x \in \CPLr(2n)_{p}$, and $i \in \{0, \dots p-1\}$, if $d_i(x)$ is nonzero, then $\myDA(d_i(x)) \in \{ \myDA(x) , \myDA(x) + 1 \}.$
\end{lemma}

\begin{proof} We must argue that deleting a bar cannot decrease the number of dividers, and can increase it by at most one. After a bar deletion, any pre-existing divider certainly still doesn't intersect the graffito, but it may:
\begin{itemize}
    \item have moved `off the end' of the graffito, if the deletion was one of the endmost bars, or
    \item have been identified with an adjacent divider, if the bar between them was deleted.
\end{itemize}
In either of these cases, the bar deletion must have unpinned a loop:
\begin{itemize}
    \item if there was a divider next to an endmost bar, then necessarily there was a loop supported only on that bar, and
    \item if there were two adjacent dividers, then the bar between them contained a loop supported only on that bar.
\end{itemize}
In other words, if the number of dividers decreases, then actually we have $d_i(x) = 0$. 

Also, a bar deletion performs only one multiplication, so creates only one new diagram, and hence can increase the number of dividers by at most one.
\end{proof}

\begin{remark} An important bar deletion increasing the number of dividers is
\begin{center}
$d_1\Biggl($
\begin{tikzpicture}[x=1.5cm,y=-.5cm,baseline=-0.7cm]

\def\wid{\standardWidth}
\def\hei{\standardHeight}
\def\nodesize{3}
\def\ang{90}

\node[v, minimum size=\nodesize] (11) at (0* \wid,0*\hei) {};
\node[v, minimum size=\nodesize] (12) at (0* \wid,1*\hei) {};
\node[v, minimum size=\nodesize] (13) at (0* \wid,2*\hei) {};
\node[v, minimum size=\nodesize] (14) at (0* \wid,3*\hei) {};

\node[v, minimum size=\nodesize] (21) at (1* \wid,0*\hei) {};
\node[v, minimum size=\nodesize] (22) at (1* \wid,1*\hei) {};
\node[v, minimum size=\nodesize] (23) at (1* \wid,2*\hei) {};
\node[v, minimum size=\nodesize] (24) at (1* \wid,3*\hei) {};

\node[v, minimum size=\nodesize] (31) at (2* \wid,0*\hei) {};
\node[v, minimum size=\nodesize] (32) at (2* \wid,1*\hei) {};
\node[v, minimum size=\nodesize] (33) at (2* \wid,2*\hei) {};
\node[v, minimum size=\nodesize] (34) at (2* \wid,3*\hei) {};

\draw[e] (11) to[in =180, out =180] (12);
\draw[e] (13) to[in =180, out =180] (14);
\draw[e] (12) to[in =0, out =0] (13);
\draw[e] (11) to[out = 0, in = 180] (23);
\draw[e] (14) to[out = 0, in = 180] (24);
\draw[e] (21) to[out = 180, in = 180] (22);
\draw[e] (23) to[in =0, out =0] (24);
\draw[e] (21) to[out = 0, in = 180] (31);
\draw[e] (22) to[out = 0, in = 180] (34);
\draw[e] (32) to[out = 180, in = 180] (33);
\draw[e] (31) to[in =0, out =0] (32);
\draw[e] (34) to[in =0, out =0] (33);

\draw[very thick] (0* \wid,-0.25*\hei) -- (0* \wid,3.25*\hei);
\draw[very thick] (1* \wid,-0.25*\hei) -- (1* \wid,3.25*\hei);
\draw[very thick] (2* \wid,-0.25*\hei) -- (2* \wid,3.25*\hei);
\end{tikzpicture}
$\Biggr) = $
\begin{tikzpicture}[x=1.5cm,y=-.5cm,baseline=-0.7cm]

\def\wid{\standardWidth}
\def\hei{\standardHeight}
\def\nodesize{3}
\def\ang{90}

\node[v, minimum size=\nodesize] (11) at (0* \wid,0*\hei) {};
\node[v, minimum size=\nodesize] (12) at (0* \wid,1*\hei) {};
\node[v, minimum size=\nodesize] (13) at (0* \wid,2*\hei) {};
\node[v, minimum size=\nodesize] (14) at (0* \wid,3*\hei) {};

\node[] (21) at (1* \wid,0*\hei) {};
\node[] (22) at (1* \wid,1*\hei) {};
\node[] (23) at (1* \wid,2*\hei) {};
\node[] (24) at (1* \wid,3*\hei) {};

\node[v, minimum size=\nodesize] (31) at (2* \wid,0*\hei) {};
\node[v, minimum size=\nodesize] (32) at (2* \wid,1*\hei) {};
\node[v, minimum size=\nodesize] (33) at (2* \wid,2*\hei) {};
\node[v, minimum size=\nodesize] (34) at (2* \wid,3*\hei) {};

\draw[e] (11) to[in =180, out =180] (12);
\draw[e] (13) to[in =180, out =180] (14);
\draw[e] (12) to[in =0, out =0] (13);
\draw[e] (11) to[out = 0, in = 180] (23.center);
\draw[e] (14) to[out = 0, in = 180] (24.center);
\draw[e] (21.center) to[out = 180, in = 180] (22.center);
\draw[e] (23.center) to[in =0, out =0] (24.center);
\draw[e] (21.center) to[out = 0, in = 180] (31);
\draw[e] (22.center) to[out = 0, in = 180] (34);
\draw[e] (32) to[out = 180, in = 180] (33);
\draw[e] (31) to[in =0, out =0] (32);
\draw[e] (34) to[in =0, out =0] (33);

\draw[very thick] (0* \wid,-0.25*\hei) -- (0* \wid,3.25*\hei);
\draw[very thick] (2* \wid,-0.25*\hei) -- (2* \wid,3.25*\hei);
\end{tikzpicture}
$=$
\begin{tikzpicture}[x=1.5cm,y=-.5cm,baseline=-0.7cm]

\def\wid{\standardWidth}
\def\hei{\standardHeight}
\def\nodesize{3}
\def\ang{90}

\node[v, minimum size=\nodesize] (11) at (0* \wid,0*\hei) {};
\node[v, minimum size=\nodesize] (12) at (0* \wid,1*\hei) {};
\node[v, minimum size=\nodesize] (13) at (0* \wid,2*\hei) {};
\node[v, minimum size=\nodesize] (14) at (0* \wid,3*\hei) {};

\node[v, minimum size=\nodesize] (31) at (1* \wid,0*\hei) {};
\node[v, minimum size=\nodesize] (32) at (1* \wid,1*\hei) {};
\node[v, minimum size=\nodesize] (33) at (1* \wid,2*\hei) {};
\node[v, minimum size=\nodesize] (34) at (1* \wid,3*\hei) {};

\draw[e] (11) to[in =180, out =180] (12);
\draw[e] (13) to[in =180, out =180] (14);
\draw[e] (12) to[in =0, out =0] (13);
\draw[e] (11) to[in =0, out =0] (14);
\draw[e] (32) to[out = 180, in = 180] (33);
\draw[e] (31) to[out = 180, in = 180] (34);
\draw[e] (31) to[in =0, out =0] (32);
\draw[e] (34) to[in =0, out =0] (33);

\draw[very thick] (0* \wid,-0.25*\hei) -- (0* \wid,3.25*\hei);
\draw[very thick] (1* \wid,-0.25*\hei) -- (1* \wid,3.25*\hei);

\draw[very thick, red, dashed] (0.5* \wid,-0.25*\hei) -- (0.5* \wid,3.25*\hei);
\end{tikzpicture}
.
\end{center}
\end{remark}

This filtration also interacts well with the juxtaposition product.
\begin{lemma}
    \label{lem:DividersOfProduct} For graffiti $x$ and $y$, we have $\myDA(xy) = \myDA(x) + \myDA(y) + 1$.
\end{lemma}

\begin{proof} The product is juxtaposition: it retains all dividers from $x$ and $y$, and we get a new divider between the two.
\end{proof}

For the formal deduction of the spectral sequence it will be convenient to flip the indexing around to get an increasing filtration. Define $$ \myDA'(x) =  \deg(x) - 1 - \myDA(x).$$  Note that $\myDA'(x)$ counts the number of \emph{non-dividers}: positions where there could have been a divider, but isn't. \cref{lem:downArrowGEQ} gives the following corollary.

\begin{corollary} \label{cor:NonDividersLESS} For a graffito $x$, we have $\myDA'(d(x)) \in \{ \myDA'(x) - 1, \myDA'(x) - 2\}$.
\end{corollary}

\begin{proof} Use \cref{lem:downArrowGEQ}, noting that \begin{align*}
    \myDA'(d(x)) & =  \deg(d(x)) - 1 - \myDA(d(x))  = \deg(x) - 2 - \myDA(d(x)). \qedhere
\end{align*}
 \end{proof}

\begin{corollary}
    \label{cor:NonDividersOfProduct} For graffiti $x$ and $y$, we have $\myDA'(xy) = \myDA'(x) + \myDA'(y)$.
\end{corollary}

\begin{proof}
    Using \cref{lem:DividersOfProduct} and the fact that the product is graded, we have \begin{align*}
    \myDA'(xy) & = \deg(xy) - 1 - \myDA(xy)  \\
    & = \deg(x) + \deg(y) - 1 - \myDA(x) - \myDA(y) - 1 \\
    & = \myDA'(x) + \myDA'(y). \qedhere
\end{align*}
\end{proof}

We are now ready to define our filtration $$F_{*,*}^{- \infty} \subset \dots \subset  F_{*,*}^p \subset F_{*,*}^{p+1} \subset \dots \subset \CPLr(2n)_{*,*}.$$

\begin{definition} \label{def: filtration} Let $F^{p}_{*,*}$ be the subcomplex of $\CPLr(2n)_{*,*}$ consisting of those graffiti $x$ for which $\myDA'(x) \leq p$ (this is a subcomplex by \cref{cor:NonDividersLESS}).
\end{definition}

The terms $F^{p}_{*,*}$ are zero for $p < 0$, since the number of non-dividers lies between 0 and $\deg(x) - 1$. By \cref{cor:NonDividersOfProduct} this is a filtration of dgas.

\section{The spectral sequence} \label{section: spectral sequence}

Again, assume $a=0$. The previous section gives a bounded increasing filtration $F^p_{*,*}$ on the dga $\CPLr(2n)_{*,*}$, which is zero in negative degrees. We get (see e.g. \cite[Section 2.2]{McCleary}) a homological type spectral sequence of algebras 
\begin{align*} E^1_{p,q,*} = H_{p+q,*}(\faktor{F^p_{*,*}}{F^{p-1}_{*,*}}) \implies & H_{p+q,*}(\CPLr(2n)_{*,*}).
\end{align*}

\begin{lemma} The $q$-line of the $E^0$-page consists of graffiti with $q-1$ dividers.
\end{lemma}

\begin{proof} If $x$ lies in bidegree $p,q$, then $\myDA'(x) = p$, and
$$q = (p+q)-p = \deg(x) - \myDA'(x) = \deg(x) - (\deg(x) - 1 - \myDA(x)) = 1 + \myDA(x). \qedhere$$
\end{proof}

The $d^0$-differential takes homology in each complex $\faktor{F^p}{F^{p-1}}$ individually. By \cref{cor:NonDividersLESS}, the boundary map strictly decreases filtration, so the differential $d^0$ is zero everywhere, and we can identify our $E^1$-page with the $E^0$-page, rewriting
\begin{align*}
    E^1_{p,q,*} \cong \faktor{F^p_{*,*}}{F^{p-1}_{*,*}} \implies & H_{p+q,*}(\CPLr(2n)_{*,*}).
\end{align*}

The grading by number of loops $w$ gives a direct sum decomposition
$$E^r_{p,q} \cong \bigoplus_{w \geq 1} E^r_{p,q,w}.$$
It follows from the above considerations that $E^1_{*,q,w}$ is the subquotient of $\CPLr(2n)_{*,*}$ consisting of graffiti with $w$ loops and $q-1$ dividers. The product is additive in $w$, so this is not a decomposition as spectral sequences of algebras.

\begin{definition} \label{def:complexCij} For $w \geq 1$ and $j \geq 0$, let the \emph{complex with $w$ loops and $j$ dividers}, denoted $C_*[w,j]$, be the subquotient of $\CPLr(2n)_{*,w}$ obtained by taking graffiti with $j$ dividers modulo those with more than $j$ dividers. Equivalently, $C_*[w,j]$ is the $q=j+1$-row of the $E^1$-page $E^1_{*,q,w}$, with differential $d^1$ and grading total degree.

Similarly, we write $C_*^L[w,j]$, $C_*^R[w,j]$, and  $C_*^{LR}[w,j]$ for the left, right, and two-sided open complexes (c.f. \cref{def: graffiti}) consisting of open graffiti where closing up (both) ends gives a graffito with $w$ loops and $j$ dividers.
\end{definition}

\begin{remark} Since the differential decreases filtration by at most 2 (in the strong sense of \cref{cor:NonDividersLESS}), the $r$-th cycles and $r$-th boundaries are constant for $r \geq 3$ so the spectral sequence collapses at $E^3$, and in particular the $E^\infty$-page is the homology of $E^2_{p,q,*}$ with respect to $d^2$. We will not use this.
\end{remark}

In our case, $2n=4$, we will show in \cref{section: mainTechnical} that the $E^2$-page is as follows. We will ultimately show (\cref{section:basechange}) that there is a compatible spectral sequence for the model, and \cref{thm: main} will follow by a comparison of $E^2$-pages.

\begin{theorem} \label{thm:mainTechnical} For any ground ring $R$, when $a=0$ and $2n=4$, we have:
\begin{enumerate}
    \item \label{part:1} The homology of the complex $C_*[1,0]$ with one loop and no dividers is a single copy of $R$ in total degree 1. It is generated by the homology class of
\begin{center}
$\varphi(x) : =$
\begin{tikzpicture}[x=1.5cm,y=-.5cm,baseline=-0.7cm]
\def\wid{\standardWidth}
\def\hei{\standardHeight}
\def\nodesize{3}
\def\ang{90}

\node[v, minimum size=\nodesize] (11) at (0* \wid,0*\hei) {};
\node[v, minimum size=\nodesize] (12) at (0* \wid,1*\hei) {};
\node[v, minimum size=\nodesize] (13) at (0* \wid,2*\hei) {};
\node[v, minimum size=\nodesize] (14) at (0* \wid,3*\hei) {};

\draw[e] (11) to[in =180, out =180] (12);
\draw[e] (13) to[in =180, out =180] (14);
\draw[e] (12) to[in =0, out =0] (13);
\draw[e] (11) to[out = 0, in = 10] (14);

\draw[very thick] (0* \wid,-0.25*\hei) -- (0* \wid,3.25*\hei);
\end{tikzpicture}
.
\end{center}
    \item \label{part:2} The homology of the complex $C_*[2,0]$ with two loops and no dividers is a single copy of $R$ in total degree 3. It is generated by the homology class of
    \begin{center}
    $\varphi(y) : =$
    \begin{tikzpicture}[x=1.5cm,y=-.5cm,baseline=-0.7cm]
        \def\wid{\standardWidth}
            \def\hei{\standardHeight}
            \def\nodesize{3}

            \node[v, minimum size=\nodesize] (01) at (-1* \wid,0*\hei) {};
            \node[v, minimum size=\nodesize] (02) at (-1* \wid,1*\hei) {};
            \node[v, minimum size=\nodesize] (03) at (-1* \wid,2*\hei) {};
            \node[v, minimum size=\nodesize] (04) at (-1* \wid,3*\hei) {};

            \node[v, minimum size=\nodesize] (11) at (0* \wid,0*\hei) {};
            \node[v, minimum size=\nodesize] (12) at (0* \wid,1*\hei) {};
            \node[v, minimum size=\nodesize] (13) at (0* \wid,2*\hei) {};
            \node[v, minimum size=\nodesize] (14) at (0* \wid,3*\hei) {};

            \node[v, minimum size=\nodesize] (21) at (1* \wid,0*\hei) {};
            \node[v, minimum size=\nodesize] (22) at (1* \wid,1*\hei) {};
            \node[v, minimum size=\nodesize] (23) at (1* \wid,2*\hei) {};
            \node[v, minimum size=\nodesize] (24) at (1* \wid,3*\hei) {};

            \draw[e] (02) to[in =0, out =0] (03);
            \draw[e] (01) to[in =180, out =180] (02);
            \draw[e] (03) to[in =180, out =180] (04);
            \draw[e] (11) to[in =180, out =180] (12);
            \draw[e] (13) to[out = 180, in = 0] (01);
            \draw[e] (14) to[out = 180, in = 0] (04);

            \draw[e] (11) to[in =0, out =0] (12);
            \draw[e] (23) to[in =180, out =180] (22);
            \draw[e] (21) to[in =0, out =0] (22);
            \draw[e] (23) to[in =0, out =0] (24);
            \draw[e] (13) to[out = 0, in = 180] (21);
            \draw[e] (14) to[out = 0, in = 180] (24);

            \draw[very thick] (0* \wid,-0.25*\hei) -- (0* \wid,3.25*\hei);
            \draw[very thick] (-1* \wid,-0.25*\hei) -- (-1* \wid,3.25*\hei);
            \draw[very thick] (1* \wid,-0.25*\hei) -- (1* \wid,3.25*\hei);
    \end{tikzpicture}
    $+$ 
    \begin{tikzpicture}[x=1.5cm,y=-.5cm,baseline=-0.7cm]
       \def\wid{\standardWidth}
            \def\hei{\standardHeight}
            \def\nodesize{3}

            \node[v, minimum size=\nodesize] (01) at (-1* \wid,0*\hei) {};
            \node[v, minimum size=\nodesize] (02) at (-1* \wid,1*\hei) {};
            \node[v, minimum size=\nodesize] (03) at (-1* \wid,2*\hei) {};
            \node[v, minimum size=\nodesize] (04) at (-1* \wid,3*\hei) {};

            \node[v, minimum size=\nodesize] (11) at (0* \wid,0*\hei) {};
            \node[v, minimum size=\nodesize] (12) at (0* \wid,1*\hei) {};
            \node[v, minimum size=\nodesize] (13) at (0* \wid,2*\hei) {};
            \node[v, minimum size=\nodesize] (14) at (0* \wid,3*\hei) {};

            \node[v, minimum size=\nodesize] (21) at (1* \wid,0*\hei) {};
            \node[v, minimum size=\nodesize] (22) at (1* \wid,1*\hei) {};
            \node[v, minimum size=\nodesize] (23) at (1* \wid,2*\hei) {};
            \node[v, minimum size=\nodesize] (24) at (1* \wid,3*\hei) {};

            \draw[e] (02) to[in =0, out =0] (03);
            \draw[e] (01) to[in =180, out =180] (02);
            \draw[e] (03) to[in =180, out =180] (04);
            \draw[e] (13) to[in =180, out =180] (14);
            \draw[e] (11) to[out = 180, in = 0] (01);
            \draw[e] (12) to[out = 180, in = 0] (04);

            \draw[e] (13) to[in =0, out =0] (14);
            \draw[e] (23) to[in =180, out =180] (22);
            \draw[e] (21) to[in =0, out =0] (22);
            \draw[e] (23) to[in =0, out =0] (24);
            \draw[e] (11) to[out = 0, in = 180] (21);
            \draw[e] (12) to[out = 0, in = 180] (24);

            \draw[very thick] (0* \wid,-0.25*\hei) -- (0* \wid,3.25*\hei);
            \draw[very thick] (-1* \wid,-0.25*\hei) -- (-1* \wid,3.25*\hei);
            \draw[very thick] (1* \wid,-0.25*\hei) -- (1* \wid,3.25*\hei); 
    \end{tikzpicture}
    \\
    $-$
    \begin{tikzpicture}[x=1.5cm,y=-.5cm,baseline=-0.7cm]
       \def\wid{\standardWidth}
            \def\hei{\standardHeight}
            \def\nodesize{3}

            \node[v, minimum size=\nodesize] (01) at (-1* \wid,0*\hei) {};
            \node[v, minimum size=\nodesize] (02) at (-1* \wid,1*\hei) {};
            \node[v, minimum size=\nodesize] (03) at (-1* \wid,2*\hei) {};
            \node[v, minimum size=\nodesize] (04) at (-1* \wid,3*\hei) {};

            \node[v, minimum size=\nodesize] (11) at (0* \wid,0*\hei) {};
            \node[v, minimum size=\nodesize] (12) at (0* \wid,1*\hei) {};
            \node[v, minimum size=\nodesize] (13) at (0* \wid,2*\hei) {};
            \node[v, minimum size=\nodesize] (14) at (0* \wid,3*\hei) {};

            \node[v, minimum size=\nodesize] (21) at (1* \wid,0*\hei) {};
            \node[v, minimum size=\nodesize] (22) at (1* \wid,1*\hei) {};
            \node[v, minimum size=\nodesize] (23) at (1* \wid,2*\hei) {};
            \node[v, minimum size=\nodesize] (24) at (1* \wid,3*\hei) {};

            \draw[e] (02) to[in =0, out =0] (03);
            \draw[e] (01) to[in =180, out =180] (02);
            \draw[e] (03) to[in =180, out =180] (04);
            \draw[e] (11) to[in =180, out =180] (12);
            \draw[e] (13) to[out = 180, in = 0] (01);
            \draw[e] (14) to[out = 180, in = 0] (04);

            \draw[e] (13) to[in =0, out =0] (14);
            \draw[e] (23) to[in =180, out =180] (22);
            \draw[e] (21) to[in =0, out =0] (22);
            \draw[e] (23) to[in =0, out =0] (24);
            \draw[e] (11) to[out = 0, in = 180] (21);
            \draw[e] (12) to[out = 0, in = 180] (24);

            \draw[very thick] (0* \wid,-0.25*\hei) -- (0* \wid,3.25*\hei);
            \draw[very thick] (-1* \wid,-0.25*\hei) -- (-1* \wid,3.25*\hei);
            \draw[very thick] (1* \wid,-0.25*\hei) -- (1* \wid,3.25*\hei); 
    \end{tikzpicture}
    $-$
    \begin{tikzpicture}[x=1.5cm,y=-.5cm,baseline=-0.7cm]
       \def\wid{\standardWidth}
            \def\hei{\standardHeight}
            \def\nodesize{3}

            \node[v, minimum size=\nodesize] (01) at (-1* \wid,0*\hei) {};
            \node[v, minimum size=\nodesize] (02) at (-1* \wid,1*\hei) {};
            \node[v, minimum size=\nodesize] (03) at (-1* \wid,2*\hei) {};
            \node[v, minimum size=\nodesize] (04) at (-1* \wid,3*\hei) {};

            \node[v, minimum size=\nodesize] (11) at (0* \wid,0*\hei) {};
            \node[v, minimum size=\nodesize] (12) at (0* \wid,1*\hei) {};
            \node[v, minimum size=\nodesize] (13) at (0* \wid,2*\hei) {};
            \node[v, minimum size=\nodesize] (14) at (0* \wid,3*\hei) {};

            \node[v, minimum size=\nodesize] (21) at (1* \wid,0*\hei) {};
            \node[v, minimum size=\nodesize] (22) at (1* \wid,1*\hei) {};
            \node[v, minimum size=\nodesize] (23) at (1* \wid,2*\hei) {};
            \node[v, minimum size=\nodesize] (24) at (1* \wid,3*\hei) {};

            \draw[e] (02) to[in =0, out =0] (03);
            \draw[e] (01) to[in =180, out =180] (02);
            \draw[e] (03) to[in =180, out =180] (04);
            \draw[e] (13) to[in =180, out =180] (14);
            \draw[e] (11) to[out = 180, in = 0] (01);
            \draw[e] (12) to[out = 180, in = 0] (04);

            \draw[e] (11) to[in =0, out =0] (12);
            \draw[e] (23) to[in =180, out =180] (22);
            \draw[e] (21) to[in =0, out =0] (22);
            \draw[e] (23) to[in =0, out =0] (24);
            \draw[e] (13) to[out = 0, in = 180] (21);
            \draw[e] (14) to[out = 0, in = 180] (24);

            \draw[very thick] (0* \wid,-0.25*\hei) -- (0* \wid,3.25*\hei);
            \draw[very thick] (-1* \wid,-0.25*\hei) -- (-1* \wid,3.25*\hei);
            \draw[very thick] (1* \wid,-0.25*\hei) -- (1* \wid,3.25*\hei); 
    \end{tikzpicture}
.
\end{center}
    \item \label{part:igeq3} For $w \geq 3$, the homology of the complex $C_*[w,0]$ vanishes.
\end{enumerate}    
\end{theorem}

This can be thought of as describing the $E^2$-page as a tensor algebra, as follows.

\begin{corollary} \label{cor: E2 as tensor} The underlying algebra of the $E^2$-page $E^2_{*,*,*}$ of the spectral sequence associated to the filtration by dividers is the nonunital tensor algebra on the above classes $\varphi(x)$ and $\varphi(y)$.
\end{corollary}

\begin{proof}  We think of the complex $C_*[w,p]$ with $w$ loops and $q-1$ dividers as `tensored over the dividers' (over $R$). Since the dividers don't interact at all with the differential,
$$C_*[w,q-1] \cong \bigoplus_{w_1 + \dots + w_{q} = w} C_*[w_1,0] \otimes_R \dots \otimes_R C_*[w_{q},0],$$
so the underlying algebra of the $E^1$-page $E^1_{*,*,*}$ of the whole spectral sequence is the nonunital tensor algebra on $E^1_{*,1,*}=C_*[*,0]$

By \cref{thm:mainTechnical}, the homology of each factor $C_*[w,0]$ is in particular $R$-free, so homology commutes with the tensor product, and the underlying algebra of the $E^2$-page $E^2_{*,*,*}$ of the whole spectral sequence is the nonunital tensor algebra on $E^2_{*,1,*}=H_*(C_*[*,0])$, which we computed as \cref{thm:mainTechnical}. The result follows.
\end{proof}

\section{Proof of \cref{thm:mainTechnical}} \label{section: mainTechnical}

We will prove the three parts individually. In this section $a=0$ and $2n=4$.

\subsection{A basic observation}

The following lemma will be useful.

\begin{lemma} \label{lem:basicObs}Let $S$ be a nonempty set of letters. Let $A_*$ be the dga with 
\begin{itemize}
    \item underlying algebra the augmentation ideal of the free associative algebra $T_R[S]$, where $\deg(x) = 1$ for all $x \in S$ (so degree is just word length), and
    \item differential in degree $p$ the alternating sum $\sum_{i=1}^{p} (-1)^{i+1}d_i$, where $d_i$ deletes the $i$-th letter of a monomial.
\end{itemize}
Then $A_*$ is contractible, with homology generated by any single letter $\ell$, i.e. $$H_i(A_*) \cong \begin{cases} R\{\ell\} & i=1 \\ 
0 & \textrm{otherwise.}
\end{cases}$$
\end{lemma}

For shorthand, we will call the complex $A_*$ the \emph{complex of words on $S$}.

\begin{proof} Choose a letter $\ell \in S$, and consider the map $ H: A_* \xrightarrow{w \mapsto \ell w} A_*$, which has degree 1. If $w$ is a word in $S$ of length at least $2$, then $$d(H(w)) = d(\ell w) = w - \ell d(w) = w - H(d(w)).$$

A word of length $1$ is just a single letter, $x$ say, and since $d(x)=0$, $$d(H(x)) + H(d(x))=d(H(x)) = d (\ell x) = x - \ell.$$

This shows that $H$ is a chain homotopy between the identity, and the map $f$ which is zero in degrees $\geq 2$ and which in degree 1 sends each letter $x$ to $\ell$. The claim follows.
\end{proof}

\subsection{A change of perspective} \label{subsection: ChangeOfPerspective}

We want to make our complexes look more like complexes of words. We have seen that by definition we may write a graffito $x \in \CPLr(4)_{p,*}$ as $$\beta_0 \lvert \beta_1 \lvert \dots \lvert \beta_{p-1} \lvert \beta_p,$$
where $\beta_0\in S^{\vee}(0,4)$, $\beta_i \in \overline{\TL}_{4} = I_{\leq 3}$, and $\beta_p \in S(4,0)$.

A Temperley--Lieb diagram is determined by its two link states: having selected these, there is only one planar way of joining the edges up. We want to change our point of view using this observation so that the parts of the diagrams living on the bars are what's important. In pictures, this change of perspective is as follows.

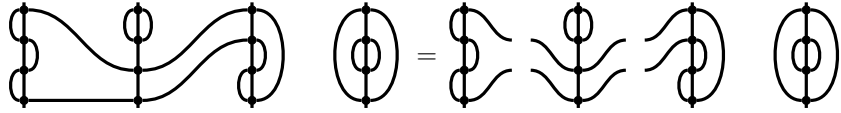
\begin{center}
\begin{tikzpicture}[x=1.5cm,y=-.5cm,baseline=-0.7cm]

\def\wid{\standardWidth}
\def\hei{\standardHeight}
\def\nodesize{3}
\def\ang{90}

\node[v, minimum size=\nodesize] (11) at (0* \wid,0*\hei) {};
\node[v, minimum size=\nodesize] (12) at (0* \wid,1*\hei) {};
\node[v, minimum size=\nodesize] (13) at (0* \wid,2*\hei) {};
\node[v, minimum size=\nodesize] (14) at (0* \wid,3*\hei) {};

\draw[e] (11) to[in =180, out =180] (12);
\draw[e] (13) to[in =180, out =180] (14);

\node[v, minimum size=\nodesize] (21) at (1* \wid,0*\hei) {};
\node[v, minimum size=\nodesize] (22) at (1* \wid,1*\hei) {};
\node[v, minimum size=\nodesize] (23) at (1* \wid,2*\hei) {};
\node[v, minimum size=\nodesize] (24) at (1* \wid,3*\hei) {};

\draw[e] (12) to[in =0, out =0] (13);
\draw[e] (11) to[out = 0, in = 180] (23);
\draw[e] (14) to[out = 0, in = 180] (24);
\draw[e] (21) to[out = 180, in = 180] (22);

\node[v, minimum size=\nodesize] (31) at (2* \wid,0*\hei) {};
\node[v, minimum size=\nodesize] (32) at (2* \wid,1*\hei) {};
\node[v, minimum size=\nodesize] (33) at (2* \wid,2*\hei) {};
\node[v, minimum size=\nodesize] (34) at (2* \wid,3*\hei) {};

\draw[e] (21) to[in =0, out =0] (22);
\draw[e] (23) to[out = 0, in = 180] (31);
\draw[e] (24) to[out = 0, in = 180] (32);
\draw[e] (33) to[out = 180, in = 180] (34);

\draw[e] (31) to[in =0, out =0] (34);
\draw[e] (32) to[in =0, out =0] (33);

\node[v, minimum size=\nodesize] (41) at (3* \wid,0*\hei) {};
\node[v, minimum size=\nodesize] (42) at (3* \wid,1*\hei) {};
\node[v, minimum size=\nodesize] (43) at (3* \wid,2*\hei) {};
\node[v, minimum size=\nodesize] (44) at (3* \wid,3*\hei) {};

\draw[e] (41) to[in =180, out =180] (44);
\draw[e] (42) to[in =180, out =180] (43);

\draw[e] (41) to[in =0, out =0] (44);
\draw[e] (42) to[in =0, out =0] (43);

\draw[very thick] (0* \wid,-0.25*\hei) -- (0* \wid,3.25*\hei);
\draw[very thick] (1* \wid,-0.25*\hei) -- (1* \wid,3.25*\hei);
\draw[very thick] (2* \wid,-0.25*\hei) -- (2* \wid,3.25*\hei);
\draw[very thick] (3* \wid,-0.25*\hei) -- (3* \wid,3.25*\hei);
\end{tikzpicture}
=
\begin{tikzpicture}[x=1.5cm,y=-.5cm,baseline=-0.7cm]

\def\wid{\standardWidth}
\def\hei{\standardHeight}
\def\nodesize{3}
\def\ang{90}

\node[v, minimum size=\nodesize] (11) at (0* \wid,0*\hei) {};
\node[v, minimum size=\nodesize] (12) at (0* \wid,1*\hei) {};
\node[v, minimum size=\nodesize] (13) at (0* \wid,2*\hei) {};
\node[v, minimum size=\nodesize] (14) at (0* \wid,3*\hei) {};

\node[] (1h1) at (0.5* \wid,1*\hei) {};
\node[] (1h2) at (0.5* \wid,2*\hei) {};

\draw[e] (11) to[in =180, out =180] (12);
\draw[e] (13) to[in =180, out =180] (14);

\node[v, minimum size=\nodesize] (21) at (1* \wid,0*\hei) {};
\node[v, minimum size=\nodesize] (22) at (1* \wid,1*\hei) {};
\node[v, minimum size=\nodesize] (23) at (1* \wid,2*\hei) {};
\node[v, minimum size=\nodesize] (24) at (1* \wid,3*\hei) {};

\node[] (2h1) at (1.5* \wid,1*\hei) {};
\node[] (2h2) at (1.5* \wid,2*\hei) {};

\draw[e] (12) to[in =0, out =0] (13);
\draw[e] (11) to[out = 0, in = 180] (1h1);
\draw[e] (1h1) to[out = 0, in = 180] (23);
\draw[e] (14) to[out = 0, in = 180] (1h2);
\draw[e] (1h2) to[out = 0, in = 180] (24);
\draw[e] (21) to[out = 180, in = 180] (22);

\node[v, minimum size=\nodesize] (31) at (2* \wid,0*\hei) {};
\node[v, minimum size=\nodesize] (32) at (2* \wid,1*\hei) {};
\node[v, minimum size=\nodesize] (33) at (2* \wid,2*\hei) {};
\node[v, minimum size=\nodesize] (34) at (2* \wid,3*\hei) {};

\draw[e] (21) to[in =0, out =0] (22);
\draw[e] (23) to[out = 0, in = 180] (2h1);
\draw[e] (2h1) to[out = 0, in = 180] (31);
\draw[e] (24) to[out = 0, in = 180] (2h2);
\draw[e] (2h2) to[out = 0, in = 180] (32);
\draw[e] (33) to[out = 180, in = 180] (34);

\draw[e] (31) to[in =0, out =0] (34);
\draw[e] (32) to[in =0, out =0] (33);

\node[v, minimum size=\nodesize] (41) at (3* \wid,0*\hei) {};
\node[v, minimum size=\nodesize] (42) at (3* \wid,1*\hei) {};
\node[v, minimum size=\nodesize] (43) at (3* \wid,2*\hei) {};
\node[v, minimum size=\nodesize] (44) at (3* \wid,3*\hei) {};

\draw[e] (41) to[in =180, out =180] (44);
\draw[e] (42) to[in =180, out =180] (43);

\draw[e] (41) to[in =0, out =0] (44);
\draw[e] (42) to[in =0, out =0] (43);

\draw[very thick] (0* \wid,-0.25*\hei) -- (0* \wid,3.25*\hei);
\draw[very thick] (1* \wid,-0.25*\hei) -- (1* \wid,3.25*\hei);
\draw[very thick] (2* \wid,-0.25*\hei) -- (2* \wid,3.25*\hei);
\draw[very thick] (3* \wid,-0.25*\hei) -- (3* \wid,3.25*\hei);
\end{tikzpicture}

\end{center}

In algebraic terms, what we are doing is this. Recall from \cref{section:basicDefinitions} that we can write a diagram $\beta$ having $k$ left-to-right connections as the tensor product $\beta = \beta^L \otimes_R \beta^R$ of its left and right link states (abusing the isomorphism of \cref{lem:RtensorIso}). In this notation, our procedure is a `taking apart and recombining' which rewrites a graffito $x = \beta_0 \lvert \beta_1 \lvert \dots \lvert \beta_{p-1} \lvert \beta_p$ as follows:
\begin{align*}
    \beta_0 \lvert \beta_1 \lvert \dots \lvert \beta_{p-1} \lvert \beta_p &  = \beta_0 \lvert ( \beta_1^L \otimes \beta_1^R)  \lvert ( \beta_2^L \otimes \beta_2^R ) \lvert \dots \lvert ( \beta_{p-1}^L \otimes \beta_{p-1}^R) \lvert \beta_p \\
    & = (\beta_0 \lvert \beta_1^L) \otimes (\beta_1^R \lvert \beta_2^L) \otimes \dots \otimes (\beta_{p-1}^R \lvert \beta_p),
\end{align*}
remembering that the vertical bars are just a different notation for tensor product (though it helps here to keep the two `different' tensor products distinct). We will regard this latter expression as a word in all `letters' which could arise as one of the diagrams $(\alpha \lvert \beta_1^L), (\beta_1^R \lvert \beta_2^L) , \dots , (\beta_{p-1}^R \lvert \gamma)$.

Enumerating the possible letters gives the following alphabet.

\begin{lemma} \label{lem:alphabet} Any graffito may be represented uniquely as a word in the following letters, which we file by how many connections they have on each side.

Two connections each side:

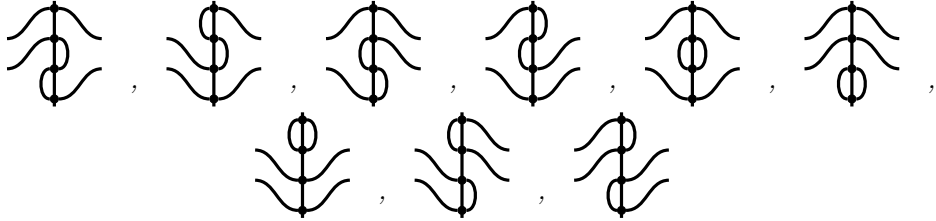
\begin{center}
\begin{tikzpicture}[x=1.5cm,y=-.5cm,baseline=-1.05cm]

\def\wid{\standardWidth}
\def\hei{\standardHeight}
\def\nodesize{3}
\def\ang{90}

\node[] (0h1) at (-0.5* \wid,1*\hei) {};
\node[] (0h2) at (-0.5* \wid,2*\hei) {};
\node[v, minimum size=\nodesize] (11) at (0* \wid,0*\hei) {};
\node[v, minimum size=\nodesize] (12) at (0* \wid,1*\hei) {};
\node[v, minimum size=\nodesize] (13) at (0* \wid,2*\hei) {};
\node[v, minimum size=\nodesize] (14) at (0* \wid,3*\hei) {};
\node[] (1h1) at (0.5* \wid,1*\hei) {};
\node[] (1h2) at (0.5* \wid,2*\hei) {};

\draw[e] (11) to[in =0, out =180] (0h1);
\draw[e] (12) to[in =0, out =180] (0h2);
\draw[e] (13) to[in =180, out =180] (14);

\draw[e] (12) to[in =0, out =0] (13);
\draw[e] (11) to[out = 0, in = 180] (1h1);
\draw[e] (14) to[out = 0, in = 180] (1h2);

\draw[very thick] (0* \wid,-0.25*\hei) -- (0* \wid,3.25*\hei);
\end{tikzpicture}
,
\begin{tikzpicture}[x=1.5cm,y=-.5cm,baseline=-1.05cm]

\def\wid{\standardWidth}
\def\hei{\standardHeight}
\def\nodesize{3}
\def\ang{90}

\node[] (0h1) at (-0.5* \wid,1*\hei) {};
\node[] (0h2) at (-0.5* \wid,2*\hei) {};
\node[v, minimum size=\nodesize] (11) at (0* \wid,0*\hei) {};
\node[v, minimum size=\nodesize] (12) at (0* \wid,1*\hei) {};
\node[v, minimum size=\nodesize] (13) at (0* \wid,2*\hei) {};
\node[v, minimum size=\nodesize] (14) at (0* \wid,3*\hei) {};
\node[] (1h1) at (0.5* \wid,1*\hei) {};
\node[] (1h2) at (0.5* \wid,2*\hei) {};

\draw[e] (13) to[in =00, out =180] (0h1);
\draw[e] (14) to[in =0, out =180] (0h2);
\draw[e] (12) to[in =180, out =180] (11);

\draw[e] (12) to[in =0, out =0] (13);
\draw[e] (11) to[out = 0, in = 180] (1h1);
\draw[e] (14) to[out = 0, in = 180] (1h2);

\draw[very thick] (0* \wid,-0.25*\hei) -- (0* \wid,3.25*\hei);
\end{tikzpicture}
,
\begin{tikzpicture}[x=1.5cm,y=-.5cm,baseline=-1.05cm]

\def\wid{\standardWidth}
\def\hei{\standardHeight}
\def\nodesize{3}
\def\ang{90}

\node[] (0h1) at (-0.5* \wid,1*\hei) {};
\node[] (0h2) at (-0.5* \wid,2*\hei) {};
\node[v, minimum size=\nodesize] (11) at (0* \wid,0*\hei) {};
\node[v, minimum size=\nodesize] (12) at (0* \wid,1*\hei) {};
\node[v, minimum size=\nodesize] (13) at (0* \wid,2*\hei) {};
\node[v, minimum size=\nodesize] (14) at (0* \wid,3*\hei) {};
\node[] (1h1) at (0.5* \wid,1*\hei) {};
\node[] (1h2) at (0.5* \wid,2*\hei) {};

\draw[e] (11) to[in =0, out =180] (0h1);
\draw[e] (14) to[in =0, out =180] (0h2);
\draw[e] (13) to[in =180, out =180] (12);

\draw[e] (13) to[in =0, out =0] (14);
\draw[e] (11) to[out = 0, in = 180] (1h1);
\draw[e] (12) to[out = 0, in = 180] (1h2);

\draw[very thick] (0* \wid,-0.25*\hei) -- (0* \wid,3.25*\hei);
\end{tikzpicture}
,
\begin{tikzpicture}[x=1.5cm,y=-.5cm,baseline=-1.05cm]

\def\wid{\standardWidth}
\def\hei{\standardHeight}
\def\nodesize{3}
\def\ang{90}

\node[] (0h1) at (-0.5* \wid,1*\hei) {};
\node[] (0h2) at (-0.5* \wid,2*\hei) {};
\node[v, minimum size=\nodesize] (11) at (0* \wid,0*\hei) {};
\node[v, minimum size=\nodesize] (12) at (0* \wid,1*\hei) {};
\node[v, minimum size=\nodesize] (13) at (0* \wid,2*\hei) {};
\node[v, minimum size=\nodesize] (14) at (0* \wid,3*\hei) {};
\node[] (1h1) at (0.5* \wid,1*\hei) {};
\node[] (1h2) at (0.5* \wid,2*\hei) {};

\draw[e] (11) to[in =0, out =180] (0h1);
\draw[e] (14) to[in =0, out =180] (0h2);
\draw[e] (12) to[in =180, out =180] (13);

\draw[e] (12) to[in =0, out =0] (11);
\draw[e] (14) to[out = 0, in = 180] (1h2);
\draw[e] (13) to[out = 0, in = 180] (1h1);

\draw[very thick] (0* \wid,-0.25*\hei) -- (0* \wid,3.25*\hei);
\end{tikzpicture}
,
\begin{tikzpicture}[x=1.5cm,y=-.5cm,baseline=-1.05cm]

\def\wid{\standardWidth}
\def\hei{\standardHeight}
\def\nodesize{3}
\def\ang{90}

\node[] (0h1) at (-0.5* \wid,1*\hei) {};
\node[] (0h2) at (-0.5* \wid,2*\hei) {};
\node[v, minimum size=\nodesize] (11) at (0* \wid,0*\hei) {};
\node[v, minimum size=\nodesize] (12) at (0* \wid,1*\hei) {};
\node[v, minimum size=\nodesize] (13) at (0* \wid,2*\hei) {};
\node[v, minimum size=\nodesize] (14) at (0* \wid,3*\hei) {};
\node[] (1h1) at (0.5* \wid,1*\hei) {};
\node[] (1h2) at (0.5* \wid,2*\hei) {};

\draw[e] (11) to[in =0, out =180] (0h1);
\draw[e] (14) to[in =0, out =180] (0h2);
\draw[e] (13) to[in =180, out =180] (12);

\draw[e] (12) to[in =0, out =0] (13);
\draw[e] (11) to[out = 0, in = 180] (1h1);
\draw[e] (14) to[out = 0, in = 180] (1h2);

\draw[very thick] (0* \wid,-0.25*\hei) -- (0* \wid,3.25*\hei);
\end{tikzpicture}
,
\begin{tikzpicture}[x=1.5cm,y=-.5cm,baseline=-1.05cm]

\def\wid{\standardWidth}
\def\hei{\standardHeight}
\def\nodesize{3}
\def\ang{90}

\node[] (0h1) at (-0.5* \wid,1*\hei) {};
\node[] (0h2) at (-0.5* \wid,2*\hei) {};
\node[v, minimum size=\nodesize] (11) at (0* \wid,0*\hei) {};
\node[v, minimum size=\nodesize] (12) at (0* \wid,1*\hei) {};
\node[v, minimum size=\nodesize] (13) at (0* \wid,2*\hei) {};
\node[v, minimum size=\nodesize] (14) at (0* \wid,3*\hei) {};
\node[] (1h1) at (0.5* \wid,1*\hei) {};
\node[] (1h2) at (0.5* \wid,2*\hei) {};

\draw[e] (11) to[in =0, out =180] (0h1);
\draw[e] (12) to[in =0, out =180] (0h2);
\draw[e] (13) to[in =180, out =180] (14);

\draw[e] (13) to[in =0, out =0] (14);
\draw[e] (11) to[out = 0, in = 180] (1h1);
\draw[e] (12) to[out = 0, in = 180] (1h2);

\draw[very thick] (0* \wid,-0.25*\hei) -- (0* \wid,3.25*\hei);
\end{tikzpicture}
,
\begin{tikzpicture}[x=1.5cm,y=-.5cm,baseline=-1.05cm]

\def\wid{\standardWidth}
\def\hei{\standardHeight}
\def\nodesize{3}
\def\ang{90}

\node[] (0h1) at (-0.5* \wid,1*\hei) {};
\node[] (0h2) at (-0.5* \wid,2*\hei) {};
\node[v, minimum size=\nodesize] (11) at (0* \wid,0*\hei) {};
\node[v, minimum size=\nodesize] (12) at (0* \wid,1*\hei) {};
\node[v, minimum size=\nodesize] (13) at (0* \wid,2*\hei) {};
\node[v, minimum size=\nodesize] (14) at (0* \wid,3*\hei) {};
\node[] (1h1) at (0.5* \wid,1*\hei) {};
\node[] (1h2) at (0.5* \wid,2*\hei) {};

\draw[e] (12) to[in =180, out =180] (11);
\draw[e] (13) to[in =0, out =180] (0h1);
\draw[e] (14) to[in =0, out =180] (0h2);

\draw[e] (12) to[in =0, out =0] (11);
\draw[e] (14) to[out = 0, in = 180] (1h2);
\draw[e] (13) to[out = 0, in = 180] (1h1);

\draw[very thick] (0* \wid,-0.25*\hei) -- (0* \wid,3.25*\hei);
\end{tikzpicture}
,
\begin{tikzpicture}[x=1.5cm,y=-.5cm,baseline=-1.05cm]

\def\wid{\standardWidth}
\def\hei{\standardHeight}
\def\nodesize{3}
\def\ang{90}

\node[] (0h1) at (-0.5* \wid,1*\hei) {};
\node[] (0h2) at (-0.5* \wid,2*\hei) {};
\node[v, minimum size=\nodesize] (11) at (0* \wid,0*\hei) {};
\node[v, minimum size=\nodesize] (12) at (0* \wid,1*\hei) {};
\node[v, minimum size=\nodesize] (13) at (0* \wid,2*\hei) {};
\node[v, minimum size=\nodesize] (14) at (0* \wid,3*\hei) {};
\node[] (1h1) at (0.5* \wid,1*\hei) {};
\node[] (1h2) at (0.5* \wid,2*\hei) {};

\draw[e] (13) to[in =0, out =180] (0h1);
\draw[e] (14) to[in =0, out =180] (0h2);
\draw[e] (11) to[in =180, out =180] (12);

\draw[e] (13) to[in =0, out =0] (14);
\draw[e] (11) to[out = 0, in = 180] (1h1);
\draw[e] (12) to[out = 0, in = 180] (1h2);

\draw[very thick] (0* \wid,-0.25*\hei) -- (0* \wid,3.25*\hei);
\end{tikzpicture}
,
\begin{tikzpicture}[x=1.5cm,y=-.5cm,baseline=-1.05cm]

\def\wid{\standardWidth}
\def\hei{\standardHeight}
\def\nodesize{3}
\def\ang{90}

\node[] (0h1) at (-0.5* \wid,1*\hei) {};
\node[] (0h2) at (-0.5* \wid,2*\hei) {};
\node[v, minimum size=\nodesize] (11) at (0* \wid,0*\hei) {};
\node[v, minimum size=\nodesize] (12) at (0* \wid,1*\hei) {};
\node[v, minimum size=\nodesize] (13) at (0* \wid,2*\hei) {};
\node[v, minimum size=\nodesize] (14) at (0* \wid,3*\hei) {};
\node[] (1h1) at (0.5* \wid,1*\hei) {};
\node[] (1h2) at (0.5* \wid,2*\hei) {};

\draw[e] (13) to[in =180, out =180] (14);
\draw[e] (11) to[in =0, out =180] (0h1);
\draw[e] (12) to[in =0, out =180] (0h2);

\draw[e] (12) to[in =0, out =0] (11);
\draw[e] (14) to[out = 0, in = 180] (1h2);
\draw[e] (13) to[out = 0, in = 180] (1h1);

\draw[very thick] (0* \wid,-0.25*\hei) -- (0* \wid,3.25*\hei);
\end{tikzpicture}
\end{center}

No connections on the left, two on the right:

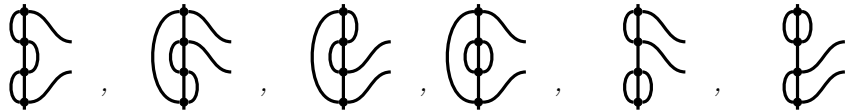
\begin{center}
\begin{tikzpicture}[x=1.5cm,y=-.5cm,baseline=-1.05cm]

\def\wid{\standardWidth}
\def\hei{\standardHeight}
\def\nodesize{3}
\def\ang{90}

\node[] (0h1) at (0.5* \wid,1*\hei) {};
\node[] (0h2) at (0.5* \wid,2*\hei) {};
\node[v, minimum size=\nodesize] (11) at (0* \wid,0*\hei) {};
\node[v, minimum size=\nodesize] (12) at (0* \wid,1*\hei) {};
\node[v, minimum size=\nodesize] (13) at (0* \wid,2*\hei) {};
\node[v, minimum size=\nodesize] (14) at (0* \wid,3*\hei) {};
\node[] (1h1) at (0.5* \wid,1*\hei) {};
\node[] (1h2) at (0.5* \wid,2*\hei) {};

\draw[e] (11) to[in =180, out =180] (12);
\draw[e] (13) to[in =180, out =180] (14);

\draw[e] (12) to[in =0, out =0] (13);
\draw[e] (11) to[out = 0, in = 180] (1h1);
\draw[e] (14) to[out = 0, in = 180] (1h2);

\draw[very thick] (0* \wid,-0.25*\hei) -- (0* \wid,3.25*\hei);
\end{tikzpicture}
,
\begin{tikzpicture}[x=1.5cm,y=-.5cm,baseline=-1.05cm]

\def\wid{\standardWidth}
\def\hei{\standardHeight}
\def\nodesize{3}
\def\ang{90}

\node[] (0h1) at (-0.5* \wid,1*\hei) {};
\node[] (0h2) at (-0.5* \wid,2*\hei) {};
\node[v, minimum size=\nodesize] (11) at (0* \wid,0*\hei) {};
\node[v, minimum size=\nodesize] (12) at (0* \wid,1*\hei) {};
\node[v, minimum size=\nodesize] (13) at (0* \wid,2*\hei) {};
\node[v, minimum size=\nodesize] (14) at (0* \wid,3*\hei) {};
\node[] (1h1) at (0.5* \wid,1*\hei) {};
\node[] (1h2) at (0.5* \wid,2*\hei) {};

\draw[e] (11) to[in =180, out =180] (14);
\draw[e] (13) to[in =180, out =180] (12);

\draw[e] (13) to[in =0, out =0] (14);
\draw[e] (11) to[out = 0, in = 180] (1h1);
\draw[e] (12) to[out = 0, in = 180] (1h2);

\draw[very thick] (0* \wid,-0.25*\hei) -- (0* \wid,3.25*\hei);
\end{tikzpicture}
,
\begin{tikzpicture}[x=1.5cm,y=-.5cm,baseline=-1.05cm]

\def\wid{\standardWidth}
\def\hei{\standardHeight}
\def\nodesize{3}
\def\ang{90}

\node[] (0h1) at (-0.5* \wid,1*\hei) {};
\node[] (0h2) at (-0.5* \wid,2*\hei) {};
\node[v, minimum size=\nodesize] (11) at (0* \wid,0*\hei) {};
\node[v, minimum size=\nodesize] (12) at (0* \wid,1*\hei) {};
\node[v, minimum size=\nodesize] (13) at (0* \wid,2*\hei) {};
\node[v, minimum size=\nodesize] (14) at (0* \wid,3*\hei) {};
\node[] (1h1) at (0.5* \wid,1*\hei) {};
\node[] (1h2) at (0.5* \wid,2*\hei) {};

\draw[e] (14) to[in =180, out =180] (11);
\draw[e] (12) to[in =180, out =180] (13);

\draw[e] (12) to[in =0, out =0] (11);
\draw[e] (14) to[out = 0, in = 180] (1h2);
\draw[e] (13) to[out = 0, in = 180] (1h1);

\draw[very thick] (0* \wid,-0.25*\hei) -- (0* \wid,3.25*\hei);
\end{tikzpicture}
,
\begin{tikzpicture}[x=1.5cm,y=-.5cm,baseline=-1.05cm]

\def\wid{\standardWidth}
\def\hei{\standardHeight}
\def\nodesize{3}
\def\ang{90}

\node[] (0h1) at (0.5* \wid,1*\hei) {};
\node[] (0h2) at (0.5* \wid,2*\hei) {};
\node[v, minimum size=\nodesize] (11) at (0* \wid,0*\hei) {};
\node[v, minimum size=\nodesize] (12) at (0* \wid,1*\hei) {};
\node[v, minimum size=\nodesize] (13) at (0* \wid,2*\hei) {};
\node[v, minimum size=\nodesize] (14) at (0* \wid,3*\hei) {};
\node[] (1h1) at (0.5* \wid,1*\hei) {};
\node[] (1h2) at (0.5* \wid,2*\hei) {};

\draw[e] (11) to[in =180, out =180] (14);
\draw[e] (13) to[in =180, out =180] (12);

\draw[e] (12) to[in =0, out =0] (13);
\draw[e] (11) to[out = 0, in = 180] (1h1);
\draw[e] (14) to[out = 0, in = 180] (1h2);

\draw[very thick] (0* \wid,-0.25*\hei) -- (0* \wid,3.25*\hei);
\end{tikzpicture}
,
\begin{tikzpicture}[x=1.5cm,y=-.5cm,baseline=-1.05cm]

\def\wid{\standardWidth}
\def\hei{\standardHeight}
\def\nodesize{3}
\def\ang{90}

\node[] (0h1) at (-0.5* \wid,1*\hei) {};
\node[] (0h2) at (-0.5* \wid,2*\hei) {};
\node[v, minimum size=\nodesize] (11) at (0* \wid,0*\hei) {};
\node[v, minimum size=\nodesize] (12) at (0* \wid,1*\hei) {};
\node[v, minimum size=\nodesize] (13) at (0* \wid,2*\hei) {};
\node[v, minimum size=\nodesize] (14) at (0* \wid,3*\hei) {};
\node[] (1h1) at (0.5* \wid,1*\hei) {};
\node[] (1h2) at (0.5* \wid,2*\hei) {};

\draw[e] (11) to[in =180, out =180] (12);
\draw[e] (13) to[in =180, out =180] (14);

\draw[e] (13) to[in =0, out =0] (14);
\draw[e] (11) to[out = 0, in = 180] (1h1);
\draw[e] (12) to[out = 0, in = 180] (1h2);

\draw[very thick] (0* \wid,-0.25*\hei) -- (0* \wid,3.25*\hei);
\end{tikzpicture}
,
\begin{tikzpicture}[x=1.5cm,y=-.5cm,baseline=-1.05cm]

\def\wid{\standardWidth}
\def\hei{\standardHeight}
\def\nodesize{3}
\def\ang{90}

\node[] (0h1) at (-0.5* \wid,1*\hei) {};
\node[] (0h2) at (-0.5* \wid,2*\hei) {};
\node[v, minimum size=\nodesize] (11) at (0* \wid,0*\hei) {};
\node[v, minimum size=\nodesize] (12) at (0* \wid,1*\hei) {};
\node[v, minimum size=\nodesize] (13) at (0* \wid,2*\hei) {};
\node[v, minimum size=\nodesize] (14) at (0* \wid,3*\hei) {};
\node[] (1h1) at (0.5* \wid,1*\hei) {};
\node[] (1h2) at (0.5* \wid,2*\hei) {};

\draw[e] (12) to[in =180, out =180] (11);
\draw[e] (14) to[in =180, out =180] (13);

\draw[e] (12) to[in =0, out =0] (11);
\draw[e] (14) to[out = 0, in = 180] (1h2);
\draw[e] (13) to[out = 0, in = 180] (1h1);

\draw[very thick] (0* \wid,-0.25*\hei) -- (0* \wid,3.25*\hei);
\end{tikzpicture}
\end{center}

No connections on the right, two on the left (the reflections of the row above):

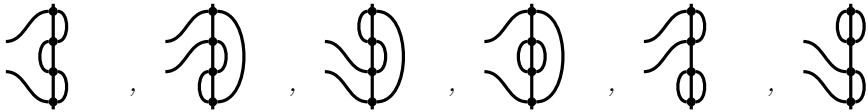
\begin{center}
    \begin{tikzpicture}[x=1.5cm,y=-.5cm,baseline=-1.05cm]

\def\wid{\standardWidth}
\def\hei{\standardHeight}
\def\nodesize{3}
\def\ang{90}

\node[] (0h1) at (-0.5* \wid,1*\hei) {};
\node[] (0h2) at (-0.5* \wid,2*\hei) {};
\node[v, minimum size=\nodesize] (11) at (0* \wid,0*\hei) {};
\node[v, minimum size=\nodesize] (12) at (0* \wid,1*\hei) {};
\node[v, minimum size=\nodesize] (13) at (0* \wid,2*\hei) {};
\node[v, minimum size=\nodesize] (14) at (0* \wid,3*\hei) {};
\node[] (1h1) at (0.5* \wid,1*\hei) {};
\node[] (1h2) at (0.5* \wid,2*\hei) {};

\draw[e] (11) to[in =0, out =0] (12);
\draw[e] (13) to[in =0, out =0] (14);

\draw[e] (12) to[in =180, out =180] (13);
\draw[e] (11) to[out = 180, in = 0] (0h1);
\draw[e] (14) to[out = 180, in = 0] (0h2);

\draw[very thick] (0* \wid,-0.25*\hei) -- (0* \wid,3.25*\hei);
\end{tikzpicture}
,
\begin{tikzpicture}[x=1.5cm,y=-.5cm,baseline=-1.05cm]

\def\wid{\standardWidth}
\def\hei{\standardHeight}
\def\nodesize{3}
\def\ang{90}

\node[] (0h1) at (-0.5* \wid,1*\hei) {};
\node[] (0h2) at (-0.5* \wid,2*\hei) {};
\node[v, minimum size=\nodesize] (11) at (0* \wid,0*\hei) {};
\node[v, minimum size=\nodesize] (12) at (0* \wid,1*\hei) {};
\node[v, minimum size=\nodesize] (13) at (0* \wid,2*\hei) {};
\node[v, minimum size=\nodesize] (14) at (0* \wid,3*\hei) {};
\node[] (1h1) at (0.5* \wid,1*\hei) {};
\node[] (1h2) at (0.5* \wid,2*\hei) {};

\draw[e] (11) to[in =0, out =0] (14);
\draw[e] (13) to[in =0, out =0] (12);

\draw[e] (13) to[in =180, out =180] (14);
\draw[e] (11) to[out = 180, in = 0] (0h1);
\draw[e] (12) to[out = 180, in = 0] (0h2);

\draw[very thick] (0* \wid,-0.25*\hei) -- (0* \wid,3.25*\hei);
\end{tikzpicture}
,
\begin{tikzpicture}[x=1.5cm,y=-.5cm,baseline=-1.05cm]

\def\wid{\standardWidth}
\def\hei{\standardHeight}
\def\nodesize{3}
\def\ang{90}

\node[] (0h1) at (-0.5* \wid,1*\hei) {};
\node[] (0h2) at (-0.5* \wid,2*\hei) {};
\node[v, minimum size=\nodesize] (11) at (0* \wid,0*\hei) {};
\node[v, minimum size=\nodesize] (12) at (0* \wid,1*\hei) {};
\node[v, minimum size=\nodesize] (13) at (0* \wid,2*\hei) {};
\node[v, minimum size=\nodesize] (14) at (0* \wid,3*\hei) {};
\node[] (1h1) at (0.5* \wid,1*\hei) {};
\node[] (1h2) at (0.5* \wid,2*\hei) {};

\draw[e] (14) to[in =0, out =0] (11);
\draw[e] (12) to[in =0, out =0] (13);

\draw[e] (12) to[in =180, out =180] (11);
\draw[e] (14) to[out = 180, in = 0] (0h2);
\draw[e] (13) to[out = 180, in = 0] (0h1);

\draw[very thick] (0* \wid,-0.25*\hei) -- (0* \wid,3.25*\hei);
\end{tikzpicture}
,
\begin{tikzpicture}[x=1.5cm,y=-.5cm,baseline=-1.05cm]

\def\wid{\standardWidth}
\def\hei{\standardHeight}
\def\nodesize{3}
\def\ang{90}

\node[] (0h1) at (-0.5* \wid,1*\hei) {};
\node[] (0h2) at (-0.5* \wid,2*\hei) {};
\node[v, minimum size=\nodesize] (11) at (0* \wid,0*\hei) {};
\node[v, minimum size=\nodesize] (12) at (0* \wid,1*\hei) {};
\node[v, minimum size=\nodesize] (13) at (0* \wid,2*\hei) {};
\node[v, minimum size=\nodesize] (14) at (0* \wid,3*\hei) {};
\node[] (1h1) at (0.5* \wid,1*\hei) {};
\node[] (1h2) at (0.5* \wid,2*\hei) {};

\draw[e] (11) to[in =0, out =0] (14);
\draw[e] (12) to[in =0, out =0] (13);

\draw[e] (12) to[in =180, out =180] (13);
\draw[e] (11) to[out = 180, in = 0] (0h1);
\draw[e] (14) to[out = 180, in = 0] (0h2);

\draw[very thick] (0* \wid,-0.25*\hei) -- (0* \wid,3.25*\hei);
\end{tikzpicture}
,
\begin{tikzpicture}[x=1.5cm,y=-.5cm,baseline=-1.05cm]

\def\wid{\standardWidth}
\def\hei{\standardHeight}
\def\nodesize{3}
\def\ang{90}

\node[] (0h1) at (-0.5* \wid,1*\hei) {};
\node[] (0h2) at (-0.5* \wid,2*\hei) {};
\node[v, minimum size=\nodesize] (11) at (0* \wid,0*\hei) {};
\node[v, minimum size=\nodesize] (12) at (0* \wid,1*\hei) {};
\node[v, minimum size=\nodesize] (13) at (0* \wid,2*\hei) {};
\node[v, minimum size=\nodesize] (14) at (0* \wid,3*\hei) {};
\node[] (1h1) at (0.5* \wid,1*\hei) {};
\node[] (1h2) at (0.5* \wid,2*\hei) {};

\draw[e] (11) to[in =0, out =0] (12);
\draw[e] (13) to[in =0, out =0] (14);

\draw[e] (13) to[in =180, out =180] (14);
\draw[e] (11) to[out = 180, in = 0] (0h1);
\draw[e] (12) to[out = 180, in = 0] (0h2);

\draw[very thick] (0* \wid,-0.25*\hei) -- (0* \wid,3.25*\hei);
\end{tikzpicture}
,
\begin{tikzpicture}[x=1.5cm,y=-.5cm,baseline=-1.05cm]

\def\wid{\standardWidth}
\def\hei{\standardHeight}
\def\nodesize{3}
\def\ang{90}

\node[] (0h1) at (-0.5* \wid,1*\hei) {};
\node[] (0h2) at (-0.5* \wid,2*\hei) {};
\node[v, minimum size=\nodesize] (11) at (0* \wid,0*\hei) {};
\node[v, minimum size=\nodesize] (12) at (0* \wid,1*\hei) {};
\node[v, minimum size=\nodesize] (13) at (0* \wid,2*\hei) {};
\node[v, minimum size=\nodesize] (14) at (0* \wid,3*\hei) {};
\node[] (1h1) at (0.5* \wid,1*\hei) {};
\node[] (1h2) at (0.5* \wid,2*\hei) {};

\draw[e] (14) to[in =0, out =0] (13);
\draw[e] (12) to[in =0, out =0] (11);

\draw[e] (12) to[in =180, out =180] (11);
\draw[e] (14) to[out = 180, in = 0] (0h2);
\draw[e] (13) to[out = 180, in = 0] (0h1);

\draw[very thick] (0* \wid,-0.25*\hei) -- (0* \wid,3.25*\hei);
\end{tikzpicture}
\end{center}

No connections each side:

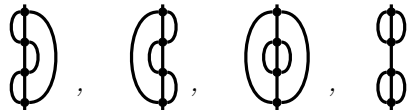
\begin{center}
\begin{tikzpicture}[x=1.5cm,y=-.5cm,baseline=-1.05cm]
\def\wid{\standardWidth}
\def\hei{\standardHeight}
\def\nodesize{3}
\def\ang{90}

\node[v, minimum size=\nodesize] (11) at (0* \wid,0*\hei) {};
\node[v, minimum size=\nodesize] (12) at (0* \wid,1*\hei) {};
\node[v, minimum size=\nodesize] (13) at (0* \wid,2*\hei) {};
\node[v, minimum size=\nodesize] (14) at (0* \wid,3*\hei) {};

\draw[e] (11) to[in =180, out =180] (12);
\draw[e] (13) to[in =180, out =180] (14);

\draw[e] (12) to[in =0, out =0] (13);
\draw[e] (11) to[out = 0, in = 0] (14);

\draw[very thick] (0* \wid,-0.25*\hei) -- (0* \wid,3.25*\hei);
\end{tikzpicture}
,
\quad
\begin{tikzpicture}[x=1.5cm,y=-.5cm,baseline=-1.05cm]
\def\wid{\standardWidth}
\def\hei{\standardHeight}
\def\nodesize{3}
\def\ang{90}

\node[v, minimum size=\nodesize] (11) at (0* \wid,0*\hei) {};
\node[v, minimum size=\nodesize] (12) at (0* \wid,1*\hei) {};
\node[v, minimum size=\nodesize] (13) at (0* \wid,2*\hei) {};
\node[v, minimum size=\nodesize] (14) at (0* \wid,3*\hei) {};

\draw[e] (11) to[in =0, out =0] (12);
\draw[e] (13) to[in =0, out =0] (14);

\draw[e] (12) to[in =180, out =180] (13);
\draw[e] (11) to[out = 180, in = 180] (14);

\draw[very thick] (0* \wid,-0.25*\hei) -- (0* \wid,3.25*\hei);
\end{tikzpicture}
,
\quad
\begin{tikzpicture}[x=1.5cm,y=-.5cm,baseline=-1.05cm]
\def\wid{\standardWidth}
\def\hei{\standardHeight}
\def\nodesize{3}
\def\ang{90}

\node[v, minimum size=\nodesize] (11) at (0* \wid,0*\hei) {};
\node[v, minimum size=\nodesize] (12) at (0* \wid,1*\hei) {};
\node[v, minimum size=\nodesize] (13) at (0* \wid,2*\hei) {};
\node[v, minimum size=\nodesize] (14) at (0* \wid,3*\hei) {};

\draw[e] (11) to[in =180, out =180] (14);
\draw[e] (13) to[in =180, out =180] (12);

\draw[e] (12) to[in =0, out =0] (13);
\draw[e] (11) to[out = 0, in = 0] (14);

\draw[very thick] (0* \wid,-0.25*\hei) -- (0* \wid,3.25*\hei);
\end{tikzpicture}
,
\quad
\begin{tikzpicture}[x=1.5cm,y=-.5cm,baseline=-1.05cm]
\def\wid{\standardWidth}
\def\hei{\standardHeight}
\def\nodesize{3}
\def\ang{90}

\node[v, minimum size=\nodesize] (11) at (0* \wid,0*\hei) {};
\node[v, minimum size=\nodesize] (12) at (0* \wid,1*\hei) {};
\node[v, minimum size=\nodesize] (13) at (0* \wid,2*\hei) {};
\node[v, minimum size=\nodesize] (14) at (0* \wid,3*\hei) {};

\draw[e] (11) to[in =180, out =180] (12);
\draw[e] (13) to[in =180, out =180] (14);

\draw[e] (11) to[in =0, out =0] (12);
\draw[e] (13) to[out = 0, in = 0] (14);

\draw[very thick] (0* \wid,-0.25*\hei) -- (0* \wid,3.25*\hei);
\end{tikzpicture}
\end{center} 
\vspace{-0.35cm}
\qed
\end{lemma}

Note that not all words in these letters correspond to allowed graffiti: consecutive letters must have the same number of connections. The point of working in the reduced bar construction $\CPLr(4)_{*,*} \subset \CPL(4)_{*,*}$ is that we do not need letters with four connections on a side. This shortens the alphabet a lot.

\subsection{Part (\ref{part:1})}

In this subsection we will prove \cref{thm:mainTechnical} (\ref{part:1}). We must show that the complex $C_*[1,0]$ of graffiti with no dividers and one loop is contractible.

\cref{lem:alphabet} tells us that each basis element of this complex in degree $p$ may be written as a word of length $p$ in the letters of that lemma. Not all of these letters can appear in a word corresponding to a one-loop graffito with no dividers. For example, any graffito containing the letter
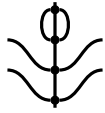
\begin{center}
\begin{tikzpicture}[x=1.5cm,y=-.5cm,baseline=-1.05cm]
\def\wid{\standardWidth}
\def\hei{\standardHeight}
\def\nodesize{3}
\def\ang{90}

\node[] (0h1) at (-0.5* \wid,1*\hei) {};
\node[] (0h2) at (-0.5* \wid,2*\hei) {};
\node[v, minimum size=\nodesize] (11) at (0* \wid,0*\hei) {};
\node[v, minimum size=\nodesize] (12) at (0* \wid,1*\hei) {};
\node[v, minimum size=\nodesize] (13) at (0* \wid,2*\hei) {};
\node[v, minimum size=\nodesize] (14) at (0* \wid,3*\hei) {};
\node[] (1h1) at (0.5* \wid,1*\hei) {};
\node[] (1h2) at (0.5* \wid,2*\hei) {};

\draw[e] (12) to[in =180, out =180] (11);
\draw[e] (13) to[in =0, out =180] (0h1);
\draw[e] (14) to[in =0, out =180] (0h2);

\draw[e] (12) to[in =0, out =0] (11);
\draw[e] (14) to[out = 0, in = 180] (1h2);
\draw[e] (13) to[out = 0, in = 180] (1h1);

\draw[very thick] (0* \wid,-0.25*\hei) -- (0* \wid,3.25*\hei);
\end{tikzpicture}
\end{center}
must have at least two loops. Precisely, we have the following.

\begin{lemma} \label{lem:oneLoopAlphabet} The elements of the graffiti basis of the complex $C_*[1,0]$ with one loop and no dividers are precisely the words in the following letters, which are a subset of those of \cref{lem:alphabet}, subject to the following rules.

Two connections each side:

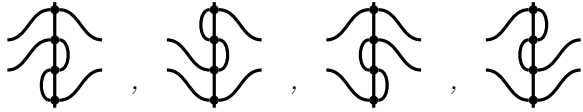
\begin{center}
\begin{tikzpicture}[x=1.5cm,y=-.5cm,baseline=-1.05cm]

\def\wid{\standardWidth}
\def\hei{\standardHeight}
\def\nodesize{3}
\def\ang{90}

\node[] (0h1) at (-0.5* \wid,1*\hei) {};
\node[] (0h2) at (-0.5* \wid,2*\hei) {};
\node[v, minimum size=\nodesize] (11) at (0* \wid,0*\hei) {};
\node[v, minimum size=\nodesize] (12) at (0* \wid,1*\hei) {};
\node[v, minimum size=\nodesize] (13) at (0* \wid,2*\hei) {};
\node[v, minimum size=\nodesize] (14) at (0* \wid,3*\hei) {};
\node[] (1h1) at (0.5* \wid,1*\hei) {};
\node[] (1h2) at (0.5* \wid,2*\hei) {};

\draw[e] (11) to[in =0, out =180] (0h1);
\draw[e] (12) to[in =0, out =180] (0h2);
\draw[e] (13) to[in =180, out =180] (14);

\draw[e] (12) to[in =0, out =0] (13);
\draw[e] (11) to[out = 0, in = 180] (1h1);
\draw[e] (14) to[out = 0, in = 180] (1h2);

\draw[very thick] (0* \wid,-0.25*\hei) -- (0* \wid,3.25*\hei);
\end{tikzpicture}
,
\begin{tikzpicture}[x=1.5cm,y=-.5cm,baseline=-1.05cm]

\def\wid{\standardWidth}
\def\hei{\standardHeight}
\def\nodesize{3}
\def\ang{90}

\node[] (0h1) at (-0.5* \wid,1*\hei) {};
\node[] (0h2) at (-0.5* \wid,2*\hei) {};
\node[v, minimum size=\nodesize] (11) at (0* \wid,0*\hei) {};
\node[v, minimum size=\nodesize] (12) at (0* \wid,1*\hei) {};
\node[v, minimum size=\nodesize] (13) at (0* \wid,2*\hei) {};
\node[v, minimum size=\nodesize] (14) at (0* \wid,3*\hei) {};
\node[] (1h1) at (0.5* \wid,1*\hei) {};
\node[] (1h2) at (0.5* \wid,2*\hei) {};

\draw[e] (13) to[in =00, out =180] (0h1);
\draw[e] (14) to[in =0, out =180] (0h2);
\draw[e] (12) to[in =180, out =180] (11);

\draw[e] (12) to[in =0, out =0] (13);
\draw[e] (11) to[out = 0, in = 180] (1h1);
\draw[e] (14) to[out = 0, in = 180] (1h2);

\draw[very thick] (0* \wid,-0.25*\hei) -- (0* \wid,3.25*\hei);
\end{tikzpicture}
,
\begin{tikzpicture}[x=1.5cm,y=-.5cm,baseline=-1.05cm]

\def\wid{\standardWidth}
\def\hei{\standardHeight}
\def\nodesize{3}
\def\ang{90}

\node[] (0h1) at (-0.5* \wid,1*\hei) {};
\node[] (0h2) at (-0.5* \wid,2*\hei) {};
\node[v, minimum size=\nodesize] (11) at (0* \wid,0*\hei) {};
\node[v, minimum size=\nodesize] (12) at (0* \wid,1*\hei) {};
\node[v, minimum size=\nodesize] (13) at (0* \wid,2*\hei) {};
\node[v, minimum size=\nodesize] (14) at (0* \wid,3*\hei) {};
\node[] (1h1) at (0.5* \wid,1*\hei) {};
\node[] (1h2) at (0.5* \wid,2*\hei) {};

\draw[e] (11) to[in =0, out =180] (0h1);
\draw[e] (14) to[in =0, out =180] (0h2);
\draw[e] (13) to[in =180, out =180] (12);

\draw[e] (13) to[in =0, out =0] (14);
\draw[e] (11) to[out = 0, in = 180] (1h1);
\draw[e] (12) to[out = 0, in = 180] (1h2);

\draw[very thick] (0* \wid,-0.25*\hei) -- (0* \wid,3.25*\hei);
\end{tikzpicture}
,
\begin{tikzpicture}[x=1.5cm,y=-.5cm,baseline=-1.05cm]

\def\wid{\standardWidth}
\def\hei{\standardHeight}
\def\nodesize{3}
\def\ang{90}

\node[] (0h1) at (-0.5* \wid,1*\hei) {};
\node[] (0h2) at (-0.5* \wid,2*\hei) {};
\node[v, minimum size=\nodesize] (11) at (0* \wid,0*\hei) {};
\node[v, minimum size=\nodesize] (12) at (0* \wid,1*\hei) {};
\node[v, minimum size=\nodesize] (13) at (0* \wid,2*\hei) {};
\node[v, minimum size=\nodesize] (14) at (0* \wid,3*\hei) {};
\node[] (1h1) at (0.5* \wid,1*\hei) {};
\node[] (1h2) at (0.5* \wid,2*\hei) {};

\draw[e] (11) to[in =0, out =180] (0h1);
\draw[e] (14) to[in =0, out =180] (0h2);
\draw[e] (12) to[in =180, out =180] (13);

\draw[e] (12) to[in =0, out =0] (11);
\draw[e] (14) to[out = 0, in = 180] (1h2);
\draw[e] (13) to[out = 0, in = 180] (1h1);

\draw[very thick] (0* \wid,-0.25*\hei) -- (0* \wid,3.25*\hei);
\end{tikzpicture}
\end{center}

No connections on the left, two on the right:

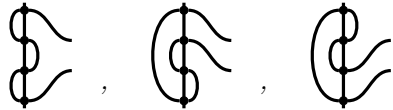
\begin{center}
\begin{tikzpicture}[x=1.5cm,y=-.5cm,baseline=-1.05cm]

\def\wid{\standardWidth}
\def\hei{\standardHeight}
\def\nodesize{3}
\def\ang{90}

\node[] (0h1) at (0.5* \wid,1*\hei) {};
\node[] (0h2) at (0.5* \wid,2*\hei) {};
\node[v, minimum size=\nodesize] (11) at (0* \wid,0*\hei) {};
\node[v, minimum size=\nodesize] (12) at (0* \wid,1*\hei) {};
\node[v, minimum size=\nodesize] (13) at (0* \wid,2*\hei) {};
\node[v, minimum size=\nodesize] (14) at (0* \wid,3*\hei) {};
\node[] (1h1) at (0.5* \wid,1*\hei) {};
\node[] (1h2) at (0.5* \wid,2*\hei) {};

\draw[e] (11) to[in =180, out =180] (12);
\draw[e] (13) to[in =180, out =180] (14);

\draw[e] (12) to[in =0, out =0] (13);
\draw[e] (11) to[out = 0, in = 180] (1h1);
\draw[e] (14) to[out = 0, in = 180] (1h2);

\draw[very thick] (0* \wid,-0.25*\hei) -- (0* \wid,3.25*\hei);
\end{tikzpicture}
,
\begin{tikzpicture}[x=1.5cm,y=-.5cm,baseline=-1.05cm]

\def\wid{\standardWidth}
\def\hei{\standardHeight}
\def\nodesize{3}
\def\ang{90}

\node[] (0h1) at (-0.5* \wid,1*\hei) {};
\node[] (0h2) at (-0.5* \wid,2*\hei) {};
\node[v, minimum size=\nodesize] (11) at (0* \wid,0*\hei) {};
\node[v, minimum size=\nodesize] (12) at (0* \wid,1*\hei) {};
\node[v, minimum size=\nodesize] (13) at (0* \wid,2*\hei) {};
\node[v, minimum size=\nodesize] (14) at (0* \wid,3*\hei) {};
\node[] (1h1) at (0.5* \wid,1*\hei) {};
\node[] (1h2) at (0.5* \wid,2*\hei) {};

\draw[e] (11) to[in =180, out =180] (14);
\draw[e] (13) to[in =180, out =180] (12);

\draw[e] (13) to[in =0, out =0] (14);
\draw[e] (11) to[out = 0, in = 180] (1h1);
\draw[e] (12) to[out = 0, in = 180] (1h2);

\draw[very thick] (0* \wid,-0.25*\hei) -- (0* \wid,3.25*\hei);
\end{tikzpicture}
,
\begin{tikzpicture}[x=1.5cm,y=-.5cm,baseline=-1.05cm]

\def\wid{\standardWidth}
\def\hei{\standardHeight}
\def\nodesize{3}
\def\ang{90}

\node[] (0h1) at (-0.5* \wid,1*\hei) {};
\node[] (0h2) at (-0.5* \wid,2*\hei) {};
\node[v, minimum size=\nodesize] (11) at (0* \wid,0*\hei) {};
\node[v, minimum size=\nodesize] (12) at (0* \wid,1*\hei) {};
\node[v, minimum size=\nodesize] (13) at (0* \wid,2*\hei) {};
\node[v, minimum size=\nodesize] (14) at (0* \wid,3*\hei) {};
\node[] (1h1) at (0.5* \wid,1*\hei) {};
\node[] (1h2) at (0.5* \wid,2*\hei) {};

\draw[e] (14) to[in =180, out =180] (11);
\draw[e] (12) to[in =180, out =180] (13);

\draw[e] (12) to[in =0, out =0] (11);
\draw[e] (14) to[out = 0, in = 180] (1h2);
\draw[e] (13) to[out = 0, in = 180] (1h1);

\draw[very thick] (0* \wid,-0.25*\hei) -- (0* \wid,3.25*\hei);
\end{tikzpicture}
\end{center}

No connections on the right, two on the left (the reflections of the row above):

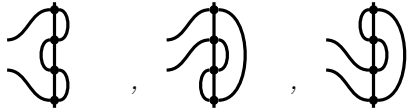
\begin{center}
    \begin{tikzpicture}[x=1.5cm,y=-.5cm,baseline=-1.05cm]

\def\wid{\standardWidth}
\def\hei{\standardHeight}
\def\nodesize{3}
\def\ang{90}

\node[] (0h1) at (-0.5* \wid,1*\hei) {};
\node[] (0h2) at (-0.5* \wid,2*\hei) {};
\node[v, minimum size=\nodesize] (11) at (0* \wid,0*\hei) {};
\node[v, minimum size=\nodesize] (12) at (0* \wid,1*\hei) {};
\node[v, minimum size=\nodesize] (13) at (0* \wid,2*\hei) {};
\node[v, minimum size=\nodesize] (14) at (0* \wid,3*\hei) {};
\node[] (1h1) at (0.5* \wid,1*\hei) {};
\node[] (1h2) at (0.5* \wid,2*\hei) {};

\draw[e] (11) to[in =0, out =0] (12);
\draw[e] (13) to[in =0, out =0] (14);

\draw[e] (12) to[in =180, out =180] (13);
\draw[e] (11) to[out = 180, in = 0] (0h1);
\draw[e] (14) to[out = 180, in = 0] (0h2);

\draw[very thick] (0* \wid,-0.25*\hei) -- (0* \wid,3.25*\hei);
\end{tikzpicture}
,
\begin{tikzpicture}[x=1.5cm,y=-.5cm,baseline=-1.05cm]

\def\wid{\standardWidth}
\def\hei{\standardHeight}
\def\nodesize{3}
\def\ang{90}

\node[] (0h1) at (-0.5* \wid,1*\hei) {};
\node[] (0h2) at (-0.5* \wid,2*\hei) {};
\node[v, minimum size=\nodesize] (11) at (0* \wid,0*\hei) {};
\node[v, minimum size=\nodesize] (12) at (0* \wid,1*\hei) {};
\node[v, minimum size=\nodesize] (13) at (0* \wid,2*\hei) {};
\node[v, minimum size=\nodesize] (14) at (0* \wid,3*\hei) {};
\node[] (1h1) at (0.5* \wid,1*\hei) {};
\node[] (1h2) at (0.5* \wid,2*\hei) {};

\draw[e] (11) to[in =0, out =0] (14);
\draw[e] (13) to[in =0, out =0] (12);

\draw[e] (13) to[in =180, out =180] (14);
\draw[e] (11) to[out = 180, in = 0] (0h1);
\draw[e] (12) to[out = 180, in = 0] (0h2);

\draw[very thick] (0* \wid,-0.25*\hei) -- (0* \wid,3.25*\hei);
\end{tikzpicture}
,
\begin{tikzpicture}[x=1.5cm,y=-.5cm,baseline=-1.05cm]

\def\wid{\standardWidth}
\def\hei{\standardHeight}
\def\nodesize{3}
\def\ang{90}

\node[] (0h1) at (-0.5* \wid,1*\hei) {};
\node[] (0h2) at (-0.5* \wid,2*\hei) {};
\node[v, minimum size=\nodesize] (11) at (0* \wid,0*\hei) {};
\node[v, minimum size=\nodesize] (12) at (0* \wid,1*\hei) {};
\node[v, minimum size=\nodesize] (13) at (0* \wid,2*\hei) {};
\node[v, minimum size=\nodesize] (14) at (0* \wid,3*\hei) {};
\node[] (1h1) at (0.5* \wid,1*\hei) {};
\node[] (1h2) at (0.5* \wid,2*\hei) {};

\draw[e] (14) to[in =0, out =0] (11);
\draw[e] (12) to[in =0, out =0] (13);

\draw[e] (12) to[in =180, out =180] (11);
\draw[e] (14) to[out = 180, in = 0] (0h2);
\draw[e] (13) to[out = 180, in = 0] (0h1);

\draw[very thick] (0* \wid,-0.25*\hei) -- (0* \wid,3.25*\hei);
\end{tikzpicture}
\end{center}

No connections each side (these form a basis for degree 0, and cannot appear in other degrees)

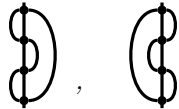
\begin{center}
\begin{tikzpicture}[x=1.5cm,y=-.5cm,baseline=-1.05cm]
\def\wid{\standardWidth}
\def\hei{\standardHeight}
\def\nodesize{3}
\def\ang{90}

\node[v, minimum size=\nodesize] (11) at (0* \wid,0*\hei) {};
\node[v, minimum size=\nodesize] (12) at (0* \wid,1*\hei) {};
\node[v, minimum size=\nodesize] (13) at (0* \wid,2*\hei) {};
\node[v, minimum size=\nodesize] (14) at (0* \wid,3*\hei) {};

\draw[e] (11) to[in =180, out =180] (12);
\draw[e] (13) to[in =180, out =180] (14);

\draw[e] (12) to[in =0, out =0] (13);
\draw[e] (11) to[out = 0, in = 0] (14);

\draw[very thick] (0* \wid,-0.25*\hei) -- (0* \wid,3.25*\hei);
\end{tikzpicture}
,
\quad
\begin{tikzpicture}[x=1.5cm,y=-.5cm,baseline=-1.05cm]
\def\wid{\standardWidth}
\def\hei{\standardHeight}
\def\nodesize{3}
\def\ang{90}

\node[v, minimum size=\nodesize] (11) at (0* \wid,0*\hei) {};
\node[v, minimum size=\nodesize] (12) at (0* \wid,1*\hei) {};
\node[v, minimum size=\nodesize] (13) at (0* \wid,2*\hei) {};
\node[v, minimum size=\nodesize] (14) at (0* \wid,3*\hei) {};

\draw[e] (11) to[in =0, out =0] (12);
\draw[e] (13) to[in =0, out =0] (14);

\draw[e] (12) to[in =180, out =180] (13);
\draw[e] (11) to[out = 180, in = 180] (14);

\draw[very thick] (0* \wid,-0.25*\hei) -- (0* \wid,3.25*\hei);
\end{tikzpicture}
\end{center}
The basis consists of any word in these letters satisfying the following rules:
\begin{itemize}
    \item the first letter of a word of length at least 2 must have no connections on the left, and two on the right,
    \item and `dually', the last letter of a word of length at least 2 must have two connections on the left, and none on the right.
    \item other letters of a word of length at least 2 must have two connections on each side, and
    \item words of length 1 consists of a letter with no connections on either side.
\end{itemize}

\end{lemma}

\begin{proof} Case-by-case.
\end{proof}

\begin{remark} This complex fails to be like the one from \cref{lem:basicObs} in two ways. Firstly, there are restrictions on which letter can go where. Secondly, the differential is not quite given by letter deletion, as follows. Deleting a bar other than the first or last \emph{does} have the effect of deleting the letter living on that bar:
\begin{center}
$d_1($
\begin{tikzpicture}[x=1.5cm,y=-.5cm,baseline=-0.7cm]
\def\wid{\standardWidth}
\def\hei{\standardHeight}
\def\nodesize{3}
\def\ang{90}

\node[v, minimum size=\nodesize] (11) at (0* \wid,0*\hei) {};
\node[v, minimum size=\nodesize] (12) at (0* \wid,1*\hei) {};
\node[v, minimum size=\nodesize] (13) at (0* \wid,2*\hei) {};
\node[v, minimum size=\nodesize] (14) at (0* \wid,3*\hei) {};

\node[] (1h1) at (0.5* \wid,1*\hei) {};
\node[] (1h2) at (0.5* \wid,2*\hei) {};

\draw[e] (11) to[in =180, out =180] (14);
\draw[e] (13) to[in =180, out =180] (12);

\node[v, minimum size=\nodesize] (21) at (1* \wid,0*\hei) {};
\node[v, minimum size=\nodesize] (22) at (1* \wid,1*\hei) {};
\node[v, minimum size=\nodesize] (23) at (1* \wid,2*\hei) {};
\node[v, minimum size=\nodesize] (24) at (1* \wid,3*\hei) {};

\node[] (2h1) at (1.5* \wid,1*\hei) {};
\node[] (2h2) at (1.5* \wid,2*\hei) {};

\draw[e] (11) to[in =0, out =0] (12);
\draw[e] (13) to[out = 0, in = 180] (1h1);
\draw[e] (1h1) to[out = 0, in = 180] (23);
\draw[e] (14) to[out = 0, in = 180] (1h2);
\draw[e] (1h2) to[out = 0, in = 180] (24);
\draw[e] (21) to[out = 180, in = 180] (22);

\node[v, minimum size=\nodesize] (31) at (2* \wid,0*\hei) {};
\node[v, minimum size=\nodesize] (32) at (2* \wid,1*\hei) {};
\node[v, minimum size=\nodesize] (33) at (2* \wid,2*\hei) {};
\node[v, minimum size=\nodesize] (34) at (2* \wid,3*\hei) {};

\draw[e] (23) to[in =0, out =0] (22);
\draw[e] (21) to[out = 0, in = 180] (2h1);
\draw[e] (2h1) to[out = 0, in = 180] (31);
\draw[e] (24) to[out = 0, in = 180] (2h2);
\draw[e] (2h2) to[out = 0, in = 180] (32);
\draw[e] (33) to[out = 180, in = 180] (34);

\draw[e] (31) to[in =0, out =0] (34);
\draw[e] (32) to[in =0, out =0] (33);

\draw[very thick] (0* \wid,-0.25*\hei) -- (0* \wid,3.25*\hei);
\draw[very thick] (1* \wid,-0.25*\hei) -- (1* \wid,3.25*\hei);
\draw[very thick] (2* \wid,-0.25*\hei) -- (2* \wid,3.25*\hei);
\end{tikzpicture}
$)$
\quad
$=$
\quad
\begin{tikzpicture}[x=1.5cm,y=-.5cm,baseline=-0.7cm]
\def\wid{\standardWidth}
\def\hei{\standardHeight}
\def\nodesize{3}
\def\ang{90}

\node[v, minimum size=\nodesize] (11) at (0* \wid,0*\hei) {};
\node[v, minimum size=\nodesize] (12) at (0* \wid,1*\hei) {};
\node[v, minimum size=\nodesize] (13) at (0* \wid,2*\hei) {};
\node[v, minimum size=\nodesize] (14) at (0* \wid,3*\hei) {};

\node[] (1h1) at (0.5* \wid,1*\hei) {};
\node[] (1h2) at (0.5* \wid,2*\hei) {};

\draw[e] (11) to[in =180, out =180] (14);
\draw[e] (13) to[in =180, out =180] (12);

\draw[e] (11) to[in =0, out =0] (12);
\draw[e] (13) to[out = 0, in = 180] (1h1);
\draw[e] (14) to[out = 0, in = 180] (1h2);

\node[v, minimum size=\nodesize] (31) at (1* \wid,0*\hei) {};
\node[v, minimum size=\nodesize] (32) at (1* \wid,1*\hei) {};
\node[v, minimum size=\nodesize] (33) at (1* \wid,2*\hei) {};
\node[v, minimum size=\nodesize] (34) at (1* \wid,3*\hei) {};

\draw[e] (1h1) to[out = 0, in = 180] (31);
\draw[e] (1h2) to[out = 0, in = 180] (32);
\draw[e] (33) to[out = 180, in = 180] (34);

\draw[e] (31) to[in =0, out =0] (34);
\draw[e] (32) to[in =0, out =0] (33);

\draw[very thick] (0* \wid,-0.25*\hei) -- (0* \wid,3.25*\hei);
\draw[very thick] (1* \wid,-0.25*\hei) -- (1* \wid,3.25*\hei);
\end{tikzpicture}
,
\end{center}
but deleting an endmost bar changes the neighbouring letter, by joining up the two hanging edges:
\begin{center}
$d_0($
\begin{tikzpicture}[x=1.5cm,y=-.5cm,baseline=-0.7cm]
\def\wid{\standardWidth}
\def\hei{\standardHeight}
\def\nodesize{3}
\def\ang{90}

\node[v, minimum size=\nodesize] (11) at (0* \wid,0*\hei) {};
\node[v, minimum size=\nodesize] (12) at (0* \wid,1*\hei) {};
\node[v, minimum size=\nodesize] (13) at (0* \wid,2*\hei) {};
\node[v, minimum size=\nodesize] (14) at (0* \wid,3*\hei) {};

\node[] (1h1) at (0.5* \wid,1*\hei) {};
\node[] (1h2) at (0.5* \wid,2*\hei) {};

\draw[e] (11) to[in =180, out =180] (14);
\draw[e] (13) to[in =180, out =180] (12);

\node[v, minimum size=\nodesize] (21) at (1* \wid,0*\hei) {};
\node[v, minimum size=\nodesize] (22) at (1* \wid,1*\hei) {};
\node[v, minimum size=\nodesize] (23) at (1* \wid,2*\hei) {};
\node[v, minimum size=\nodesize] (24) at (1* \wid,3*\hei) {};

\node[] (2h1) at (1.5* \wid,1*\hei) {};
\node[] (2h2) at (1.5* \wid,2*\hei) {};

\draw[e] (11) to[in =0, out =0] (12);
\draw[e] (13) to[out = 0, in = 180] (1h1);
\draw[e] (1h1) to[out = 0, in = 180] (23);
\draw[e] (14) to[out = 0, in = 180] (1h2);
\draw[e] (1h2) to[out = 0, in = 180] (24);
\draw[e] (21) to[out = 180, in = 180] (22);

\node[v, minimum size=\nodesize] (31) at (2* \wid,0*\hei) {};
\node[v, minimum size=\nodesize] (32) at (2* \wid,1*\hei) {};
\node[v, minimum size=\nodesize] (33) at (2* \wid,2*\hei) {};
\node[v, minimum size=\nodesize] (34) at (2* \wid,3*\hei) {};

\draw[e] (23) to[in =0, out =0] (22);
\draw[e] (21) to[out = 0, in = 180] (2h1);
\draw[e] (2h1) to[out = 0, in = 180] (31);
\draw[e] (24) to[out = 0, in = 180] (2h2);
\draw[e] (2h2) to[out = 0, in = 180] (32);
\draw[e] (33) to[out = 180, in = 180] (34);

\draw[e] (31) to[in =0, out =0] (34);
\draw[e] (32) to[in =0, out =0] (33);

\draw[very thick] (0* \wid,-0.25*\hei) -- (0* \wid,3.25*\hei);
\draw[very thick] (1* \wid,-0.25*\hei) -- (1* \wid,3.25*\hei);
\draw[very thick] (2* \wid,-0.25*\hei) -- (2* \wid,3.25*\hei);
\end{tikzpicture}
$)$
\quad
$=$
\quad
\begin{tikzpicture}[x=1.5cm,y=-.5cm,baseline=-0.7cm]
\def\wid{\standardWidth}
\def\hei{\standardHeight}
\def\nodesize{3}
\def\ang{90}

\node[v, minimum size=\nodesize] (11) at (0* \wid,0*\hei) {};
\node[v, minimum size=\nodesize] (12) at (0* \wid,1*\hei) {};
\node[v, minimum size=\nodesize] (13) at (0* \wid,2*\hei) {};
\node[v, minimum size=\nodesize] (14) at (0* \wid,3*\hei) {};

\node[] (1h1) at (0.5* \wid,1*\hei) {};
\node[] (1h2) at (0.5* \wid,2*\hei) {};

\draw[e] (11) to[in =180, out =180] (12);
\draw[e] (13) to[in =180, out =180] (14);

\draw[e] (12) to[in =0, out =0] (13);
\draw[e] (11) to[out = 0, in = 180] (1h1);
\draw[e] (14) to[out = 0, in = 180] (1h2);

\node[v, minimum size=\nodesize] (31) at (1* \wid,0*\hei) {};
\node[v, minimum size=\nodesize] (32) at (1* \wid,1*\hei) {};
\node[v, minimum size=\nodesize] (33) at (1* \wid,2*\hei) {};
\node[v, minimum size=\nodesize] (34) at (1* \wid,3*\hei) {};

\draw[e] (1h1) to[out = 0, in = 180] (31);
\draw[e] (1h2) to[out = 0, in = 180] (32);
\draw[e] (33) to[out = 180, in = 180] (34);

\draw[e] (31) to[in =0, out =0] (34);
\draw[e] (32) to[in =0, out =0] (33);

\draw[very thick] (0* \wid,-0.25*\hei) -- (0* \wid,3.25*\hei);
\draw[very thick] (1* \wid,-0.25*\hei) -- (1* \wid,3.25*\hei);
\end{tikzpicture}
.
\end{center}
\end{remark}

On the other hand, for the complex $C_*^{LR}[1,0]$ of two-sided open graffiti which close up to graffiti with no loops and no dividers (\cref{def:complexCij}), we have the following, which again follows from \cref{lem:alphabet} by checking case-by-case.

\begin{lemma} \label{lem:openEnds} The complex $C_*^{LR}[1,0]$ is isomorphic to the complex of words in the four letters 
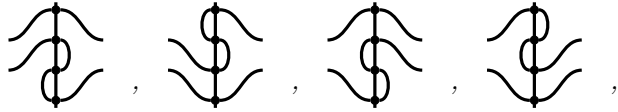
\begin{center}
\begin{tikzpicture}[x=1.5cm,y=-.5cm,baseline=-1.05cm]

\def\wid{\standardWidth}
\def\hei{\standardHeight}
\def\nodesize{3}
\def\ang{90}

\node[] (0h1) at (-0.5* \wid,1*\hei) {};
\node[] (0h2) at (-0.5* \wid,2*\hei) {};
\node[v, minimum size=\nodesize] (11) at (0* \wid,0*\hei) {};
\node[v, minimum size=\nodesize] (12) at (0* \wid,1*\hei) {};
\node[v, minimum size=\nodesize] (13) at (0* \wid,2*\hei) {};
\node[v, minimum size=\nodesize] (14) at (0* \wid,3*\hei) {};
\node[] (1h1) at (0.5* \wid,1*\hei) {};
\node[] (1h2) at (0.5* \wid,2*\hei) {};

\draw[e] (11) to[in =0, out =180] (0h1);
\draw[e] (12) to[in =0, out =180] (0h2);
\draw[e] (13) to[in =180, out =180] (14);

\draw[e] (12) to[in =0, out =0] (13);
\draw[e] (11) to[out = 0, in = 180] (1h1);
\draw[e] (14) to[out = 0, in = 180] (1h2);

\draw[very thick] (0* \wid,-0.25*\hei) -- (0* \wid,3.25*\hei);
\end{tikzpicture}
,
\begin{tikzpicture}[x=1.5cm,y=-.5cm,baseline=-1.05cm]

\def\wid{\standardWidth}
\def\hei{\standardHeight}
\def\nodesize{3}
\def\ang{90}

\node[] (0h1) at (-0.5* \wid,1*\hei) {};
\node[] (0h2) at (-0.5* \wid,2*\hei) {};
\node[v, minimum size=\nodesize] (11) at (0* \wid,0*\hei) {};
\node[v, minimum size=\nodesize] (12) at (0* \wid,1*\hei) {};
\node[v, minimum size=\nodesize] (13) at (0* \wid,2*\hei) {};
\node[v, minimum size=\nodesize] (14) at (0* \wid,3*\hei) {};
\node[] (1h1) at (0.5* \wid,1*\hei) {};
\node[] (1h2) at (0.5* \wid,2*\hei) {};

\draw[e] (13) to[in =00, out =180] (0h1);
\draw[e] (14) to[in =0, out =180] (0h2);
\draw[e] (12) to[in =180, out =180] (11);

\draw[e] (12) to[in =0, out =0] (13);
\draw[e] (11) to[out = 0, in = 180] (1h1);
\draw[e] (14) to[out = 0, in = 180] (1h2);

\draw[very thick] (0* \wid,-0.25*\hei) -- (0* \wid,3.25*\hei);
\end{tikzpicture}
,
\begin{tikzpicture}[x=1.5cm,y=-.5cm,baseline=-1.05cm]

\def\wid{\standardWidth}
\def\hei{\standardHeight}
\def\nodesize{3}
\def\ang{90}

\node[] (0h1) at (-0.5* \wid,1*\hei) {};
\node[] (0h2) at (-0.5* \wid,2*\hei) {};
\node[v, minimum size=\nodesize] (11) at (0* \wid,0*\hei) {};
\node[v, minimum size=\nodesize] (12) at (0* \wid,1*\hei) {};
\node[v, minimum size=\nodesize] (13) at (0* \wid,2*\hei) {};
\node[v, minimum size=\nodesize] (14) at (0* \wid,3*\hei) {};
\node[] (1h1) at (0.5* \wid,1*\hei) {};
\node[] (1h2) at (0.5* \wid,2*\hei) {};

\draw[e] (11) to[in =0, out =180] (0h1);
\draw[e] (14) to[in =0, out =180] (0h2);
\draw[e] (13) to[in =180, out =180] (12);

\draw[e] (13) to[in =0, out =0] (14);
\draw[e] (11) to[out = 0, in = 180] (1h1);
\draw[e] (12) to[out = 0, in = 180] (1h2);

\draw[very thick] (0* \wid,-0.25*\hei) -- (0* \wid,3.25*\hei);
\end{tikzpicture}
,
\begin{tikzpicture}[x=1.5cm,y=-.5cm,baseline=-1.05cm]

\def\wid{\standardWidth}
\def\hei{\standardHeight}
\def\nodesize{3}
\def\ang{90}

\node[] (0h1) at (-0.5* \wid,1*\hei) {};
\node[] (0h2) at (-0.5* \wid,2*\hei) {};
\node[v, minimum size=\nodesize] (11) at (0* \wid,0*\hei) {};
\node[v, minimum size=\nodesize] (12) at (0* \wid,1*\hei) {};
\node[v, minimum size=\nodesize] (13) at (0* \wid,2*\hei) {};
\node[v, minimum size=\nodesize] (14) at (0* \wid,3*\hei) {};
\node[] (1h1) at (0.5* \wid,1*\hei) {};
\node[] (1h2) at (0.5* \wid,2*\hei) {};

\draw[e] (11) to[in =0, out =180] (0h1);
\draw[e] (14) to[in =0, out =180] (0h2);
\draw[e] (12) to[in =180, out =180] (13);

\draw[e] (12) to[in =0, out =0] (11);
\draw[e] (14) to[out = 0, in = 180] (1h2);
\draw[e] (13) to[out = 0, in = 180] (1h1);

\draw[very thick] (0* \wid,-0.25*\hei) -- (0* \wid,3.25*\hei);
\end{tikzpicture}
,
\end{center}
so by \cref{lem:basicObs} its homology is generated by any single letter. \qed
\end{lemma}

This information about the two-sided open complex will allow us to understand the closed complex $C_*[1,0]$, using the short exact sequences $$0 \to K \to S(4,2) \to S(4,0) \to 0 \textrm{ and } 0 \to K^\vee \to S^\vee(2,4) \to S^\vee(0,4) \to 0$$ of \cref{rmk:SES}. Some preparation is necessary.

\begin{lemma} \label{lem:SESKernelID} As an $R$-module, the kernel $K^\vee$ is free of rank 1 on generator
\begin{center}
\begin{tikzpicture}[x=1.5cm,y=-.5cm,baseline=-0.7cm]

\def\wid{\standardWidth}
\def\hei{\standardHeight}
\def\nodesize{3}
\def\ang{90}

\node[] (0h1) at (-0.5* \wid,1*\hei) {};
\node[] (0h2) at (-0.5* \wid,2*\hei) {};
\node[v, minimum size=\nodesize] (11) at (0* \wid,0*\hei) {};
\node[v, minimum size=\nodesize] (12) at (0* \wid,1*\hei) {};
\node[v, minimum size=\nodesize] (13) at (0* \wid,2*\hei) {};
\node[v, minimum size=\nodesize] (14) at (0* \wid,3*\hei) {};

\draw[e] (11) to[in =0, out =180] (0h1);
\draw[e] (12) to[in =0, out =180] (0h2);
\draw[e] (13) to[in =180, out =180] (14);

\draw[very thick] (0* \wid,-0.25*\hei) -- (0* \wid,3.25*\hei);
\end{tikzpicture}
\quad
$-$
\begin{tikzpicture}[x=1.5cm,y=-.5cm,baseline=-0.7cm]

\def\wid{\standardWidth}
\def\hei{\standardHeight}
\def\nodesize{3}
\def\ang{90}

\node[] (0h1) at (-0.5* \wid,1*\hei) {};
\node[] (0h2) at (-0.5* \wid,2*\hei) {};
\node[v, minimum size=\nodesize] (11) at (0* \wid,0*\hei) {};
\node[v, minimum size=\nodesize] (12) at (0* \wid,1*\hei) {};
\node[v, minimum size=\nodesize] (13) at (0* \wid,2*\hei) {};
\node[v, minimum size=\nodesize] (14) at (0* \wid,3*\hei) {};

\draw[e] (13) to[in =00, out =180] (0h1);
\draw[e] (14) to[in =0, out =180] (0h2);
\draw[e] (12) to[in =180, out =180] (11);

\draw[very thick] (0* \wid,-0.25*\hei) -- (0* \wid,3.25*\hei);
\end{tikzpicture}
$\in S^\vee(2,4)$.
\end{center}
\end{lemma}

\begin{proof} The right $\TL_4$-module $S^\vee(2,4)$ is as an $R$-module free on basis
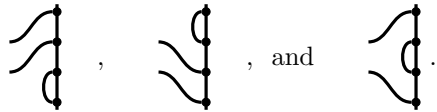
\begin{center}
\begin{tikzpicture}[x=1.5cm,y=-.5cm,baseline=-0.7cm]

\def\wid{\standardWidth}
\def\hei{\standardHeight}
\def\nodesize{3}
\def\ang{90}

\node[] (0h1) at (-0.5* \wid,1*\hei) {};
\node[] (0h2) at (-0.5* \wid,2*\hei) {};
\node[v, minimum size=\nodesize] (11) at (0* \wid,0*\hei) {};
\node[v, minimum size=\nodesize] (12) at (0* \wid,1*\hei) {};
\node[v, minimum size=\nodesize] (13) at (0* \wid,2*\hei) {};
\node[v, minimum size=\nodesize] (14) at (0* \wid,3*\hei) {};

\draw[e] (11) to[in =0, out =180] (0h1);
\draw[e] (12) to[in =0, out =180] (0h2);
\draw[e] (13) to[in =180, out =180] (14);

\draw[very thick] (0* \wid,-0.25*\hei) -- (0* \wid,3.25*\hei);
\end{tikzpicture}
\quad
, \quad 
\begin{tikzpicture}[x=1.5cm,y=-.5cm,baseline=-0.7cm]

\def\wid{\standardWidth}
\def\hei{\standardHeight}
\def\nodesize{3}
\def\ang{90}

\node[] (0h1) at (-0.5* \wid,1*\hei) {};
\node[] (0h2) at (-0.5* \wid,2*\hei) {};
\node[v, minimum size=\nodesize] (11) at (0* \wid,0*\hei) {};
\node[v, minimum size=\nodesize] (12) at (0* \wid,1*\hei) {};
\node[v, minimum size=\nodesize] (13) at (0* \wid,2*\hei) {};
\node[v, minimum size=\nodesize] (14) at (0* \wid,3*\hei) {};

\draw[e] (13) to[in =00, out =180] (0h1);
\draw[e] (14) to[in =0, out =180] (0h2);
\draw[e] (12) to[in =180, out =180] (11);

\draw[very thick] (0* \wid,-0.25*\hei) -- (0* \wid,3.25*\hei);
\end{tikzpicture}
\quad
, \textrm{ and}
\quad
\begin{tikzpicture}[x=1.5cm,y=-.5cm,baseline=-0.7cm]

\def\wid{\standardWidth}
\def\hei{\standardHeight}
\def\nodesize{3}
\def\ang{90}

\node[] (0h1) at (-0.5* \wid,1*\hei) {};
\node[] (0h2) at (-0.5* \wid,2*\hei) {};
\node[v, minimum size=\nodesize] (11) at (0* \wid,0*\hei) {};
\node[v, minimum size=\nodesize] (12) at (0* \wid,1*\hei) {};
\node[v, minimum size=\nodesize] (13) at (0* \wid,2*\hei) {};
\node[v, minimum size=\nodesize] (14) at (0* \wid,3*\hei) {};

\draw[e] (11) to[in =00, out =180] (0h1);
\draw[e] (14) to[in =0, out =180] (0h2);
\draw[e] (12) to[in =180, out =180] (13);

\draw[very thick] (0* \wid,-0.25*\hei) -- (0* \wid,3.25*\hei);
\end{tikzpicture}
.
\end{center}
Closing up ends $S^{\vee}(2,4) \to S^{\vee}(0,4)$ sends the first two basis elements to 
\begin{center}
\begin{tikzpicture}[x=1.5cm,y=-.5cm,baseline=-0.7cm]

\def\wid{\standardWidth}
\def\hei{\standardHeight}
\def\nodesize{3}
\def\ang{90}

\node[v, minimum size=\nodesize] (11) at (0* \wid,0*\hei) {};
\node[v, minimum size=\nodesize] (12) at (0* \wid,1*\hei) {};
\node[v, minimum size=\nodesize] (13) at (0* \wid,2*\hei) {};
\node[v, minimum size=\nodesize] (14) at (0* \wid,3*\hei) {};

\draw[e] (11) to[in =180, out =180] (12);
\draw[e] (14) to[in =180, out =180] (13);

\draw[very thick] (0* \wid,-0.25*\hei) -- (0* \wid,3.25*\hei);
\end{tikzpicture}
, and the last one to
\begin{tikzpicture}[x=1.5cm,y=-.5cm,baseline=-0.7cm]

\def\wid{\standardWidth}
\def\hei{\standardHeight}
\def\nodesize{3}
\def\ang{90}

\node[v, minimum size=\nodesize] (11) at (0* \wid,0*\hei) {};
\node[v, minimum size=\nodesize] (12) at (0* \wid,1*\hei) {};
\node[v, minimum size=\nodesize] (13) at (0* \wid,2*\hei) {};
\node[v, minimum size=\nodesize] (14) at (0* \wid,3*\hei) {};

\draw[e] (11) to[in =180, out =180] (14);
\draw[e] (12) to[in =180, out =180] (13);

\draw[very thick] (0* \wid,-0.25*\hei) -- (0* \wid,3.25*\hei);
\end{tikzpicture}
.
\end{center}
The result follows.
\end{proof}

\begin{proposition}[\cref{thm:mainTechnical} (\ref{part:1})] \label{prop:main1} The complex $C_*[1,0]$ with one loop and no dividers is contractible, with homology generated by the (class of the) diagram 
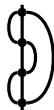
\begin{center}
\begin{tikzpicture}[x=1.5cm,y=-.5cm,baseline=-1.05cm]
\def\wid{\standardWidth}
\def\hei{\standardHeight}
\def\nodesize{3}
\def\ang{90}

\node[v, minimum size=\nodesize] (11) at (0* \wid,0*\hei) {};
\node[v, minimum size=\nodesize] (12) at (0* \wid,1*\hei) {};
\node[v, minimum size=\nodesize] (13) at (0* \wid,2*\hei) {};
\node[v, minimum size=\nodesize] (14) at (0* \wid,3*\hei) {};

\draw[e] (11) to[in =180, out =180] (12);
\draw[e] (13) to[in =180, out =180] (14);

\draw[e] (12) to[in =0, out =0] (13);
\draw[e] (11) to[out = 0, in = 0] (14);

\draw[very thick] (0* \wid,-0.25*\hei) -- (0* \wid,3.25*\hei);
\end{tikzpicture}
\end{center}
The complexes $C_*^R[1,0]$ and $C_*^L[1,0]$ are also contractible.
\end{proposition}

\begin{proof} Closing up left ends (i.e.~applying $\overline{\mathrm{Bar}}_{*-1}((-),\TL_4,S(4,2))$ to the short exact sequence of \cref{rmk:SES}) gives two short exact sequences of chain complexes
\begin{align*}
    & 0 \to \overline{\mathrm{Bar}}_{*-1}(K^\vee,\TL_4,S(4,2)) \to C_*^{LR}[1,0] \to C_*^{R}[1,0] \to 0, \textrm{ and } \\ 
    & 0 \to \overline{\mathrm{Bar}}_{*-1}(K^\vee,\TL_4,S(4,0)) \to C_*^{L}[1,0] \to C_*[1,0] \to 0.
\end{align*}

Consider the first short exact sequence. By \cref{lem:openEnds,lem:SESKernelID}, the bar construction $\overline{\mathrm{Bar}}_{*-1}(K^\vee,\TL_4,S(4,2))$ is the subcomplex of $C_*^{LR}[1,0]$ spanned by differences of pairs of words differing only in their first letter, one beginning with
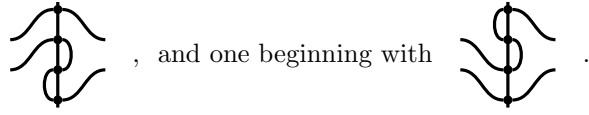
\begin{center}
\begin{tikzpicture}[x=1.5cm,y=-.5cm,baseline=-0.7cm]

\def\wid{\standardWidth}
\def\hei{\standardHeight}
\def\nodesize{3}
\def\ang{90}

\node[] (0h1) at (-0.5* \wid,1*\hei) {};
\node[] (0h2) at (-0.5* \wid,2*\hei) {};
\node[v, minimum size=\nodesize] (11) at (0* \wid,0*\hei) {};
\node[v, minimum size=\nodesize] (12) at (0* \wid,1*\hei) {};
\node[v, minimum size=\nodesize] (13) at (0* \wid,2*\hei) {};
\node[v, minimum size=\nodesize] (14) at (0* \wid,3*\hei) {};
\node[] (1h1) at (0.5* \wid,1*\hei) {};
\node[] (1h2) at (0.5* \wid,2*\hei) {};

\draw[e] (11) to[in =0, out =180] (0h1);
\draw[e] (12) to[in =0, out =180] (0h2);
\draw[e] (13) to[in =180, out =180] (14);

\draw[e] (12) to[in =0, out =0] (13);
\draw[e] (11) to[out = 0, in = 180] (1h1);
\draw[e] (14) to[out = 0, in = 180] (1h2);

\draw[very thick] (0* \wid,-0.25*\hei) -- (0* \wid,3.25*\hei);
\end{tikzpicture}
,
\textrm{ and one beginning with}
\begin{tikzpicture}[x=1.5cm,y=-.5cm,baseline=-0.7cm]

\def\wid{\standardWidth}
\def\hei{\standardHeight}
\def\nodesize{3}
\def\ang{90}

\node[] (0h1) at (-0.5* \wid,1*\hei) {};
\node[] (0h2) at (-0.5* \wid,2*\hei) {};
\node[v, minimum size=\nodesize] (11) at (0* \wid,0*\hei) {};
\node[v, minimum size=\nodesize] (12) at (0* \wid,1*\hei) {};
\node[v, minimum size=\nodesize] (13) at (0* \wid,2*\hei) {};
\node[v, minimum size=\nodesize] (14) at (0* \wid,3*\hei) {};
\node[] (1h1) at (0.5* \wid,1*\hei) {};
\node[] (1h2) at (0.5* \wid,2*\hei) {};

\draw[e] (13) to[in =00, out =180] (0h1);
\draw[e] (14) to[in =0, out =180] (0h2);
\draw[e] (12) to[in =180, out =180] (11);

\draw[e] (12) to[in =0, out =0] (13);
\draw[e] (11) to[out = 0, in = 180] (1h1);
\draw[e] (14) to[out = 0, in = 180] (1h2);

\draw[very thick] (0* \wid,-0.25*\hei) -- (0* \wid,3.25*\hei);
\end{tikzpicture}
.
\end{center}

The first face map $d_0$ deletes this bar, identifying the two words (since they differed only in this first letter), so carrying their difference to zero. That is, $d_0 = 0$ in $\overline{\mathrm{Bar}}_{*-1}(K^\vee,\TL_4,S(4,2))$, so the differential may be rewritten as $d = \sum_{i = 0}^{p-1}(-1)^i d_i = \sum_{i=1}^{p-1}(-1)^i d_i$. The common part of a difference of two words in $C_*^{LR}[1,0]$ differing only in their first letter is precisely a word in $C_*^{LR}[1,0]$ having one letter less.

These considerations identify $\overline{\mathrm{Bar}}_{*-1}(K^\vee,\TL_4,S(4,2))$ as an augmented suspension of $C_*^{LR}[1,0]$ (whose homology we know by \cref{lem:openEnds}). The augmentation sends (the suspension of) the single homology class to the generator
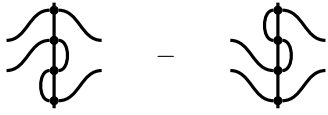
\begin{center}
    \begin{tikzpicture}[x=1.5cm,y=-.5cm,baseline=-0.7cm]

\def\wid{\standardWidth}
\def\hei{\standardHeight}
\def\nodesize{3}
\def\ang{90}

\node[] (0h1) at (-0.5* \wid,1*\hei) {};
\node[] (0h2) at (-0.5* \wid,2*\hei) {};
\node[v, minimum size=\nodesize] (11) at (0* \wid,0*\hei) {};
\node[v, minimum size=\nodesize] (12) at (0* \wid,1*\hei) {};
\node[v, minimum size=\nodesize] (13) at (0* \wid,2*\hei) {};
\node[v, minimum size=\nodesize] (14) at (0* \wid,3*\hei) {};
\node[] (1h1) at (0.5* \wid,1*\hei) {};
\node[] (1h2) at (0.5* \wid,2*\hei) {};

\draw[e] (11) to[in =0, out =180] (0h1);
\draw[e] (12) to[in =0, out =180] (0h2);
\draw[e] (13) to[in =180, out =180] (14);

\draw[e] (12) to[in =0, out =0] (13);
\draw[e] (11) to[out = 0, in = 180] (1h1);
\draw[e] (14) to[out = 0, in = 180] (1h2);

\draw[very thick] (0* \wid,-0.25*\hei) -- (0* \wid,3.25*\hei);
\end{tikzpicture}
\quad $-$
\quad
\begin{tikzpicture}[x=1.5cm,y=-.5cm,baseline=-0.7cm]

\def\wid{\standardWidth}
\def\hei{\standardHeight}
\def\nodesize{3}
\def\ang{90}

\node[] (0h1) at (-0.5* \wid,1*\hei) {};
\node[] (0h2) at (-0.5* \wid,2*\hei) {};
\node[v, minimum size=\nodesize] (11) at (0* \wid,0*\hei) {};
\node[v, minimum size=\nodesize] (12) at (0* \wid,1*\hei) {};
\node[v, minimum size=\nodesize] (13) at (0* \wid,2*\hei) {};
\node[v, minimum size=\nodesize] (14) at (0* \wid,3*\hei) {};
\node[] (1h1) at (0.5* \wid,1*\hei) {};
\node[] (1h2) at (0.5* \wid,2*\hei) {};

\draw[e] (13) to[in =00, out =180] (0h1);
\draw[e] (14) to[in =0, out =180] (0h2);
\draw[e] (12) to[in =180, out =180] (11);

\draw[e] (12) to[in =0, out =0] (13);
\draw[e] (11) to[out = 0, in = 180] (1h1);
\draw[e] (14) to[out = 0, in = 180] (1h2);

\draw[very thick] (0* \wid,-0.25*\hei) -- (0* \wid,3.25*\hei);
\end{tikzpicture}
\end{center}
of $\overline{\mathrm{Bar}}_{0}(K^\vee,\TL_4,S(4,2))$, so the homology of $\overline{\mathrm{Bar}}_{*-1}(K^\vee,\TL_4,S(4,2))$ vanishes. The long exact sequence obtained from our short exact sequence of chain complexes then identifies the homology of $C_*^R[1,0]$ with that of $C_*^{LR}[1,0]$, which again is contractible by \cref{lem:openEnds}. Since $C_*^L[1,0] \cong C_*^R[1,0]$ by horizontal reflection, it is also contractible, and an analogous argument with the second short exact sequence now gives contractbility of $C_*[1,0]$. \end{proof}

\subsection{Part (\ref{part:2})}

In this section we show that the homology of the complex $C_*[2,0]$ with $2$ loops and no dividers is a single copy of $R$ in degree 3, establishing the second part of \cref{thm:mainTechnical}.

\begin{lemma} \label{lem:2noAlphabet} A graffito with $2$ loops and no dividers contains precisely one of the following letters, and this letter occurs precisely once.
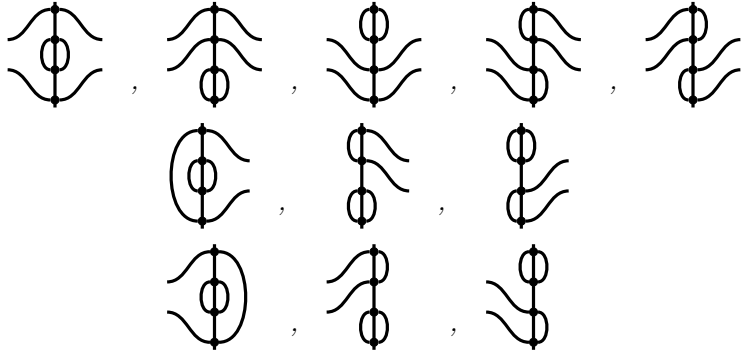
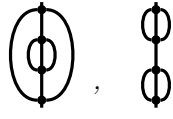
\begin{center}
\begin{tikzpicture}[x=1.5cm,y=-.5cm,baseline=-1.05cm]

\def\wid{\standardWidth}
\def\hei{\standardHeight}
\def\nodesize{3}
\def\ang{90}

\node[] (0h1) at (-0.5* \wid,1*\hei) {};
\node[] (0h2) at (-0.5* \wid,2*\hei) {};
\node[v, minimum size=\nodesize] (11) at (0* \wid,0*\hei) {};
\node[v, minimum size=\nodesize] (12) at (0* \wid,1*\hei) {};
\node[v, minimum size=\nodesize] (13) at (0* \wid,2*\hei) {};
\node[v, minimum size=\nodesize] (14) at (0* \wid,3*\hei) {};
\node[] (1h1) at (0.5* \wid,1*\hei) {};
\node[] (1h2) at (0.5* \wid,2*\hei) {};

\draw[e] (11) to[in =0, out =180] (0h1);
\draw[e] (14) to[in =0, out =180] (0h2);
\draw[e] (13) to[in =180, out =180] (12);

\draw[e] (12) to[in =0, out =0] (13);
\draw[e] (11) to[out = 0, in = 180] (1h1);
\draw[e] (14) to[out = 0, in = 180] (1h2);

\draw[very thick] (0* \wid,-0.25*\hei) -- (0* \wid,3.25*\hei);
\end{tikzpicture}
,
\begin{tikzpicture}[x=1.5cm,y=-.5cm,baseline=-1.05cm]

\def\wid{\standardWidth}
\def\hei{\standardHeight}
\def\nodesize{3}
\def\ang{90}

\node[] (0h1) at (-0.5* \wid,1*\hei) {};
\node[] (0h2) at (-0.5* \wid,2*\hei) {};
\node[v, minimum size=\nodesize] (11) at (0* \wid,0*\hei) {};
\node[v, minimum size=\nodesize] (12) at (0* \wid,1*\hei) {};
\node[v, minimum size=\nodesize] (13) at (0* \wid,2*\hei) {};
\node[v, minimum size=\nodesize] (14) at (0* \wid,3*\hei) {};
\node[] (1h1) at (0.5* \wid,1*\hei) {};
\node[] (1h2) at (0.5* \wid,2*\hei) {};

\draw[e] (11) to[in =0, out =180] (0h1);
\draw[e] (12) to[in =0, out =180] (0h2);
\draw[e] (13) to[in =180, out =180] (14);

\draw[e] (13) to[in =0, out =0] (14);
\draw[e] (11) to[out = 0, in = 180] (1h1);
\draw[e] (12) to[out = 0, in = 180] (1h2);

\draw[very thick] (0* \wid,-0.25*\hei) -- (0* \wid,3.25*\hei);
\end{tikzpicture}
,
\begin{tikzpicture}[x=1.5cm,y=-.5cm,baseline=-1.05cm]

\def\wid{\standardWidth}
\def\hei{\standardHeight}
\def\nodesize{3}
\def\ang{90}

\node[] (0h1) at (-0.5* \wid,1*\hei) {};
\node[] (0h2) at (-0.5* \wid,2*\hei) {};
\node[v, minimum size=\nodesize] (11) at (0* \wid,0*\hei) {};
\node[v, minimum size=\nodesize] (12) at (0* \wid,1*\hei) {};
\node[v, minimum size=\nodesize] (13) at (0* \wid,2*\hei) {};
\node[v, minimum size=\nodesize] (14) at (0* \wid,3*\hei) {};
\node[] (1h1) at (0.5* \wid,1*\hei) {};
\node[] (1h2) at (0.5* \wid,2*\hei) {};

\draw[e] (12) to[in =180, out =180] (11);
\draw[e] (13) to[in =0, out =180] (0h1);
\draw[e] (14) to[in =0, out =180] (0h2);

\draw[e] (12) to[in =0, out =0] (11);
\draw[e] (14) to[out = 0, in = 180] (1h2);
\draw[e] (13) to[out = 0, in = 180] (1h1);

\draw[very thick] (0* \wid,-0.25*\hei) -- (0* \wid,3.25*\hei);
\end{tikzpicture}
,
\begin{tikzpicture}[x=1.5cm,y=-.5cm,baseline=-1.05cm]

\def\wid{\standardWidth}
\def\hei{\standardHeight}
\def\nodesize{3}
\def\ang{90}

\node[] (0h1) at (-0.5* \wid,1*\hei) {};
\node[] (0h2) at (-0.5* \wid,2*\hei) {};
\node[v, minimum size=\nodesize] (11) at (0* \wid,0*\hei) {};
\node[v, minimum size=\nodesize] (12) at (0* \wid,1*\hei) {};
\node[v, minimum size=\nodesize] (13) at (0* \wid,2*\hei) {};
\node[v, minimum size=\nodesize] (14) at (0* \wid,3*\hei) {};
\node[] (1h1) at (0.5* \wid,1*\hei) {};
\node[] (1h2) at (0.5* \wid,2*\hei) {};

\draw[e] (13) to[in =0, out =180] (0h1);
\draw[e] (14) to[in =0, out =180] (0h2);
\draw[e] (11) to[in =180, out =180] (12);

\draw[e] (13) to[in =0, out =0] (14);
\draw[e] (11) to[out = 0, in = 180] (1h1);
\draw[e] (12) to[out = 0, in = 180] (1h2);

\draw[very thick] (0* \wid,-0.25*\hei) -- (0* \wid,3.25*\hei);
\end{tikzpicture}
,
\begin{tikzpicture}[x=1.5cm,y=-.5cm,baseline=-1.05cm]

\def\wid{\standardWidth}
\def\hei{\standardHeight}
\def\nodesize{3}
\def\ang{90}

\node[] (0h1) at (-0.5* \wid,1*\hei) {};
\node[] (0h2) at (-0.5* \wid,2*\hei) {};
\node[v, minimum size=\nodesize] (11) at (0* \wid,0*\hei) {};
\node[v, minimum size=\nodesize] (12) at (0* \wid,1*\hei) {};
\node[v, minimum size=\nodesize] (13) at (0* \wid,2*\hei) {};
\node[v, minimum size=\nodesize] (14) at (0* \wid,3*\hei) {};
\node[] (1h1) at (0.5* \wid,1*\hei) {};
\node[] (1h2) at (0.5* \wid,2*\hei) {};

\draw[e] (13) to[in =180, out =180] (14);
\draw[e] (11) to[in =0, out =180] (0h1);
\draw[e] (12) to[in =0, out =180] (0h2);

\draw[e] (12) to[in =0, out =0] (11);
\draw[e] (14) to[out = 0, in = 180] (1h2);
\draw[e] (13) to[out = 0, in = 180] (1h1);

\draw[very thick] (0* \wid,-0.25*\hei) -- (0* \wid,3.25*\hei);
\end{tikzpicture}
\end{center}

\begin{center}
\begin{tikzpicture}[x=1.5cm,y=-.5cm,baseline=-1.05cm]

\def\wid{\standardWidth}
\def\hei{\standardHeight}
\def\nodesize{3}
\def\ang{90}

\node[] (0h1) at (0.5* \wid,1*\hei) {};
\node[] (0h2) at (0.5* \wid,2*\hei) {};
\node[v, minimum size=\nodesize] (11) at (0* \wid,0*\hei) {};
\node[v, minimum size=\nodesize] (12) at (0* \wid,1*\hei) {};
\node[v, minimum size=\nodesize] (13) at (0* \wid,2*\hei) {};
\node[v, minimum size=\nodesize] (14) at (0* \wid,3*\hei) {};
\node[] (1h1) at (0.5* \wid,1*\hei) {};
\node[] (1h2) at (0.5* \wid,2*\hei) {};

\draw[e] (11) to[in =180, out =180] (14);
\draw[e] (13) to[in =180, out =180] (12);

\draw[e] (12) to[in =0, out =0] (13);
\draw[e] (11) to[out = 0, in = 180] (1h1);
\draw[e] (14) to[out = 0, in = 180] (1h2);

\draw[very thick] (0* \wid,-0.25*\hei) -- (0* \wid,3.25*\hei);
\end{tikzpicture}
,
\begin{tikzpicture}[x=1.5cm,y=-.5cm,baseline=-1.05cm]

\def\wid{\standardWidth}
\def\hei{\standardHeight}
\def\nodesize{3}
\def\ang{90}

\node[] (0h1) at (-0.5* \wid,1*\hei) {};
\node[] (0h2) at (-0.5* \wid,2*\hei) {};
\node[v, minimum size=\nodesize] (11) at (0* \wid,0*\hei) {};
\node[v, minimum size=\nodesize] (12) at (0* \wid,1*\hei) {};
\node[v, minimum size=\nodesize] (13) at (0* \wid,2*\hei) {};
\node[v, minimum size=\nodesize] (14) at (0* \wid,3*\hei) {};
\node[] (1h1) at (0.5* \wid,1*\hei) {};
\node[] (1h2) at (0.5* \wid,2*\hei) {};

\draw[e] (11) to[in =180, out =180] (12);
\draw[e] (13) to[in =180, out =180] (14);

\draw[e] (13) to[in =0, out =0] (14);
\draw[e] (11) to[out = 0, in = 180] (1h1);
\draw[e] (12) to[out = 0, in = 180] (1h2);

\draw[very thick] (0* \wid,-0.25*\hei) -- (0* \wid,3.25*\hei);
\end{tikzpicture}
,
\begin{tikzpicture}[x=1.5cm,y=-.5cm,baseline=-1.05cm]

\def\wid{\standardWidth}
\def\hei{\standardHeight}
\def\nodesize{3}
\def\ang{90}

\node[] (0h1) at (-0.5* \wid,1*\hei) {};
\node[] (0h2) at (-0.5* \wid,2*\hei) {};
\node[v, minimum size=\nodesize] (11) at (0* \wid,0*\hei) {};
\node[v, minimum size=\nodesize] (12) at (0* \wid,1*\hei) {};
\node[v, minimum size=\nodesize] (13) at (0* \wid,2*\hei) {};
\node[v, minimum size=\nodesize] (14) at (0* \wid,3*\hei) {};
\node[] (1h1) at (0.5* \wid,1*\hei) {};
\node[] (1h2) at (0.5* \wid,2*\hei) {};

\draw[e] (12) to[in =180, out =180] (11);
\draw[e] (14) to[in =180, out =180] (13);

\draw[e] (12) to[in =0, out =0] (11);
\draw[e] (14) to[out = 0, in = 180] (1h2);
\draw[e] (13) to[out = 0, in = 180] (1h1);

\draw[very thick] (0* \wid,-0.25*\hei) -- (0* \wid,3.25*\hei);
\end{tikzpicture}
\end{center}

\begin{center}
\begin{tikzpicture}[x=1.5cm,y=-.5cm,baseline=-1.05cm]

\def\wid{\standardWidth}
\def\hei{\standardHeight}
\def\nodesize{3}
\def\ang{90}

\node[] (0h1) at (-0.5* \wid,1*\hei) {};
\node[] (0h2) at (-0.5* \wid,2*\hei) {};
\node[v, minimum size=\nodesize] (11) at (0* \wid,0*\hei) {};
\node[v, minimum size=\nodesize] (12) at (0* \wid,1*\hei) {};
\node[v, minimum size=\nodesize] (13) at (0* \wid,2*\hei) {};
\node[v, minimum size=\nodesize] (14) at (0* \wid,3*\hei) {};
\node[] (1h1) at (0.5* \wid,1*\hei) {};
\node[] (1h2) at (0.5* \wid,2*\hei) {};

\draw[e] (11) to[in =0, out =0] (14);
\draw[e] (12) to[in =0, out =0] (13);

\draw[e] (12) to[in =180, out =180] (13);
\draw[e] (11) to[out = 180, in = 0] (0h1);
\draw[e] (14) to[out = 180, in = 0] (0h2);

\draw[very thick] (0* \wid,-0.25*\hei) -- (0* \wid,3.25*\hei);
\end{tikzpicture}
,
\begin{tikzpicture}[x=1.5cm,y=-.5cm,baseline=-1.05cm]

\def\wid{\standardWidth}
\def\hei{\standardHeight}
\def\nodesize{3}
\def\ang{90}

\node[] (0h1) at (-0.5* \wid,1*\hei) {};
\node[] (0h2) at (-0.5* \wid,2*\hei) {};
\node[v, minimum size=\nodesize] (11) at (0* \wid,0*\hei) {};
\node[v, minimum size=\nodesize] (12) at (0* \wid,1*\hei) {};
\node[v, minimum size=\nodesize] (13) at (0* \wid,2*\hei) {};
\node[v, minimum size=\nodesize] (14) at (0* \wid,3*\hei) {};
\node[] (1h1) at (0.5* \wid,1*\hei) {};
\node[] (1h2) at (0.5* \wid,2*\hei) {};

\draw[e] (11) to[in =0, out =0] (12);
\draw[e] (13) to[in =0, out =0] (14);

\draw[e] (13) to[in =180, out =180] (14);
\draw[e] (11) to[out = 180, in = 0] (0h1);
\draw[e] (12) to[out = 180, in = 0] (0h2);

\draw[very thick] (0* \wid,-0.25*\hei) -- (0* \wid,3.25*\hei);
\end{tikzpicture}
,
\begin{tikzpicture}[x=1.5cm,y=-.5cm,baseline=-1.05cm]

\def\wid{\standardWidth}
\def\hei{\standardHeight}
\def\nodesize{3}
\def\ang{90}

\node[] (0h1) at (-0.5* \wid,1*\hei) {};
\node[] (0h2) at (-0.5* \wid,2*\hei) {};
\node[v, minimum size=\nodesize] (11) at (0* \wid,0*\hei) {};
\node[v, minimum size=\nodesize] (12) at (0* \wid,1*\hei) {};
\node[v, minimum size=\nodesize] (13) at (0* \wid,2*\hei) {};
\node[v, minimum size=\nodesize] (14) at (0* \wid,3*\hei) {};
\node[] (1h1) at (0.5* \wid,1*\hei) {};
\node[] (1h2) at (0.5* \wid,2*\hei) {};

\draw[e] (14) to[in =0, out =0] (13);
\draw[e] (12) to[in =0, out =0] (11);

\draw[e] (12) to[in =180, out =180] (11);
\draw[e] (14) to[out = 180, in = 0] (0h2);
\draw[e] (13) to[out = 180, in = 0] (0h1);

\draw[very thick] (0* \wid,-0.25*\hei) -- (0* \wid,3.25*\hei);
\end{tikzpicture}
\end{center}

\begin{center}

\begin{tikzpicture}[x=1.5cm,y=-.5cm,baseline=-1.05cm]
\def\wid{\standardWidth}
\def\hei{\standardHeight}
\def\nodesize{3}
\def\ang{90}

\node[v, minimum size=\nodesize] (11) at (0* \wid,0*\hei) {};
\node[v, minimum size=\nodesize] (12) at (0* \wid,1*\hei) {};
\node[v, minimum size=\nodesize] (13) at (0* \wid,2*\hei) {};
\node[v, minimum size=\nodesize] (14) at (0* \wid,3*\hei) {};

\draw[e] (11) to[in =180, out =180] (14);
\draw[e] (13) to[in =180, out =180] (12);

\draw[e] (12) to[in =0, out =0] (13);
\draw[e] (11) to[out = 0, in = 0] (14);

\draw[very thick] (0* \wid,-0.25*\hei) -- (0* \wid,3.25*\hei);
\end{tikzpicture}
,
\quad
\begin{tikzpicture}[x=1.5cm,y=-.5cm,baseline=-1.05cm]
\def\wid{\standardWidth}
\def\hei{\standardHeight}
\def\nodesize{3}
\def\ang{90}

\node[v, minimum size=\nodesize] (11) at (0* \wid,0*\hei) {};
\node[v, minimum size=\nodesize] (12) at (0* \wid,1*\hei) {};
\node[v, minimum size=\nodesize] (13) at (0* \wid,2*\hei) {};
\node[v, minimum size=\nodesize] (14) at (0* \wid,3*\hei) {};

\draw[e] (11) to[in =180, out =180] (12);
\draw[e] (13) to[in =180, out =180] (14);

\draw[e] (11) to[in =0, out =0] (12);
\draw[e] (13) to[out = 0, in = 0] (14);

\draw[very thick] (0* \wid,-0.25*\hei) -- (0* \wid,3.25*\hei);
\end{tikzpicture}
\end{center}
\end{lemma}

\begin{proof} Since there are no dividers, there must be at least one bar where both loops are pinned. Since we are working in the reduced bar construction, at most one of these loops can continue to the left, and at most one can continue to the right. In particular, this bar is the only one pinning both loops. The letters shown above are precisely those which involve components of two distinct loops. \end{proof}

We will call this letter the \emph{pivot}. Deleting the bar at the pivot either unpins a loop, giving zero, or results in an extra divider, so we obtain the following.

\begin{lemma} \label{lem:deletePivotGet0} In $C_*[2,0]$, deleting the bar at the pivot gives zero. \qed
\end{lemma}

The idea is now that if we restrict ourselves to only one pivot, for example
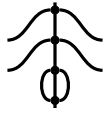
\begin{center}
    \begin{tikzpicture}[x=1.5cm,y=-.5cm,baseline=-1.05cm]

\def\wid{\standardWidth}
\def\hei{\standardHeight}
\def\nodesize{3}
\def\ang{90}

\node[] (0h1) at (-0.5* \wid,1*\hei) {};
\node[] (0h2) at (-0.5* \wid,2*\hei) {};
\node[v, minimum size=\nodesize] (11) at (0* \wid,0*\hei) {};
\node[v, minimum size=\nodesize] (12) at (0* \wid,1*\hei) {};
\node[v, minimum size=\nodesize] (13) at (0* \wid,2*\hei) {};
\node[v, minimum size=\nodesize] (14) at (0* \wid,3*\hei) {};
\node[] (1h1) at (0.5* \wid,1*\hei) {};
\node[] (1h2) at (0.5* \wid,2*\hei) {};

\draw[e] (11) to[in =0, out =180] (0h1);
\draw[e] (12) to[in =0, out =180] (0h2);
\draw[e] (13) to[in =180, out =180] (14);

\draw[e] (13) to[in =0, out =0] (14);
\draw[e] (11) to[out = 0, in = 180] (1h1);
\draw[e] (12) to[out = 0, in = 180] (1h2);

\draw[very thick] (0* \wid,-0.25*\hei) -- (0* \wid,3.25*\hei);
\end{tikzpicture}
\end{center}
then the resulting complex is just a copy of the suspension of the tensor product $C_*^R[1,0] \otimes_R C_*^L[1,0]$ of the open complexes of \cref{def:complexCij}. The homology of $C_*[2,0]$ can now be found by thinking about how the complexes associated to the different letters fit together.

We have a direct sum decomposition $$C_*[2,0] = C_*[2,0]^{\mathrm{in}} \oplus C_*[2,0]^{\mathrm{out}},$$ where $C_*[2,0]^{\mathrm{in}}$ is spanned by graffiti where one of the two loops is nested inside the other, and $C_*[2,0]^{\mathrm{out}}$ is spanned by the other graffiti.

In both cases, the following will be our `atomic' pieces. For a pivot $p$, let $C_*(p)$ be the subquotient of $C_*[2,0]$ where the pivot is $p$. For brevity, we say that $p$ is \emph{open} on the left or right if it has two connections on that side, and \emph{closed} otherwise.

\begin{lemma} The complexes $C_*(p)$ are described as follows:
\begin{enumerate}
    \item If $p$ is closed on both sides, then $C_*(p)$ consists only of the single letter $p$ in degree 1.
    \item If $p$ is open on the right (respectively left), then $C_*(p)$ is isomorphic to the suspension of the complex $C_*^L[1,0]$ (respectively $C_*^R[1,0]$) of half-open graffiti which close up to have one loop and no dividers.
    \item If $p$ is open on both sides, then $C_*(p)$ is isomorphic to the suspension of the tensor product $C_*^R[1,0] \otimes_R C_*^L[1,0]$.
\end{enumerate}
\end{lemma}

\begin{proof} Part $(1)$ is obtained by noting that if a word with more than one letter includes a letter with no connections on either side, then it has a divider.

For Part $(2)$, suppose given a word $w$ containing such a pivot $p$. Note that:
\begin{itemize}
    \item since there are no dividers, $p$ must be either the first or last letter of $w$,
    \item since $p$ is open on one side, $w$ must consist of at least two letters, and
    \item the part of $w$ excluding $p$ is precisely an element of $C_*^R[1,0]$ or $C_*^L[1,0]$, depending on whether $p$ is the first or last letter.
\end{itemize}
By \cref{lem:deletePivotGet0}, the face map at the pivot is zero, so the result follows.

Part $(3)$ is similar. Given a word $w$ containing such a pivot $p$, note that $p$ must be internal to $w$, and that $w$ must have at least one letter preceding $p$ and at least one letter following. The part preceding $p$ is precisely a right-open graffito in $C_*^R[1,0]$, and the part following is precisely a left-open graffito in $C_*^L[1,0]$. As before, the claim then follows from \cref{lem:deletePivotGet0}.
\end{proof}

Since $C_*^R[1,0]$ and $C_*^L[1,0]$ are contractible by \cref{prop:main1}, we obtain the following.
\begin{corollary} \label{cor:Cphomology} The complexes $C_*(p)$ have the following homology:
\begin{itemize}
    \item If $p$ is closed on both sides, then the homology of $C_*(p)$ consists only a copy of $R$ in degree 1, generated by $p$.
    \item If $p$ is open on one side, then the homology of $C_*(p)$ consists of a single copy of $R$ in degree 2, generated by any word of length 2.
    \item If $p$ is open on both sides, then the homology of $C_*(p)$ is generated by a single copy of $R$ in degree 3, generated by any word of length 3. \qed
\end{itemize}
\end{corollary}

\subsubsection{The nested summand}

Write $F^0(C_*[2,0]^{\mathrm{in}})$ for the subcomplex of $C_*[2,0]^{\mathrm{in}}$ spanned by graffiti where the pivot is closed on both sides. Write $F^1(C_*[2,0]^{\mathrm{in}})$ for the subcomplex spanned by $F^0(C_*[2,0]^{\mathrm{in}})$ and graffiti where the pivot is open on one side (either left or right) and write $F^2(C_*[2,0]^{\mathrm{in}})$ for the subcomplex spanned by $F^1(C_*[2,0]^{\mathrm{in}})$ and graffiti where the pivot is open on both sides, so that in particular $F^2(C_*[2,0]^{\mathrm{in}})= C_*[2,0]^{\mathrm{in}}$. The boundary map can close off the sides of the pivot, but not open them, so this gives a three-step filtration of $C_*[2,0]^{\mathrm{in}}$, and we can consider the resulting spectral sequence.

Since letters can have at most two connections on each side, it follows that if one loop is nested inside the other then the inner loop must have length 2 and be entirely supported on the pivot. Consulting the table of pivots in \cref{lem:2noAlphabet} then gives the following lemma (here we actually substitute in particular diagrams $p$ in $C_*(p)$, which may make it look like new notation at first sight, though it is not).

\begin{lemma} \label{lem:2nestedQuots} The filtration quotients are described as follows.
\begin{itemize}
    \item $F^0(C_*[2,0]^{\mathrm{in}})$ consists only of the graffito 
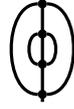
\begin{center}

\begin{tikzpicture}[x=1.5cm,y=-.5cm,baseline=-1.05cm]
\def\wid{\standardWidth}
\def\hei{\standardHeight}
\def\nodesize{3}
\def\ang{90}

\node[v, minimum size=\nodesize] (11) at (0* \wid,0*\hei) {};
\node[v, minimum size=\nodesize] (12) at (0* \wid,1*\hei) {};
\node[v, minimum size=\nodesize] (13) at (0* \wid,2*\hei) {};
\node[v, minimum size=\nodesize] (14) at (0* \wid,3*\hei) {};

\draw[e] (11) to[in =180, out =180] (14);
\draw[e] (13) to[in =180, out =180] (12);

\draw[e] (12) to[in =0, out =0] (13);
\draw[e] (11) to[out = 0, in = 0] (14);

\draw[very thick] (0* \wid,-0.25*\hei) -- (0* \wid,3.25*\hei);
\end{tikzpicture}
\end{center} in degree 1.
\item We have a direct sum decomposition \begin{center}
$\faktor{F^1(C_*[2,0]^{\mathrm{in}})}{F^0(C_*[2,0]^{\mathrm{in}})}$ 
\quad
$=$
\quad
$C_*($
\begin{tikzpicture}[x=1.5cm,y=-.5cm,baseline=-0.7cm]

\def\wid{\standardWidth}
\def\hei{\standardHeight}
\def\nodesize{3}
\def\ang{90}

\node[] (0h1) at (-0.5* \wid,1*\hei) {};
\node[] (0h2) at (-0.5* \wid,2*\hei) {};
\node[v, minimum size=\nodesize] (11) at (0* \wid,0*\hei) {};
\node[v, minimum size=\nodesize] (12) at (0* \wid,1*\hei) {};
\node[v, minimum size=\nodesize] (13) at (0* \wid,2*\hei) {};
\node[v, minimum size=\nodesize] (14) at (0* \wid,3*\hei) {};

\draw[e] (11) to[in =0, out =0] (14);
\draw[e] (12) to[in =0, out =0] (13);

\draw[e] (12) to[in =180, out =180] (13);
\draw[e] (11) to[out = 180, in = 0] (0h1);
\draw[e] (14) to[out = 180, in = 0] (0h2);

\draw[very thick] (0* \wid,-0.25*\hei) -- (0* \wid,3.25*\hei);
\end{tikzpicture}
$)$
$\oplus$
$C_*($
\begin{tikzpicture}[x=1.5cm,y=-.5cm,baseline=-0.7cm]

\def\wid{\standardWidth}
\def\hei{\standardHeight}
\def\nodesize{3}
\def\ang{90}
\node[v, minimum size=\nodesize] (11) at (0* \wid,0*\hei) {};
\node[v, minimum size=\nodesize] (12) at (0* \wid,1*\hei) {};
\node[v, minimum size=\nodesize] (13) at (0* \wid,2*\hei) {};
\node[v, minimum size=\nodesize] (14) at (0* \wid,3*\hei) {};
\node[] (1h1) at (0.5* \wid,1*\hei) {};
\node[] (1h2) at (0.5* \wid,2*\hei) {};

\draw[e] (11) to[in =180, out =180] (14);
\draw[e] (13) to[in =180, out =180] (12);

\draw[e] (12) to[in =0, out =0] (13);
\draw[e] (11) to[out = 0, in = 180] (1h1);
\draw[e] (14) to[out = 0, in = 180] (1h2);

\draw[very thick] (0* \wid,-0.25*\hei) -- (0* \wid,3.25*\hei);
\end{tikzpicture}
$)$.
\end{center}
\item We have the identification \begin{center}
$\faktor{C_*[2,0]^{\mathrm{in}}}{F^1(C_*[2,0]^{\mathrm{in}})}$ 
\quad
$=$
\quad
$C_*($
\begin{tikzpicture}[x=1.5cm,y=-.5cm,baseline=-0.7cm]

\def\wid{\standardWidth}
\def\hei{\standardHeight}
\def\nodesize{3}
\def\ang{90}

\node[] (0h1) at (-0.5* \wid,1*\hei) {};
\node[] (0h2) at (-0.5* \wid,2*\hei) {};
\node[v, minimum size=\nodesize] (11) at (0* \wid,0*\hei) {};
\node[v, minimum size=\nodesize] (12) at (0* \wid,1*\hei) {};
\node[v, minimum size=\nodesize] (13) at (0* \wid,2*\hei) {};
\node[v, minimum size=\nodesize] (14) at (0* \wid,3*\hei) {};
\node[] (1h1) at (0.5* \wid,1*\hei) {};
\node[] (1h2) at (0.5* \wid,2*\hei) {};

\draw[e] (11) to[in =0, out =180] (0h1);
\draw[e] (14) to[in =0, out =180] (0h2);
\draw[e] (13) to[in =180, out =180] (12);

\draw[e] (12) to[in =0, out =0] (13);
\draw[e] (11) to[out = 0, in = 180] (1h1);
\draw[e] (14) to[out = 0, in = 180] (1h2);

\draw[very thick] (0* \wid,-0.25*\hei) -- (0* \wid,3.25*\hei);
\end{tikzpicture}
$)$. \qed
\end{center}
\end{itemize}
\end{lemma}

\begin{proposition}
    The homology of $C_*[2,0]^{\mathrm{in}}$ vanishes.
\end{proposition}

\begin{proof}
Consider the spectral sequence associated to the filtration $F^p(C_*[2,0]^{\mathrm{in}})$, which converges to the homology of $C_*[2,0]^{\mathrm{in}}$. \cref{lem:2nestedQuots} and \cref{cor:Cphomology} tell us that the $E^1$-page $$E^1_{p,q} = H_{p+q}(\faktor{F^p(C_*[2,0]^{\mathrm{in}})}{F^{p-1}(C_*[2,0]^{\mathrm{in}})})$$ is concentrated in the $q=1$ row, in the entries $(0,1)$, $(1,1)$, and $(2,1)$, and looks as in \cref{fig:2nested} (a choice of generator of each copy of $R$ is indicated).

  \begin{figure}[h]
		\begin{tikzpicture}[scale=0.7]
            \def\colsp{3.8}
			\draw[line width=0.2cm, gray!20, <->] (-0.5*\colsp,5)--node[black, above, pos=0] {$q$}(-0.5*\colsp,-1)--(4*\colsp,-1) node[black, right] {$p$};
			\draw (-0.5*\colsp,0) node {$0$};	
			\draw (-0.5*\colsp,1) node {$1$};
			\draw (-0.5*\colsp,2) node {$2$};
            \draw (-0.5*\colsp,3) node {$3$};
            \draw (0,-1) node {$0$};
			\draw (1*\colsp,-1) node {$1$};	
            \draw (2*\colsp,-1) node {$2$};
			\draw (3*\colsp,-1) node {$3$};

            \draw (1*\colsp,1) node {$\oplus$};

            \foreach \x in {0,1,2,3} \foreach \y in {0,2,3} \draw (\x*\colsp,\y) node[black!80,scale=1.2] {$0$};
            \draw (3*\colsp,1) node[black!80,scale=1.2] {$0$};

			\foreach \x in {(3.5*\colsp,-1), (3.5*\colsp,0), (3.5*\colsp,1), (3.5*\colsp,2),(3.5*\colsp,3) }
			\draw \x node {$\cdots$};	
            \foreach \x in {(-0.5*\colsp,4), (0*\colsp,4), (1*\colsp,4), (2*\colsp,4),(3*\colsp,4) }
			\draw \x node {$\vdots$};	

            \foreach \x in {(0.3*\colsp,1), (1.6*\colsp,1) }
			\draw \x node {$\xleftarrow{}$};	

        \begin{scope}[xshift=0, yshift = 0.7cm, scale = 0.3] 
            \def\wid{\standardWidth}
            \def\hei{\standardHeight}
            \def\nodesize{3}
            \def\ang{90}

            \node[v, minimum size=\nodesize] (11) at (0* \wid,0*\hei) {};
            \node[v, minimum size=\nodesize] (12) at (0* \wid,1*\hei) {};
            \node[v, minimum size=\nodesize] (13) at (0* \wid,2*\hei) {};
            \node[v, minimum size=\nodesize] (14) at (0* \wid,3*\hei) {};

            \draw[e] (11) to[in =180, out =180] (14);
            \draw[e] (13) to[in =180, out =180] (12);

            \draw[e] (12) to[in =0, out =0] (13);
            \draw[e] (11) to[out = 0, in = 0] (14);

            \draw[thick] (0* \wid,-0.25*\hei) -- (0* \wid,3.25*\hei);
        \end{scope}

        \begin{scope}[xshift=0.8*\colsp cm, yshift = 0.7cm, scale = 0.3] 
            \def\wid{\standardWidth}
            \def\hei{\standardHeight}
            \def\nodesize{3}
            \def\ang{90}

            \node[v, minimum size=\nodesize] (01) at (-2* \wid,0*\hei) {};
            \node[v, minimum size=\nodesize] (02) at (-2* \wid,1*\hei) {};
            \node[v, minimum size=\nodesize] (03) at (-2* \wid,2*\hei) {};
            \node[v, minimum size=\nodesize] (04) at (-2* \wid,3*\hei) {};

            \node[v, minimum size=\nodesize] (11) at (0* \wid,0*\hei) {};
            \node[v, minimum size=\nodesize] (12) at (0* \wid,1*\hei) {};
            \node[v, minimum size=\nodesize] (13) at (0* \wid,2*\hei) {};
            \node[v, minimum size=\nodesize] (14) at (0* \wid,3*\hei) {};

            \draw[e] (11) to[in =0, out =0] (14);
            \draw[e] (12) to[in =0, out =0] (13);

            \draw[e] (12) to[in =180, out =180] (13);
            \draw[e] (11) to[out = 180, in = 0] (01);
            \draw[e] (14) to[out = 180, in = 0] (02);

            \draw[e] (03) to[in =0, out =0] (04);
            \draw[e] (01) to[in =180, out =180] (04);
            \draw[e] (02) to[in =180, out =180] (03);

            \draw[thick] (0* \wid,-0.25*\hei) -- (0* \wid,3.25*\hei);
            \draw[thick] (-2* \wid,-0.25*\hei) -- (-2* \wid,3.25*\hei);
        \end{scope}

        \begin{scope}[xshift=1.35*\colsp cm, yshift = 0.7cm, scale = 0.3] 
            \def\wid{\standardWidth}
            \def\hei{\standardHeight}
            \def\nodesize{3}
            \def\ang{90}

            \node[v, minimum size=\nodesize] (01) at (-2* \wid,0*\hei) {};
            \node[v, minimum size=\nodesize] (02) at (-2* \wid,1*\hei) {};
            \node[v, minimum size=\nodesize] (03) at (-2* \wid,2*\hei) {};
            \node[v, minimum size=\nodesize] (04) at (-2* \wid,3*\hei) {};

            \node[v, minimum size=\nodesize] (11) at (0* \wid,0*\hei) {};
            \node[v, minimum size=\nodesize] (12) at (0* \wid,1*\hei) {};
            \node[v, minimum size=\nodesize] (13) at (0* \wid,2*\hei) {};
            \node[v, minimum size=\nodesize] (14) at (0* \wid,3*\hei) {};

            \draw[e] (01) to[in =180, out =180] (04);
            \draw[e] (02) to[in =180, out =180] (03);

            \draw[e] (02) to[in =0, out =0] (03);
            \draw[e] (01) to[out = 0, in = 180] (11);
            \draw[e] (04) to[out = 0, in = 180] (12);

            \draw[e] (13) to[in =180, out =180] (14);
            \draw[e] (11) to[in =0, out =0] (14);
            \draw[e] (12) to[in =0, out =0] (13);

            \draw[thick] (0* \wid,-0.25*\hei) -- (0* \wid,3.25*\hei);
            \draw[thick] (-2* \wid,-0.25*\hei) -- (-2* \wid,3.25*\hei);
        \end{scope}

        \begin{scope}[xshift=2*\colsp cm, yshift = 0.7cm, scale = 0.3] 
            \def\wid{\standardWidth}
            \def\hei{\standardHeight}
            \def\nodesize{3}
            \def\ang{90}

            \node[v, minimum size=\nodesize] (01) at (-2* \wid,0*\hei) {};
            \node[v, minimum size=\nodesize] (02) at (-2* \wid,1*\hei) {};
            \node[v, minimum size=\nodesize] (03) at (-2* \wid,2*\hei) {};
            \node[v, minimum size=\nodesize] (04) at (-2* \wid,3*\hei) {};

            \node[v, minimum size=\nodesize] (11) at (0* \wid,0*\hei) {};
            \node[v, minimum size=\nodesize] (12) at (0* \wid,1*\hei) {};
            \node[v, minimum size=\nodesize] (13) at (0* \wid,2*\hei) {};
            \node[v, minimum size=\nodesize] (14) at (0* \wid,3*\hei) {};

            \node[v, minimum size=\nodesize] (21) at (2* \wid,0*\hei) {};
            \node[v, minimum size=\nodesize] (22) at (2* \wid,1*\hei) {};
            \node[v, minimum size=\nodesize] (23) at (2* \wid,2*\hei) {};
            \node[v, minimum size=\nodesize] (24) at (2* \wid,3*\hei) {};

            \draw[e] (12) to[in =0, out =0] (13);

            \draw[e] (12) to[in =180, out =180] (13);
            \draw[e] (11) to[out = 180, in = 0] (01);
            \draw[e] (14) to[out = 180, in = 0] (02);

            \draw[e] (03) to[in =0, out =0] (04);
            \draw[e] (01) to[in =180, out =180] (04);
            \draw[e] (02) to[in =180, out =180] (03);

            \draw[e] (11) to[out = 0, in = 180] (21);
            \draw[e] (14) to[out = 0, in = 180] (22);

            \draw[e] (23) to[in =180, out =180] (24);
            \draw[e] (21) to[in =0, out =0] (24);
            \draw[e] (22) to[in =0, out =0] (23);

            \draw[thick] (0* \wid,-0.25*\hei) -- (0* \wid,3.25*\hei);
            \draw[thick] (-2* \wid,-0.25*\hei) -- (-2* \wid,3.25*\hei);
            \draw[thick] (2* \wid,-0.25*\hei) -- (2* \wid,3.25*\hei);
        \end{scope}
			
		\end{tikzpicture}
\caption{The page $E^1_{p,q}$ of the spectral sequence computing the homology of $C_*[2,0]^{\mathrm{in}}$.}
\label{fig:2nested}
\end{figure}
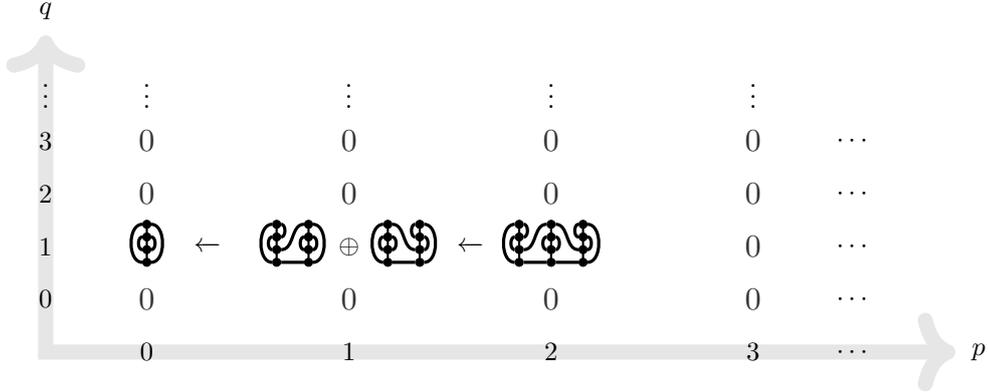
We may now check directly that these entries cancel under the $d^1$-differential, so the $E^2$-page vanishes, and the result follows.
\end{proof}

\subsubsection{The unnested summand}

As before, we will filter by `pivot closure'. For $i=0,1,2$, define $F^i(C_*[2,0]^{\textrm{out}})$ as for $C_*[2,0]^{\textrm{in}}$. This again gives a three-step filtration of $C_*[2,0]^{\textrm{out}}$, and we will again consider the resulting spectral sequence. As before, consulting the table of possible pivots (\cref{lem:2noAlphabet}) gives the following.

\begin{lemma} \label{lem:2unnestedQuots} The filtration quotients are described as follows.
\begin{itemize}
    \item $F^0(C_*[2,0]^{\mathrm{out}})$ consists only of the graffito 
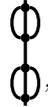
\begin{center}
\begin{tikzpicture}[x=1.5cm,y=-.5cm,baseline=-1.05cm]
\def\wid{\standardWidth}
\def\hei{\standardHeight}
\def\nodesize{3}
\def\ang{90}

\node[v, minimum size=\nodesize] (11) at (0* \wid,0*\hei) {};
\node[v, minimum size=\nodesize] (12) at (0* \wid,1*\hei) {};
\node[v, minimum size=\nodesize] (13) at (0* \wid,2*\hei) {};
\node[v, minimum size=\nodesize] (14) at (0* \wid,3*\hei) {};

\draw[e] (11) to[in =180, out =180] (12);
\draw[e] (13) to[in =180, out =180] (14);

\draw[e] (11) to[in =0, out =0] (12);
\draw[e] (13) to[out = 0, in = 0] (14);

\draw[very thick] (0* \wid,-0.25*\hei) -- (0* \wid,3.25*\hei);
\end{tikzpicture},
\end{center} in degree 1.
\item We have a direct sum decomposition \begin{center}
$\faktor{F^1(C_*[2,0]^{\mathrm{out}})}{F^0(C_*[2,0]^{\mathrm{out}})}$ 
\quad
$=$
\end{center}

\begin{center}
$= C_*($
\begin{tikzpicture}[x=1.5cm,y=-.5cm,baseline=-0.7cm]

\def\wid{\standardWidth}
\def\hei{\standardHeight}
\def\nodesize{3}
\def\ang{90}
\node[v, minimum size=\nodesize] (11) at (0* \wid,0*\hei) {};
\node[v, minimum size=\nodesize] (12) at (0* \wid,1*\hei) {};
\node[v, minimum size=\nodesize] (13) at (0* \wid,2*\hei) {};
\node[v, minimum size=\nodesize] (14) at (0* \wid,3*\hei) {};
\node[] (1h1) at (0.5* \wid,1*\hei) {};
\node[] (1h2) at (0.5* \wid,2*\hei) {};

\draw[e] (11) to[in =180, out =180] (12);
\draw[e] (13) to[in =180, out =180] (14);

\draw[e] (13) to[in =0, out =0] (14);
\draw[e] (11) to[out = 0, in = 180] (1h1);
\draw[e] (12) to[out = 0, in = 180] (1h2);

\draw[very thick] (0* \wid,-0.25*\hei) -- (0* \wid,3.25*\hei);
\end{tikzpicture}
$)$
$\oplus$
$C_*($
\begin{tikzpicture}[x=1.5cm,y=-.5cm,baseline=-0.7cm]

\def\wid{\standardWidth}
\def\hei{\standardHeight}
\def\nodesize{3}
\def\ang{90}
\node[v, minimum size=\nodesize] (11) at (0* \wid,0*\hei) {};
\node[v, minimum size=\nodesize] (12) at (0* \wid,1*\hei) {};
\node[v, minimum size=\nodesize] (13) at (0* \wid,2*\hei) {};
\node[v, minimum size=\nodesize] (14) at (0* \wid,3*\hei) {};
\node[] (1h1) at (0.5* \wid,1*\hei) {};
\node[] (1h2) at (0.5* \wid,2*\hei) {};

\draw[e] (12) to[in =180, out =180] (11);
\draw[e] (14) to[in =180, out =180] (13);

\draw[e] (12) to[in =0, out =0] (11);
\draw[e] (14) to[out = 0, in = 180] (1h2);
\draw[e] (13) to[out = 0, in = 180] (1h1);

\draw[very thick] (0* \wid,-0.25*\hei) -- (0* \wid,3.25*\hei);
\end{tikzpicture}
$)$
$\oplus$
$C_*($
\begin{tikzpicture}[x=1.5cm,y=-.5cm,baseline=-0.7cm]

\def\wid{\standardWidth}
\def\hei{\standardHeight}
\def\nodesize{3}
\def\ang{90}

\node[] (0h1) at (-0.5* \wid,1*\hei) {};
\node[] (0h2) at (-0.5* \wid,2*\hei) {};
\node[v, minimum size=\nodesize] (11) at (0* \wid,0*\hei) {};
\node[v, minimum size=\nodesize] (12) at (0* \wid,1*\hei) {};
\node[v, minimum size=\nodesize] (13) at (0* \wid,2*\hei) {};
\node[v, minimum size=\nodesize] (14) at (0* \wid,3*\hei) {};

\draw[e] (11) to[in =0, out =0] (12);
\draw[e] (13) to[in =0, out =0] (14);

\draw[e] (13) to[in =180, out =180] (14);
\draw[e] (11) to[out = 180, in = 0] (0h1);
\draw[e] (12) to[out = 180, in = 0] (0h2);

\draw[very thick] (0* \wid,-0.25*\hei) -- (0* \wid,3.25*\hei);
\end{tikzpicture}
$)$
$\oplus$
$C_*($
\begin{tikzpicture}[x=1.5cm,y=-.5cm,baseline=-0.7cm]

\def\wid{\standardWidth}
\def\hei{\standardHeight}
\def\nodesize{3}
\def\ang{90}

\node[] (0h1) at (-0.5* \wid,1*\hei) {};
\node[] (0h2) at (-0.5* \wid,2*\hei) {};
\node[v, minimum size=\nodesize] (11) at (0* \wid,0*\hei) {};
\node[v, minimum size=\nodesize] (12) at (0* \wid,1*\hei) {};
\node[v, minimum size=\nodesize] (13) at (0* \wid,2*\hei) {};
\node[v, minimum size=\nodesize] (14) at (0* \wid,3*\hei) {};

\draw[e] (14) to[in =0, out =0] (13);
\draw[e] (12) to[in =0, out =0] (11);

\draw[e] (12) to[in =180, out =180] (11);
\draw[e] (14) to[out = 180, in = 0] (0h2);
\draw[e] (13) to[out = 180, in = 0] (0h1);

\draw[very thick] (0* \wid,-0.25*\hei) -- (0* \wid,3.25*\hei);
\end{tikzpicture}
$)$.
\end{center}
\item We have a direct sum decomposition \begin{center}
$\faktor{C_*[2,0]^{\mathrm{out}}}{F^1(C_*[2,0]^{\mathrm{out}})}$ 
\quad
$= C_*($
\begin{tikzpicture}[x=1.5cm,y=-.5cm,baseline=-0.7cm]

\def\wid{\standardWidth}
\def\hei{\standardHeight}
\def\nodesize{3}
\def\ang{90}

\node[] (0h1) at (-0.5* \wid,1*\hei) {};
\node[] (0h2) at (-0.5* \wid,2*\hei) {};
\node[v, minimum size=\nodesize] (11) at (0* \wid,0*\hei) {};
\node[v, minimum size=\nodesize] (12) at (0* \wid,1*\hei) {};
\node[v, minimum size=\nodesize] (13) at (0* \wid,2*\hei) {};
\node[v, minimum size=\nodesize] (14) at (0* \wid,3*\hei) {};
\node[] (1h1) at (0.5* \wid,1*\hei) {};
\node[] (1h2) at (0.5* \wid,2*\hei) {};

\draw[e] (11) to[in =0, out =180] (0h1);
\draw[e] (12) to[in =0, out =180] (0h2);
\draw[e] (13) to[in =180, out =180] (14);

\draw[e] (13) to[in =0, out =0] (14);
\draw[e] (11) to[out = 0, in = 180] (1h1);
\draw[e] (12) to[out = 0, in = 180] (1h2);

\draw[very thick] (0* \wid,-0.25*\hei) -- (0* \wid,3.25*\hei);
\end{tikzpicture}
$)$
$\oplus$
$C_*($
\begin{tikzpicture}[x=1.5cm,y=-.5cm,baseline=-0.7cm]

\def\wid{\standardWidth}
\def\hei{\standardHeight}
\def\nodesize{3}
\def\ang{90}

\node[] (0h1) at (-0.5* \wid,1*\hei) {};
\node[] (0h2) at (-0.5* \wid,2*\hei) {};
\node[v, minimum size=\nodesize] (11) at (0* \wid,0*\hei) {};
\node[v, minimum size=\nodesize] (12) at (0* \wid,1*\hei) {};
\node[v, minimum size=\nodesize] (13) at (0* \wid,2*\hei) {};
\node[v, minimum size=\nodesize] (14) at (0* \wid,3*\hei) {};
\node[] (1h1) at (0.5* \wid,1*\hei) {};
\node[] (1h2) at (0.5* \wid,2*\hei) {};

\draw[e] (12) to[in =180, out =180] (11);
\draw[e] (13) to[in =0, out =180] (0h1);
\draw[e] (14) to[in =0, out =180] (0h2);

\draw[e] (12) to[in =0, out =0] (11);
\draw[e] (14) to[out = 0, in = 180] (1h2);
\draw[e] (13) to[out = 0, in = 180] (1h1);

\draw[very thick] (0* \wid,-0.25*\hei) -- (0* \wid,3.25*\hei);
\end{tikzpicture}
$)$
\end{center}

\begin{center}
$\oplus$
$C_*($
\begin{tikzpicture}[x=1.5cm,y=-.5cm,baseline=-0.7cm]

\def\wid{\standardWidth}
\def\hei{\standardHeight}
\def\nodesize{3}
\def\ang{90}

\node[] (0h1) at (-0.5* \wid,1*\hei) {};
\node[] (0h2) at (-0.5* \wid,2*\hei) {};
\node[v, minimum size=\nodesize] (11) at (0* \wid,0*\hei) {};
\node[v, minimum size=\nodesize] (12) at (0* \wid,1*\hei) {};
\node[v, minimum size=\nodesize] (13) at (0* \wid,2*\hei) {};
\node[v, minimum size=\nodesize] (14) at (0* \wid,3*\hei) {};
\node[] (1h1) at (0.5* \wid,1*\hei) {};
\node[] (1h2) at (0.5* \wid,2*\hei) {};

\draw[e] (13) to[in =0, out =180] (0h1);
\draw[e] (14) to[in =0, out =180] (0h2);
\draw[e] (11) to[in =180, out =180] (12);

\draw[e] (13) to[in =0, out =0] (14);
\draw[e] (11) to[out = 0, in = 180] (1h1);
\draw[e] (12) to[out = 0, in = 180] (1h2);

\draw[very thick] (0* \wid,-0.25*\hei) -- (0* \wid,3.25*\hei);
\end{tikzpicture}
$)$
$\oplus$
$C_*($
\begin{tikzpicture}[x=1.5cm,y=-.5cm,baseline=-0.7cm]

\def\wid{\standardWidth}
\def\hei{\standardHeight}
\def\nodesize{3}
\def\ang{90}

\node[] (0h1) at (-0.5* \wid,1*\hei) {};
\node[] (0h2) at (-0.5* \wid,2*\hei) {};
\node[v, minimum size=\nodesize] (11) at (0* \wid,0*\hei) {};
\node[v, minimum size=\nodesize] (12) at (0* \wid,1*\hei) {};
\node[v, minimum size=\nodesize] (13) at (0* \wid,2*\hei) {};
\node[v, minimum size=\nodesize] (14) at (0* \wid,3*\hei) {};
\node[] (1h1) at (0.5* \wid,1*\hei) {};
\node[] (1h2) at (0.5* \wid,2*\hei) {};

\draw[e] (13) to[in =180, out =180] (14);
\draw[e] (11) to[in =0, out =180] (0h1);
\draw[e] (12) to[in =0, out =180] (0h2);

\draw[e] (12) to[in =0, out =0] (11);
\draw[e] (14) to[out = 0, in = 180] (1h2);
\draw[e] (13) to[out = 0, in = 180] (1h1);

\draw[very thick] (0* \wid,-0.25*\hei) -- (0* \wid,3.25*\hei);
\end{tikzpicture}
$)$. \qed
\end{center}
\end{itemize}
\end{lemma}

We have now reduced \cref{thm:mainTechnical} (\ref{part:2}), that is, calculating the homology of $C_*[2,0]$, to calculating the homology of $C_*[2,0]^{\mathrm{out}}$. We do this now.

\begin{proposition}
    The homology of $C_*[2,0]^{\mathrm{out}}$ consists of a single copy of $R$ in degree 3, generated by the class 
    
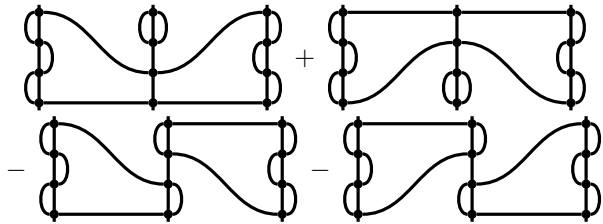
\begin{center}
    \begin{tikzpicture}[x=1.5cm,y=-.5cm,baseline=-0.7cm]
        \def\wid{\standardWidth}
            \def\hei{\standardHeight}
            \def\nodesize{3}

            \node[v, minimum size=\nodesize] (01) at (-1* \wid,0*\hei) {};
            \node[v, minimum size=\nodesize] (02) at (-1* \wid,1*\hei) {};
            \node[v, minimum size=\nodesize] (03) at (-1* \wid,2*\hei) {};
            \node[v, minimum size=\nodesize] (04) at (-1* \wid,3*\hei) {};

            \node[v, minimum size=\nodesize] (11) at (0* \wid,0*\hei) {};
            \node[v, minimum size=\nodesize] (12) at (0* \wid,1*\hei) {};
            \node[v, minimum size=\nodesize] (13) at (0* \wid,2*\hei) {};
            \node[v, minimum size=\nodesize] (14) at (0* \wid,3*\hei) {};

            \node[v, minimum size=\nodesize] (21) at (1* \wid,0*\hei) {};
            \node[v, minimum size=\nodesize] (22) at (1* \wid,1*\hei) {};
            \node[v, minimum size=\nodesize] (23) at (1* \wid,2*\hei) {};
            \node[v, minimum size=\nodesize] (24) at (1* \wid,3*\hei) {};

            \draw[e] (02) to[in =0, out =0] (03);
            \draw[e] (01) to[in =180, out =180] (02);
            \draw[e] (03) to[in =180, out =180] (04);
            \draw[e] (11) to[in =180, out =180] (12);
            \draw[e] (13) to[out = 180, in = 0] (01);
            \draw[e] (14) to[out = 180, in = 0] (04);

            \draw[e] (11) to[in =0, out =0] (12);
            \draw[e] (23) to[in =180, out =180] (22);
            \draw[e] (21) to[in =0, out =0] (22);
            \draw[e] (23) to[in =0, out =0] (24);
            \draw[e] (13) to[out = 0, in = 180] (21);
            \draw[e] (14) to[out = 0, in = 180] (24);

            \draw[very thick] (0* \wid,-0.25*\hei) -- (0* \wid,3.25*\hei);
            \draw[very thick] (-1* \wid,-0.25*\hei) -- (-1* \wid,3.25*\hei);
            \draw[very thick] (1* \wid,-0.25*\hei) -- (1* \wid,3.25*\hei);
    \end{tikzpicture}
    $+$ 
    \begin{tikzpicture}[x=1.5cm,y=-.5cm,baseline=-0.7cm]
       \def\wid{\standardWidth}
            \def\hei{\standardHeight}
            \def\nodesize{3}

            \node[v, minimum size=\nodesize] (01) at (-1* \wid,0*\hei) {};
            \node[v, minimum size=\nodesize] (02) at (-1* \wid,1*\hei) {};
            \node[v, minimum size=\nodesize] (03) at (-1* \wid,2*\hei) {};
            \node[v, minimum size=\nodesize] (04) at (-1* \wid,3*\hei) {};

            \node[v, minimum size=\nodesize] (11) at (0* \wid,0*\hei) {};
            \node[v, minimum size=\nodesize] (12) at (0* \wid,1*\hei) {};
            \node[v, minimum size=\nodesize] (13) at (0* \wid,2*\hei) {};
            \node[v, minimum size=\nodesize] (14) at (0* \wid,3*\hei) {};

            \node[v, minimum size=\nodesize] (21) at (1* \wid,0*\hei) {};
            \node[v, minimum size=\nodesize] (22) at (1* \wid,1*\hei) {};
            \node[v, minimum size=\nodesize] (23) at (1* \wid,2*\hei) {};
            \node[v, minimum size=\nodesize] (24) at (1* \wid,3*\hei) {};

            \draw[e] (02) to[in =0, out =0] (03);
            \draw[e] (01) to[in =180, out =180] (02);
            \draw[e] (03) to[in =180, out =180] (04);
            \draw[e] (13) to[in =180, out =180] (14);
            \draw[e] (11) to[out = 180, in = 0] (01);
            \draw[e] (12) to[out = 180, in = 0] (04);

            \draw[e] (13) to[in =0, out =0] (14);
            \draw[e] (23) to[in =180, out =180] (22);
            \draw[e] (21) to[in =0, out =0] (22);
            \draw[e] (23) to[in =0, out =0] (24);
            \draw[e] (11) to[out = 0, in = 180] (21);
            \draw[e] (12) to[out = 0, in = 180] (24);

            \draw[very thick] (0* \wid,-0.25*\hei) -- (0* \wid,3.25*\hei);
            \draw[very thick] (-1* \wid,-0.25*\hei) -- (-1* \wid,3.25*\hei);
            \draw[very thick] (1* \wid,-0.25*\hei) -- (1* \wid,3.25*\hei); 
    \end{tikzpicture}
    \\
    $-$
    \begin{tikzpicture}[x=1.5cm,y=-.5cm,baseline=-0.7cm]
       \def\wid{\standardWidth}
            \def\hei{\standardHeight}
            \def\nodesize{3}

            \node[v, minimum size=\nodesize] (01) at (-1* \wid,0*\hei) {};
            \node[v, minimum size=\nodesize] (02) at (-1* \wid,1*\hei) {};
            \node[v, minimum size=\nodesize] (03) at (-1* \wid,2*\hei) {};
            \node[v, minimum size=\nodesize] (04) at (-1* \wid,3*\hei) {};

            \node[v, minimum size=\nodesize] (11) at (0* \wid,0*\hei) {};
            \node[v, minimum size=\nodesize] (12) at (0* \wid,1*\hei) {};
            \node[v, minimum size=\nodesize] (13) at (0* \wid,2*\hei) {};
            \node[v, minimum size=\nodesize] (14) at (0* \wid,3*\hei) {};

            \node[v, minimum size=\nodesize] (21) at (1* \wid,0*\hei) {};
            \node[v, minimum size=\nodesize] (22) at (1* \wid,1*\hei) {};
            \node[v, minimum size=\nodesize] (23) at (1* \wid,2*\hei) {};
            \node[v, minimum size=\nodesize] (24) at (1* \wid,3*\hei) {};

            \draw[e] (02) to[in =0, out =0] (03);
            \draw[e] (01) to[in =180, out =180] (02);
            \draw[e] (03) to[in =180, out =180] (04);
            \draw[e] (11) to[in =180, out =180] (12);
            \draw[e] (13) to[out = 180, in = 0] (01);
            \draw[e] (14) to[out = 180, in = 0] (04);

            \draw[e] (13) to[in =0, out =0] (14);
            \draw[e] (23) to[in =180, out =180] (22);
            \draw[e] (21) to[in =0, out =0] (22);
            \draw[e] (23) to[in =0, out =0] (24);
            \draw[e] (11) to[out = 0, in = 180] (21);
            \draw[e] (12) to[out = 0, in = 180] (24);

            \draw[very thick] (0* \wid,-0.25*\hei) -- (0* \wid,3.25*\hei);
            \draw[very thick] (-1* \wid,-0.25*\hei) -- (-1* \wid,3.25*\hei);
            \draw[very thick] (1* \wid,-0.25*\hei) -- (1* \wid,3.25*\hei); 
    \end{tikzpicture}
    $-$
    \begin{tikzpicture}[x=1.5cm,y=-.5cm,baseline=-0.7cm]
       \def\wid{\standardWidth}
            \def\hei{\standardHeight}
            \def\nodesize{3}

            \node[v, minimum size=\nodesize] (01) at (-1* \wid,0*\hei) {};
            \node[v, minimum size=\nodesize] (02) at (-1* \wid,1*\hei) {};
            \node[v, minimum size=\nodesize] (03) at (-1* \wid,2*\hei) {};
            \node[v, minimum size=\nodesize] (04) at (-1* \wid,3*\hei) {};

            \node[v, minimum size=\nodesize] (11) at (0* \wid,0*\hei) {};
            \node[v, minimum size=\nodesize] (12) at (0* \wid,1*\hei) {};
            \node[v, minimum size=\nodesize] (13) at (0* \wid,2*\hei) {};
            \node[v, minimum size=\nodesize] (14) at (0* \wid,3*\hei) {};

            \node[v, minimum size=\nodesize] (21) at (1* \wid,0*\hei) {};
            \node[v, minimum size=\nodesize] (22) at (1* \wid,1*\hei) {};
            \node[v, minimum size=\nodesize] (23) at (1* \wid,2*\hei) {};
            \node[v, minimum size=\nodesize] (24) at (1* \wid,3*\hei) {};

            \draw[e] (02) to[in =0, out =0] (03);
            \draw[e] (01) to[in =180, out =180] (02);
            \draw[e] (03) to[in =180, out =180] (04);
            \draw[e] (13) to[in =180, out =180] (14);
            \draw[e] (11) to[out = 180, in = 0] (01);
            \draw[e] (12) to[out = 180, in = 0] (04);

            \draw[e] (11) to[in =0, out =0] (12);
            \draw[e] (23) to[in =180, out =180] (22);
            \draw[e] (21) to[in =0, out =0] (22);
            \draw[e] (23) to[in =0, out =0] (24);
            \draw[e] (13) to[out = 0, in = 180] (21);
            \draw[e] (14) to[out = 0, in = 180] (24);

            \draw[very thick] (0* \wid,-0.25*\hei) -- (0* \wid,3.25*\hei);
            \draw[very thick] (-1* \wid,-0.25*\hei) -- (-1* \wid,3.25*\hei);
            \draw[very thick] (1* \wid,-0.25*\hei) -- (1* \wid,3.25*\hei); 
    \end{tikzpicture}
    \end{center}
\end{proposition}

\begin{proof}
Consider the spectral sequence associated to the filtration $F^p(C_*[2,0]^{\mathrm{out}})$, which converges to the homology of $C_*[2,0]^{\mathrm{out}}$. \cref{lem:2unnestedQuots} and \cref{cor:Cphomology} tell us that the $E^1$-page $$E^1_{p,q} = H_{p+q}(\faktor{F^p(C_*[2,0]^{\mathrm{out}})}{F^{p-1}(C_*[2,0]^{\mathrm{out}})})$$ is concentrated in the $q=1$ row, in the entries $(0,1)$, $(1,1)$, and $(2,1)$, and looks as in \cref{fig:2unnested} (a choice of generator of each copy of $R$ is indicated).

  \begin{figure}[h]
		\begin{tikzpicture}[scale=0.7]
            \def\colsp{3.8}
            \def\extrarowsp{1}
			\draw[line width=0.2cm, gray!20, <->] (-0.5*\colsp,5+\extrarowsp)--node[black, above, pos=0] {$q$}(-0.5*\colsp,-1)--(4*\colsp,-1) node[black, right] {$p$};
			\draw (-0.5*\colsp,0) node {$0$};	
			\draw (-0.5*\colsp,1+0.5*\extrarowsp) node {$1$};
			\draw (-0.5*\colsp,2+\extrarowsp) node {$2$};
            \draw (-0.5*\colsp,3+\extrarowsp) node {$3$};
            \draw (0,-1) node {$0$};
			\draw (1*\colsp,-1) node {$1$};	
            \draw (2*\colsp,-1) node {$2$};
			\draw (3*\colsp,-1) node {$3$};

            \draw (1*\colsp,1+\extrarowsp) node {$\oplus$};
            \draw (0.5*\colsp,1) node {$\oplus$};
            \draw (1*\colsp,1) node {$\oplus$};

            \draw (2.2*\colsp,1+\extrarowsp) node {$\oplus$};
            \draw (1.63*\colsp,1) node {$\oplus$};
            \draw (2.2*\colsp,1) node {$\oplus$};

            \foreach \x in {0,1,2,3} \foreach \y in {0,2+\extrarowsp,3+\extrarowsp} \draw (\x*\colsp,\y) node[black!80,scale=1.2] {$0$};
            \draw (3*\colsp,1+0.5*\extrarowsp) node[black!80,scale=1.2] {$0$};

			\foreach \x in {(3.5*\colsp,-1), (3.5*\colsp,0), (3.5*\colsp,1+0.5*\extrarowsp), (3.5*\colsp,2+\extrarowsp),(3.5*\colsp,3+\extrarowsp) }
			\draw \x node {$\cdots$};	
            \foreach \x in {(-0.5*\colsp,4+\extrarowsp), (0*\colsp,4+\extrarowsp), (1*\colsp,4+\extrarowsp), (2*\colsp,4+\extrarowsp),(3*\colsp,4+\extrarowsp) }
			\draw \x node {$\vdots$};	

            \foreach \x in {(0.3*\colsp,1+0.5*\extrarowsp), (1.6*\colsp,1+0.5*\extrarowsp) }
			\draw \x node {$\xleftarrow{}$};	

        \begin{scope}[xshift=0, yshift = 0.7cm+0.5*\extrarowsp cm, scale = 0.3] 
            \def\wid{\standardWidth}
            \def\hei{\standardHeight}
            \def\nodesize{3}
            \def\ang{90}

            \node[v, minimum size=\nodesize] (11) at (0* \wid,0*\hei) {};
            \node[v, minimum size=\nodesize] (12) at (0* \wid,1*\hei) {};
            \node[v, minimum size=\nodesize] (13) at (0* \wid,2*\hei) {};
            \node[v, minimum size=\nodesize] (14) at (0* \wid,3*\hei) {};

            \draw[e] (11) to[in =180, out =180] (12);
            \draw[e] (13) to[in =180, out =180] (14);

            \draw[e] (11) to[in =0, out =0] (12);
            \draw[e] (13) to[out = 0, in = 0] (14);

            \draw[thick] (0* \wid,-0.25*\hei) -- (0* \wid,3.25*\hei);
        \end{scope}

        \begin{scope}[xshift=0.85*\colsp cm, yshift = 0.7cm+\extrarowsp cm, scale = 0.3] 
            \def\wid{\standardWidth}
            \def\hei{\standardHeight}
            \def\nodesize{3}
            \def\ang{90}

            \node[v, minimum size=\nodesize] (01) at (-2* \wid,0*\hei) {};
            \node[v, minimum size=\nodesize] (02) at (-2* \wid,1*\hei) {};
            \node[v, minimum size=\nodesize] (03) at (-2* \wid,2*\hei) {};
            \node[v, minimum size=\nodesize] (04) at (-2* \wid,3*\hei) {};

            \node[v, minimum size=\nodesize] (11) at (0* \wid,0*\hei) {};
            \node[v, minimum size=\nodesize] (12) at (0* \wid,1*\hei) {};
            \node[v, minimum size=\nodesize] (13) at (0* \wid,2*\hei) {};
            \node[v, minimum size=\nodesize] (14) at (0* \wid,3*\hei) {};

            \draw[e] (02) to[in =0, out =0] (03);
            \draw[e] (01) to[in =180, out =180] (02);
            \draw[e] (03) to[in =180, out =180] (04);

            \draw[e] (11) to[in =0, out =0] (12);
            \draw[e] (13) to[in =0, out =0] (14);

            \draw[e] (11) to[in =180, out =180] (12);
            \draw[e] (13) to[out = 180, in = 0] (01);
            \draw[e] (14) to[out = 180, in = 0] (04);

            \draw[thick] (0* \wid,-0.25*\hei) -- (0* \wid,3.25*\hei);
            \draw[thick] (-2* \wid,-0.25*\hei) -- (-2* \wid,3.25*\hei);
        \end{scope}

        \begin{scope}[xshift=1.35*\colsp cm, yshift = 0.7cm+\extrarowsp cm, scale = 0.3] 
            \def\wid{\standardWidth}
            \def\hei{\standardHeight}
            \def\nodesize{3}
            \def\ang{90}

            \node[v, minimum size=\nodesize] (01) at (-2* \wid,0*\hei) {};
            \node[v, minimum size=\nodesize] (02) at (-2* \wid,1*\hei) {};
            \node[v, minimum size=\nodesize] (03) at (-2* \wid,2*\hei) {};
            \node[v, minimum size=\nodesize] (04) at (-2* \wid,3*\hei) {};

            \node[v, minimum size=\nodesize] (11) at (0* \wid,0*\hei) {};
            \node[v, minimum size=\nodesize] (12) at (0* \wid,1*\hei) {};
            \node[v, minimum size=\nodesize] (13) at (0* \wid,2*\hei) {};
            \node[v, minimum size=\nodesize] (14) at (0* \wid,3*\hei) {};

            \draw[e] (02) to[in =0, out =0] (03);
            \draw[e] (01) to[in =180, out =180] (02);
            \draw[e] (03) to[in =180, out =180] (04);

            \draw[e] (14) to[in =0, out =0] (13);
            \draw[e] (12) to[in =0, out =0] (11);

            \draw[e] (14) to[in =180, out =180] (13);
            \draw[e] (12) to[out = 180, in = 0] (04);
            \draw[e] (11) to[out = 180, in = 0] (01);

            \draw[thick] (0* \wid,-0.25*\hei) -- (0* \wid,3.25*\hei);
            \draw[thick] (-2* \wid,-0.25*\hei) -- (-2* \wid,3.25*\hei);
        \end{scope}

        \begin{scope}[xshift=0.85*\colsp cm, yshift = 0.7cm, scale = 0.3] 
            \def\wid{\standardWidth}
            \def\hei{\standardHeight}
            \def\nodesize{3}
            \def\ang{90}

            \node[v, minimum size=\nodesize] (01) at (-2* \wid,0*\hei) {};
            \node[v, minimum size=\nodesize] (02) at (-2* \wid,1*\hei) {};
            \node[v, minimum size=\nodesize] (03) at (-2* \wid,2*\hei) {};
            \node[v, minimum size=\nodesize] (04) at (-2* \wid,3*\hei) {};

            \node[v, minimum size=\nodesize] (11) at (0* \wid,0*\hei) {};
            \node[v, minimum size=\nodesize] (12) at (0* \wid,1*\hei) {};
            \node[v, minimum size=\nodesize] (13) at (0* \wid,2*\hei) {};
            \node[v, minimum size=\nodesize] (14) at (0* \wid,3*\hei) {};

            \draw[e] (13) to[in =180, out =180] (12);
            \draw[e] (11) to[in =0, out =0] (12);
            \draw[e] (13) to[in =0, out =0] (14);

            \draw[e] (01) to[in =180, out =180] (02);
            \draw[e] (03) to[in =180, out =180] (04);

            \draw[e] (01) to[in =0, out =0] (02);
            \draw[e] (03) to[out = 0, in = 180] (11);
            \draw[e] (04) to[out = 0, in = 180] (14);

            \draw[thick] (0* \wid,-0.25*\hei) -- (0* \wid,3.25*\hei);
            \draw[thick] (-2* \wid,-0.25*\hei) -- (-2* \wid,3.25*\hei);
        \end{scope}

        \begin{scope}[xshift=1.35*\colsp cm, yshift = 0.7cm, scale = 0.3] 
            \def\wid{\standardWidth}
            \def\hei{\standardHeight}
            \def\nodesize{3}
            \def\ang{90}

            \node[v, minimum size=\nodesize] (01) at (-2* \wid,0*\hei) {};
            \node[v, minimum size=\nodesize] (02) at (-2* \wid,1*\hei) {};
            \node[v, minimum size=\nodesize] (03) at (-2* \wid,2*\hei) {};
            \node[v, minimum size=\nodesize] (04) at (-2* \wid,3*\hei) {};

            \node[v, minimum size=\nodesize] (11) at (0* \wid,0*\hei) {};
            \node[v, minimum size=\nodesize] (12) at (0* \wid,1*\hei) {};
            \node[v, minimum size=\nodesize] (13) at (0* \wid,2*\hei) {};
            \node[v, minimum size=\nodesize] (14) at (0* \wid,3*\hei) {};

            \draw[e] (13) to[in =180, out =180] (12);
            \draw[e] (11) to[in =0, out =0] (12);
            \draw[e] (13) to[in =0, out =0] (14);

            \draw[e] (04) to[in =180, out =180] (03);
            \draw[e] (02) to[in =180, out =180] (01);

            \draw[e] (04) to[in =0, out =0] (03);
            \draw[e] (02) to[out = 0, in = 180] (14);
            \draw[e] (01) to[out = 0, in = 180] (11);

            \draw[thick] (0* \wid,-0.25*\hei) -- (0* \wid,3.25*\hei);
            \draw[thick] (-2* \wid,-0.25*\hei) -- (-2* \wid,3.25*\hei);
        \end{scope}

        \begin{scope}[xshift=1.9*\colsp cm, yshift = 0.7cm+\extrarowsp cm, scale = 0.3] 
            \def\wid{\standardWidth}
            \def\hei{\standardHeight}
            \def\nodesize{3}

            \node[v, minimum size=\nodesize] (01) at (-2* \wid,0*\hei) {};
            \node[v, minimum size=\nodesize] (02) at (-2* \wid,1*\hei) {};
            \node[v, minimum size=\nodesize] (03) at (-2* \wid,2*\hei) {};
            \node[v, minimum size=\nodesize] (04) at (-2* \wid,3*\hei) {};

            \node[v, minimum size=\nodesize] (11) at (0* \wid,0*\hei) {};
            \node[v, minimum size=\nodesize] (12) at (0* \wid,1*\hei) {};
            \node[v, minimum size=\nodesize] (13) at (0* \wid,2*\hei) {};
            \node[v, minimum size=\nodesize] (14) at (0* \wid,3*\hei) {};

            \node[v, minimum size=\nodesize] (21) at (2* \wid,0*\hei) {};
            \node[v, minimum size=\nodesize] (22) at (2* \wid,1*\hei) {};
            \node[v, minimum size=\nodesize] (23) at (2* \wid,2*\hei) {};
            \node[v, minimum size=\nodesize] (24) at (2* \wid,3*\hei) {};

            \draw[e] (02) to[in =0, out =0] (03);
            \draw[e] (01) to[in =180, out =180] (02);
            \draw[e] (03) to[in =180, out =180] (04);
            \draw[e] (11) to[in =180, out =180] (12);
            \draw[e] (13) to[out = 180, in = 0] (01);
            \draw[e] (14) to[out = 180, in = 0] (04);

            \draw[e] (11) to[in =0, out =0] (12);
            \draw[e] (23) to[in =180, out =180] (22);
            \draw[e] (21) to[in =0, out =0] (22);
            \draw[e] (23) to[in =0, out =0] (24);
            \draw[e] (13) to[out = 0, in = 180] (21);
            \draw[e] (14) to[out = 0, in = 180] (24);

            \draw[thick] (0* \wid,-0.25*\hei) -- (0* \wid,3.25*\hei);
            \draw[thick] (-2* \wid,-0.25*\hei) -- (-2* \wid,3.25*\hei);
            \draw[thick] (2* \wid,-0.25*\hei) -- (2* \wid,3.25*\hei);
        \end{scope}

        \begin{scope}[xshift=2.5*\colsp cm, yshift = 0.7cm+\extrarowsp cm, scale = 0.3] 
            \def\wid{\standardWidth}
            \def\hei{\standardHeight}
            \def\nodesize{3}

            \node[v, minimum size=\nodesize] (01) at (-2* \wid,0*\hei) {};
            \node[v, minimum size=\nodesize] (02) at (-2* \wid,1*\hei) {};
            \node[v, minimum size=\nodesize] (03) at (-2* \wid,2*\hei) {};
            \node[v, minimum size=\nodesize] (04) at (-2* \wid,3*\hei) {};

            \node[v, minimum size=\nodesize] (11) at (0* \wid,0*\hei) {};
            \node[v, minimum size=\nodesize] (12) at (0* \wid,1*\hei) {};
            \node[v, minimum size=\nodesize] (13) at (0* \wid,2*\hei) {};
            \node[v, minimum size=\nodesize] (14) at (0* \wid,3*\hei) {};

            \node[v, minimum size=\nodesize] (21) at (2* \wid,0*\hei) {};
            \node[v, minimum size=\nodesize] (22) at (2* \wid,1*\hei) {};
            \node[v, minimum size=\nodesize] (23) at (2* \wid,2*\hei) {};
            \node[v, minimum size=\nodesize] (24) at (2* \wid,3*\hei) {};

            \draw[e] (02) to[in =0, out =0] (03);
            \draw[e] (01) to[in =180, out =180] (02);
            \draw[e] (03) to[in =180, out =180] (04);
            \draw[e] (11) to[in =180, out =180] (12);
            \draw[e] (13) to[out = 180, in = 0] (01);
            \draw[e] (14) to[out = 180, in = 0] (04);

            \draw[e] (13) to[in =0, out =0] (14);
            \draw[e] (23) to[in =180, out =180] (22);
            \draw[e] (21) to[in =0, out =0] (22);
            \draw[e] (23) to[in =0, out =0] (24);
            \draw[e] (11) to[out = 0, in = 180] (21);
            \draw[e] (12) to[out = 0, in = 180] (24);

            \draw[thick] (0* \wid,-0.25*\hei) -- (0* \wid,3.25*\hei);
            \draw[thick] (-2* \wid,-0.25*\hei) -- (-2* \wid,3.25*\hei);
            \draw[thick] (2* \wid,-0.25*\hei) -- (2* \wid,3.25*\hei);
        \end{scope}

        \begin{scope}[xshift=1.9*\colsp cm, yshift = 0.7cm, scale = 0.3] 
            \def\wid{\standardWidth}
            \def\hei{\standardHeight}
            \def\nodesize{3}

            \node[v, minimum size=\nodesize] (01) at (-2* \wid,0*\hei) {};
            \node[v, minimum size=\nodesize] (02) at (-2* \wid,1*\hei) {};
            \node[v, minimum size=\nodesize] (03) at (-2* \wid,2*\hei) {};
            \node[v, minimum size=\nodesize] (04) at (-2* \wid,3*\hei) {};

            \node[v, minimum size=\nodesize] (11) at (0* \wid,0*\hei) {};
            \node[v, minimum size=\nodesize] (12) at (0* \wid,1*\hei) {};
            \node[v, minimum size=\nodesize] (13) at (0* \wid,2*\hei) {};
            \node[v, minimum size=\nodesize] (14) at (0* \wid,3*\hei) {};

            \node[v, minimum size=\nodesize] (21) at (2* \wid,0*\hei) {};
            \node[v, minimum size=\nodesize] (22) at (2* \wid,1*\hei) {};
            \node[v, minimum size=\nodesize] (23) at (2* \wid,2*\hei) {};
            \node[v, minimum size=\nodesize] (24) at (2* \wid,3*\hei) {};

            \draw[e] (02) to[in =0, out =0] (03);
            \draw[e] (01) to[in =180, out =180] (02);
            \draw[e] (03) to[in =180, out =180] (04);
            \draw[e] (13) to[in =180, out =180] (14);
            \draw[e] (11) to[out = 180, in = 0] (01);
            \draw[e] (12) to[out = 180, in = 0] (04);

            \draw[e] (13) to[in =0, out =0] (14);
            \draw[e] (23) to[in =180, out =180] (22);
            \draw[e] (21) to[in =0, out =0] (22);
            \draw[e] (23) to[in =0, out =0] (24);
            \draw[e] (11) to[out = 0, in = 180] (21);
            \draw[e] (12) to[out = 0, in = 180] (24);

            \draw[thick] (0* \wid,-0.25*\hei) -- (0* \wid,3.25*\hei);
            \draw[thick] (-2* \wid,-0.25*\hei) -- (-2* \wid,3.25*\hei);
            \draw[thick] (2* \wid,-0.25*\hei) -- (2* \wid,3.25*\hei);
        \end{scope}

        \begin{scope}[xshift=2.5*\colsp cm, yshift = 0.7cm, scale = 0.3] 
            \def\wid{\standardWidth}
            \def\hei{\standardHeight}
            \def\nodesize{3}

            \node[v, minimum size=\nodesize] (01) at (-2* \wid,0*\hei) {};
            \node[v, minimum size=\nodesize] (02) at (-2* \wid,1*\hei) {};
            \node[v, minimum size=\nodesize] (03) at (-2* \wid,2*\hei) {};
            \node[v, minimum size=\nodesize] (04) at (-2* \wid,3*\hei) {};

            \node[v, minimum size=\nodesize] (11) at (0* \wid,0*\hei) {};
            \node[v, minimum size=\nodesize] (12) at (0* \wid,1*\hei) {};
            \node[v, minimum size=\nodesize] (13) at (0* \wid,2*\hei) {};
            \node[v, minimum size=\nodesize] (14) at (0* \wid,3*\hei) {};

            \node[v, minimum size=\nodesize] (21) at (2* \wid,0*\hei) {};
            \node[v, minimum size=\nodesize] (22) at (2* \wid,1*\hei) {};
            \node[v, minimum size=\nodesize] (23) at (2* \wid,2*\hei) {};
            \node[v, minimum size=\nodesize] (24) at (2* \wid,3*\hei) {};

            \draw[e] (02) to[in =0, out =0] (03);
            \draw[e] (01) to[in =180, out =180] (02);
            \draw[e] (03) to[in =180, out =180] (04);
            \draw[e] (13) to[in =180, out =180] (14);
            \draw[e] (11) to[out = 180, in = 0] (01);
            \draw[e] (12) to[out = 180, in = 0] (04);

            \draw[e] (11) to[in =0, out =0] (12);
            \draw[e] (23) to[in =180, out =180] (22);
            \draw[e] (21) to[in =0, out =0] (22);
            \draw[e] (23) to[in =0, out =0] (24);
            \draw[e] (13) to[out = 0, in = 180] (21);
            \draw[e] (14) to[out = 0, in = 180] (24);

            \draw[thick] (0* \wid,-0.25*\hei) -- (0* \wid,3.25*\hei);
            \draw[thick] (-2* \wid,-0.25*\hei) -- (-2* \wid,3.25*\hei);
            \draw[thick] (2* \wid,-0.25*\hei) -- (2* \wid,3.25*\hei);
        \end{scope}
			
		\end{tikzpicture}
\caption{The page $E^1_{p,q}$ of the spectral sequence computing the homology of $C_*[2,0]^{\mathrm{out}}$.}
\label{fig:2unnested}
\end{figure}
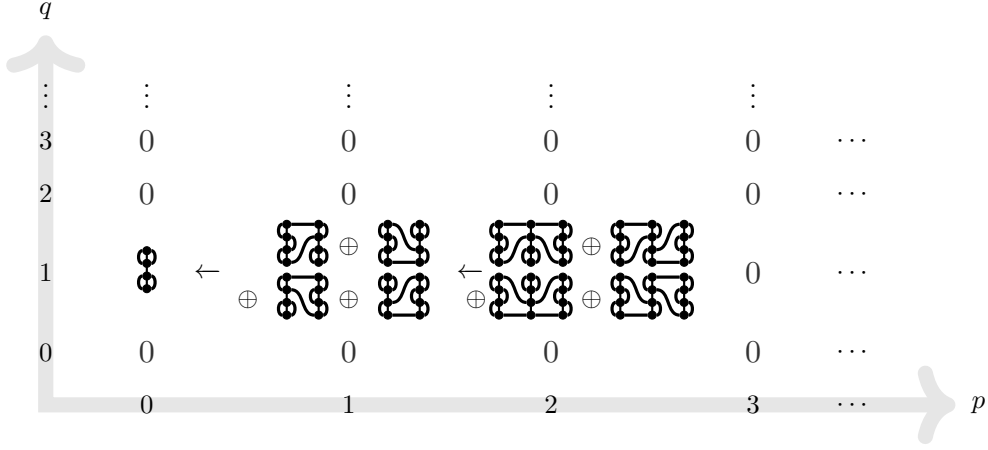
This time it does not happen that everything cancels under $d^1$, but direct computation shows that the only class surviving the $d^1$-differential is the one given in the theorem statement. The result follows.
\end{proof}

\subsection{Part (\ref{part:igeq3})}

We now show that the homology of the complex with $i$ loops and no dividers vanishes when $i \geq 3$, establishing the last part of \cref{thm:mainTechnical}.

Consider the pivots, i.e., the letters from \cref{lem:2noAlphabet}. Of these letters, we call
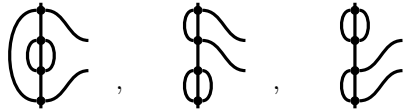
\begin{center}
\begin{tikzpicture}[x=1.5cm,y=-.5cm,baseline=-1.05cm]

\def\wid{\standardWidth}
\def\hei{\standardHeight}
\def\nodesize{3}
\def\ang{90}

\node[] (0h1) at (0.5* \wid,1*\hei) {};
\node[] (0h2) at (0.5* \wid,2*\hei) {};
\node[v, minimum size=\nodesize] (11) at (0* \wid,0*\hei) {};
\node[v, minimum size=\nodesize] (12) at (0* \wid,1*\hei) {};
\node[v, minimum size=\nodesize] (13) at (0* \wid,2*\hei) {};
\node[v, minimum size=\nodesize] (14) at (0* \wid,3*\hei) {};
\node[] (1h1) at (0.5* \wid,1*\hei) {};
\node[] (1h2) at (0.5* \wid,2*\hei) {};

\draw[e] (11) to[in =180, out =180] (14);
\draw[e] (13) to[in =180, out =180] (12);

\draw[e] (12) to[in =0, out =0] (13);
\draw[e] (11) to[out = 0, in = 180] (1h1);
\draw[e] (14) to[out = 0, in = 180] (1h2);

\draw[very thick] (0* \wid,-0.25*\hei) -- (0* \wid,3.25*\hei);
\end{tikzpicture}
,
\begin{tikzpicture}[x=1.5cm,y=-.5cm,baseline=-1.05cm]

\def\wid{\standardWidth}
\def\hei{\standardHeight}
\def\nodesize{3}
\def\ang{90}

\node[] (0h1) at (-0.5* \wid,1*\hei) {};
\node[] (0h2) at (-0.5* \wid,2*\hei) {};
\node[v, minimum size=\nodesize] (11) at (0* \wid,0*\hei) {};
\node[v, minimum size=\nodesize] (12) at (0* \wid,1*\hei) {};
\node[v, minimum size=\nodesize] (13) at (0* \wid,2*\hei) {};
\node[v, minimum size=\nodesize] (14) at (0* \wid,3*\hei) {};
\node[] (1h1) at (0.5* \wid,1*\hei) {};
\node[] (1h2) at (0.5* \wid,2*\hei) {};

\draw[e] (11) to[in =180, out =180] (12);
\draw[e] (13) to[in =180, out =180] (14);

\draw[e] (13) to[in =0, out =0] (14);
\draw[e] (11) to[out = 0, in = 180] (1h1);
\draw[e] (12) to[out = 0, in = 180] (1h2);

\draw[very thick] (0* \wid,-0.25*\hei) -- (0* \wid,3.25*\hei);
\end{tikzpicture}
,
\begin{tikzpicture}[x=1.5cm,y=-.5cm,baseline=-1.05cm]

\def\wid{\standardWidth}
\def\hei{\standardHeight}
\def\nodesize{3}
\def\ang{90}

\node[] (0h1) at (-0.5* \wid,1*\hei) {};
\node[] (0h2) at (-0.5* \wid,2*\hei) {};
\node[v, minimum size=\nodesize] (11) at (0* \wid,0*\hei) {};
\node[v, minimum size=\nodesize] (12) at (0* \wid,1*\hei) {};
\node[v, minimum size=\nodesize] (13) at (0* \wid,2*\hei) {};
\node[v, minimum size=\nodesize] (14) at (0* \wid,3*\hei) {};
\node[] (1h1) at (0.5* \wid,1*\hei) {};
\node[] (1h2) at (0.5* \wid,2*\hei) {};

\draw[e] (12) to[in =180, out =180] (11);
\draw[e] (14) to[in =180, out =180] (13);

\draw[e] (12) to[in =0, out =0] (11);
\draw[e] (14) to[out = 0, in = 180] (1h2);
\draw[e] (13) to[out = 0, in = 180] (1h1);

\draw[very thick] (0* \wid,-0.25*\hei) -- (0* \wid,3.25*\hei);
\end{tikzpicture}
\end{center}
\emph{left end pivots}, and we call
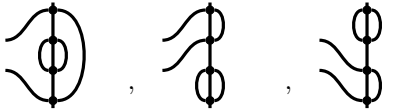
\begin{center}
\begin{tikzpicture}[x=1.5cm,y=-.5cm,baseline=-1.05cm]

\def\wid{\standardWidth}
\def\hei{\standardHeight}
\def\nodesize{3}
\def\ang{90}

\node[] (0h1) at (-0.5* \wid,1*\hei) {};
\node[] (0h2) at (-0.5* \wid,2*\hei) {};
\node[v, minimum size=\nodesize] (11) at (0* \wid,0*\hei) {};
\node[v, minimum size=\nodesize] (12) at (0* \wid,1*\hei) {};
\node[v, minimum size=\nodesize] (13) at (0* \wid,2*\hei) {};
\node[v, minimum size=\nodesize] (14) at (0* \wid,3*\hei) {};
\node[] (1h1) at (0.5* \wid,1*\hei) {};
\node[] (1h2) at (0.5* \wid,2*\hei) {};

\draw[e] (11) to[in =0, out =0] (14);
\draw[e] (12) to[in =0, out =0] (13);

\draw[e] (12) to[in =180, out =180] (13);
\draw[e] (11) to[out = 180, in = 0] (0h1);
\draw[e] (14) to[out = 180, in = 0] (0h2);

\draw[very thick] (0* \wid,-0.25*\hei) -- (0* \wid,3.25*\hei);
\end{tikzpicture}
,
\begin{tikzpicture}[x=1.5cm,y=-.5cm,baseline=-1.05cm]

\def\wid{\standardWidth}
\def\hei{\standardHeight}
\def\nodesize{3}
\def\ang{90}

\node[] (0h1) at (-0.5* \wid,1*\hei) {};
\node[] (0h2) at (-0.5* \wid,2*\hei) {};
\node[v, minimum size=\nodesize] (11) at (0* \wid,0*\hei) {};
\node[v, minimum size=\nodesize] (12) at (0* \wid,1*\hei) {};
\node[v, minimum size=\nodesize] (13) at (0* \wid,2*\hei) {};
\node[v, minimum size=\nodesize] (14) at (0* \wid,3*\hei) {};
\node[] (1h1) at (0.5* \wid,1*\hei) {};
\node[] (1h2) at (0.5* \wid,2*\hei) {};

\draw[e] (11) to[in =0, out =0] (12);
\draw[e] (13) to[in =0, out =0] (14);

\draw[e] (13) to[in =180, out =180] (14);
\draw[e] (11) to[out = 180, in = 0] (0h1);
\draw[e] (12) to[out = 180, in = 0] (0h2);

\draw[very thick] (0* \wid,-0.25*\hei) -- (0* \wid,3.25*\hei);
\end{tikzpicture}
,
\begin{tikzpicture}[x=1.5cm,y=-.5cm,baseline=-1.05cm]

\def\wid{\standardWidth}
\def\hei{\standardHeight}
\def\nodesize{3}
\def\ang{90}

\node[] (0h1) at (-0.5* \wid,1*\hei) {};
\node[] (0h2) at (-0.5* \wid,2*\hei) {};
\node[v, minimum size=\nodesize] (11) at (0* \wid,0*\hei) {};
\node[v, minimum size=\nodesize] (12) at (0* \wid,1*\hei) {};
\node[v, minimum size=\nodesize] (13) at (0* \wid,2*\hei) {};
\node[v, minimum size=\nodesize] (14) at (0* \wid,3*\hei) {};
\node[] (1h1) at (0.5* \wid,1*\hei) {};
\node[] (1h2) at (0.5* \wid,2*\hei) {};

\draw[e] (14) to[in =0, out =0] (13);
\draw[e] (12) to[in =0, out =0] (11);

\draw[e] (12) to[in =180, out =180] (11);
\draw[e] (14) to[out = 180, in = 0] (0h2);
\draw[e] (13) to[out = 180, in = 0] (0h1);

\draw[very thick] (0* \wid,-0.25*\hei) -- (0* \wid,3.25*\hei);
\end{tikzpicture}
\end{center}
\emph{right end pivots.}

\begin{lemma} A graffito with $i \geq 3$ loops and no dividers contains precisely $i-1$ of the letters from \cref{lem:2noAlphabet}. The left end pivots may only occur as the first letter of the graffito, and the right end pivots may only occur as the last letter. \qed \end{lemma}

For a graffito $w$ with $i \geq 3$ loops and no dividers, let $(p_1, \dots , p_{i-1})$ be the ordered list of the $i-1$ pivots. We call this the \emph{pivot sequence} associated to $w$.

\begin{lemma} \label{lem:pivotSequences} Let $w$ be a graffito with $i \geq 3$ loops and no dividers, and let $(p_1, \dots , p_{i-1})$ be the associated pivot sequence. For a face map $d_j$, if $d_j(w)$ is nonzero (in $C_*[i,0]$), then it must have pivot sequence $(p_1', \dots , p_{i-1}')$, where:
\begin{itemize}
    \item $p_i'=p_i$ for $2 \leq i \leq i-2$. That is, the face map can change only the first and last pivots.
    \item The pivot $p_1'$ can differ from $p_1$ only if $p_1'$ is a left end pivot and $p_1$ is not.
    \item The pivot $p_{i-1}'$ can differ from $p_{i-1}$ only if $p_{i-1}'$ is a right end pivot and $p_{i-1}$ is not.
\end{itemize}
\end{lemma}

The point is that the pivot sequence is almost invariant under the differential, except that the leftmost portion of the graffito may `collapse onto' the left hand side of the first pivot, changing the pivot letter, and likewise on the right.

\begin{proof} Deleting a bar at a pivot gives zero by \cref{lem:deletePivotGet0}. It therefore suffices to consider bar deletions which occur at non-pivot letters. Consider such a bar deletion, say at place $j$.

There are two cases. First, assume that this $j$-th letter is neither the first or last letter of the graffito. Since we assume that the graffito contains no dividers, this letter must have two connections on each side, and since it is not a pivot it must be one of the following:
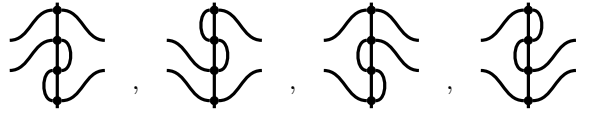
\begin{center}
\begin{tikzpicture}[x=1.5cm,y=-.5cm,baseline=-1.05cm]

\def\wid{\standardWidth}
\def\hei{\standardHeight}
\def\nodesize{3}
\def\ang{90}

\node[] (0h1) at (-0.5* \wid,1*\hei) {};
\node[] (0h2) at (-0.5* \wid,2*\hei) {};
\node[v, minimum size=\nodesize] (11) at (0* \wid,0*\hei) {};
\node[v, minimum size=\nodesize] (12) at (0* \wid,1*\hei) {};
\node[v, minimum size=\nodesize] (13) at (0* \wid,2*\hei) {};
\node[v, minimum size=\nodesize] (14) at (0* \wid,3*\hei) {};
\node[] (1h1) at (0.5* \wid,1*\hei) {};
\node[] (1h2) at (0.5* \wid,2*\hei) {};

\draw[e] (11) to[in =0, out =180] (0h1);
\draw[e] (12) to[in =0, out =180] (0h2);
\draw[e] (13) to[in =180, out =180] (14);

\draw[e] (12) to[in =0, out =0] (13);
\draw[e] (11) to[out = 0, in = 180] (1h1);
\draw[e] (14) to[out = 0, in = 180] (1h2);

\draw[very thick] (0* \wid,-0.25*\hei) -- (0* \wid,3.25*\hei);
\end{tikzpicture}
,
\begin{tikzpicture}[x=1.5cm,y=-.5cm,baseline=-1.05cm]

\def\wid{\standardWidth}
\def\hei{\standardHeight}
\def\nodesize{3}
\def\ang{90}

\node[] (0h1) at (-0.5* \wid,1*\hei) {};
\node[] (0h2) at (-0.5* \wid,2*\hei) {};
\node[v, minimum size=\nodesize] (11) at (0* \wid,0*\hei) {};
\node[v, minimum size=\nodesize] (12) at (0* \wid,1*\hei) {};
\node[v, minimum size=\nodesize] (13) at (0* \wid,2*\hei) {};
\node[v, minimum size=\nodesize] (14) at (0* \wid,3*\hei) {};
\node[] (1h1) at (0.5* \wid,1*\hei) {};
\node[] (1h2) at (0.5* \wid,2*\hei) {};

\draw[e] (13) to[in =00, out =180] (0h1);
\draw[e] (14) to[in =0, out =180] (0h2);
\draw[e] (12) to[in =180, out =180] (11);

\draw[e] (12) to[in =0, out =0] (13);
\draw[e] (11) to[out = 0, in = 180] (1h1);
\draw[e] (14) to[out = 0, in = 180] (1h2);

\draw[very thick] (0* \wid,-0.25*\hei) -- (0* \wid,3.25*\hei);
\end{tikzpicture}
,
\begin{tikzpicture}[x=1.5cm,y=-.5cm,baseline=-1.05cm]

\def\wid{\standardWidth}
\def\hei{\standardHeight}
\def\nodesize{3}
\def\ang{90}

\node[] (0h1) at (-0.5* \wid,1*\hei) {};
\node[] (0h2) at (-0.5* \wid,2*\hei) {};
\node[v, minimum size=\nodesize] (11) at (0* \wid,0*\hei) {};
\node[v, minimum size=\nodesize] (12) at (0* \wid,1*\hei) {};
\node[v, minimum size=\nodesize] (13) at (0* \wid,2*\hei) {};
\node[v, minimum size=\nodesize] (14) at (0* \wid,3*\hei) {};
\node[] (1h1) at (0.5* \wid,1*\hei) {};
\node[] (1h2) at (0.5* \wid,2*\hei) {};

\draw[e] (11) to[in =0, out =180] (0h1);
\draw[e] (14) to[in =0, out =180] (0h2);
\draw[e] (13) to[in =180, out =180] (12);

\draw[e] (13) to[in =0, out =0] (14);
\draw[e] (11) to[out = 0, in = 180] (1h1);
\draw[e] (12) to[out = 0, in = 180] (1h2);

\draw[very thick] (0* \wid,-0.25*\hei) -- (0* \wid,3.25*\hei);
\end{tikzpicture}
,
\begin{tikzpicture}[x=1.5cm,y=-.5cm,baseline=-1.05cm]

\def\wid{\standardWidth}
\def\hei{\standardHeight}
\def\nodesize{3}
\def\ang{90}

\node[] (0h1) at (-0.5* \wid,1*\hei) {};
\node[] (0h2) at (-0.5* \wid,2*\hei) {};
\node[v, minimum size=\nodesize] (11) at (0* \wid,0*\hei) {};
\node[v, minimum size=\nodesize] (12) at (0* \wid,1*\hei) {};
\node[v, minimum size=\nodesize] (13) at (0* \wid,2*\hei) {};
\node[v, minimum size=\nodesize] (14) at (0* \wid,3*\hei) {};
\node[] (1h1) at (0.5* \wid,1*\hei) {};
\node[] (1h2) at (0.5* \wid,2*\hei) {};

\draw[e] (11) to[in =0, out =180] (0h1);
\draw[e] (14) to[in =0, out =180] (0h2);
\draw[e] (12) to[in =180, out =180] (13);

\draw[e] (12) to[in =0, out =0] (11);
\draw[e] (14) to[out = 0, in = 180] (1h2);
\draw[e] (13) to[out = 0, in = 180] (1h1);

\draw[very thick] (0* \wid,-0.25*\hei) -- (0* \wid,3.25*\hei);
\end{tikzpicture}
.
\end{center}

Deleting a bar at any of these letters just deletes the letter, so in particular doesn't change the pivot sequence.

Suppose on the other hand that our non-pivot letter is either the first or last letter of the graffito. Without loss of generality, say that it is the first letter. Since it is not a pivot, it must be one of the following:

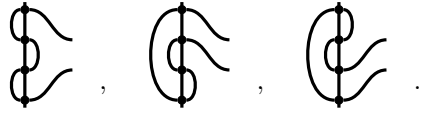
\begin{center}
\begin{tikzpicture}[x=1.5cm,y=-.5cm,baseline=-1.05cm]

\def\wid{\standardWidth}
\def\hei{\standardHeight}
\def\nodesize{3}
\def\ang{90}

\node[] (0h1) at (0.5* \wid,1*\hei) {};
\node[] (0h2) at (0.5* \wid,2*\hei) {};
\node[v, minimum size=\nodesize] (11) at (0* \wid,0*\hei) {};
\node[v, minimum size=\nodesize] (12) at (0* \wid,1*\hei) {};
\node[v, minimum size=\nodesize] (13) at (0* \wid,2*\hei) {};
\node[v, minimum size=\nodesize] (14) at (0* \wid,3*\hei) {};
\node[] (1h1) at (0.5* \wid,1*\hei) {};
\node[] (1h2) at (0.5* \wid,2*\hei) {};

\draw[e] (11) to[in =180, out =180] (12);
\draw[e] (13) to[in =180, out =180] (14);

\draw[e] (12) to[in =0, out =0] (13);
\draw[e] (11) to[out = 0, in = 180] (1h1);
\draw[e] (14) to[out = 0, in = 180] (1h2);

\draw[very thick] (0* \wid,-0.25*\hei) -- (0* \wid,3.25*\hei);
\end{tikzpicture}
,
\begin{tikzpicture}[x=1.5cm,y=-.5cm,baseline=-1.05cm]

\def\wid{\standardWidth}
\def\hei{\standardHeight}
\def\nodesize{3}
\def\ang{90}

\node[] (0h1) at (-0.5* \wid,1*\hei) {};
\node[] (0h2) at (-0.5* \wid,2*\hei) {};
\node[v, minimum size=\nodesize] (11) at (0* \wid,0*\hei) {};
\node[v, minimum size=\nodesize] (12) at (0* \wid,1*\hei) {};
\node[v, minimum size=\nodesize] (13) at (0* \wid,2*\hei) {};
\node[v, minimum size=\nodesize] (14) at (0* \wid,3*\hei) {};
\node[] (1h1) at (0.5* \wid,1*\hei) {};
\node[] (1h2) at (0.5* \wid,2*\hei) {};

\draw[e] (11) to[in =180, out =180] (14);
\draw[e] (13) to[in =180, out =180] (12);

\draw[e] (13) to[in =0, out =0] (14);
\draw[e] (11) to[out = 0, in = 180] (1h1);
\draw[e] (12) to[out = 0, in = 180] (1h2);

\draw[very thick] (0* \wid,-0.25*\hei) -- (0* \wid,3.25*\hei);
\end{tikzpicture}
,
\begin{tikzpicture}[x=1.5cm,y=-.5cm,baseline=-1.05cm]

\def\wid{\standardWidth}
\def\hei{\standardHeight}
\def\nodesize{3}
\def\ang{90}

\node[] (0h1) at (-0.5* \wid,1*\hei) {};
\node[] (0h2) at (-0.5* \wid,2*\hei) {};
\node[v, minimum size=\nodesize] (11) at (0* \wid,0*\hei) {};
\node[v, minimum size=\nodesize] (12) at (0* \wid,1*\hei) {};
\node[v, minimum size=\nodesize] (13) at (0* \wid,2*\hei) {};
\node[v, minimum size=\nodesize] (14) at (0* \wid,3*\hei) {};
\node[] (1h1) at (0.5* \wid,1*\hei) {};
\node[] (1h2) at (0.5* \wid,2*\hei) {};

\draw[e] (14) to[in =180, out =180] (11);
\draw[e] (12) to[in =180, out =180] (13);

\draw[e] (12) to[in =0, out =0] (11);
\draw[e] (14) to[out = 0, in = 180] (1h2);
\draw[e] (13) to[out = 0, in = 180] (1h1);

\draw[very thick] (0* \wid,-0.25*\hei) -- (0* \wid,3.25*\hei);
\end{tikzpicture}
.
\end{center}

Upon deleting this bar, this first letter `collapses onto' the second letter, which becomes the new first letter. The new first letter is almost the old second letter, except that the two connections on the left have been joined together. In particular, if this letter was a pivot, then it has become a left end pivot. The other letters of the word are unchanged.

In total, we have shown that the only way that the pivot sequence can change under a bar deletion is that deleting the first (respectively last) bar may turn the first (respectively last) pivot into a left end (respectively right end) pivot. The result follows.
\end{proof}

Using \cref{lem:pivotSequences}, we may construct a three-step filtration of $C_*[i,0]$ as follows. Let $F^0 C_*[i,0]$ be the subcomplex spanned by graffiti whose pivot sequences begin with a left end pivot and end with a right end pivot. Let $F^1 C_*[i,0]$ be the subcomplex spanned by graffiti whose pivot sequences either begin with a left end pivot or end with a right end pivot, and let $F^2 C_*[i,0]$ be all of $C_*[i,0]$.

\begin{proposition}[\cref{thm:mainTechnical} (\ref{part:igeq3})] \label{prop:igeq3} The homology of each filtration quotient $\faktor{F^p C_*[i,0]}{F^{p-1} C_*[i,0]}$ is zero, so the homology of $C_*[i,0]$ itself must be zero.
\end{proposition}

\begin{proof} The second claim follows from the first using the spectral sequence associated to the filtration, so it suffices to prove the first claim.

Each filtration quotient is a direct sum over pivot sequences. We will write $C_*(p_1, \dots, p_{i-1})$ for the summand associated to the pivot sequence $(p_1, \dots, p_{i-1})$. Since $i \geq 3$, this sequence consists of at least two pivots.

In the filtration quotient, we have by \cref{lem:pivotSequences} that the differential is given by deleting the non-pivot letters, except that if there is only one letter before the first pivot or after the last pivot, then deleting that letter gives zero.

In particular, we can identify $C_*(p_1, \dots, p_{i-1})$ as a shifted tensor product
$$C_*(p_1, \dots, p_{i-1}) \cong (C_*^{R}[1,0] \otimes C_*^{LR}[1,0] \otimes \dots \otimes C_*^{LR}[1,0] \otimes C_*^{L}[1,0])[i-1],$$ where we recall the definitions of the complexes of (two-sided) open graffiti from \cref{def:complexCij}. The degree shift comes from the fact that each pivot contributes 1 to the length of the word.

The homology of $C_*^{LR}[1,0]$ vanishes by \cref{lem:openEnds}, and everything is $R$-free, so the homology of this tensor product also vanishes, and the result follows.
\end{proof}

This completes the verification of \cref{thm:mainTechnical}.

\section{Proof of \cref{thm: main} and \cref{prop: model comparison}} \label{section:basechange}

In this section we will establish the new model, \cref{thm: main}, and the comparison with the old model, \cref{prop: model comparison}. In \cref{section: product and grading,section: spectral sequence,section:filtrationByDividers,section: mainTechnical}, we assumed that $a=0$, so, as discussed in \cref{sectionBar}, we must first reduce to this case. To do so, we follow \cite[Section 3]{BBRWS}. The idea that a pointed ring $(R,a)$ is equivalently a map $\mathbb{Z}[a] \to R$, so the universal pointed ring is $(\mathbb{Z}[a],a)$, which enjoys a natural map to $(\mathbb{Z},0)$. We consider $\mathbb{Z}[a]$ as a graded ring, with $|a|=1$. A \emph{$\mathbb{Z}[a]$-module with additional grading} is then a graded module over the graded ring $\mathbb{Z}[a]$.

\begin{lemma} \label{lemma: reduction techniques} Let $f \colon C \to D$ be a map of chain complexes of finitely generated free $\mathbb{Z}[a]$-modules with additional grading. Let $(R,a)$ be any pointed ring. \begin{enumerate}
    \item If $f \otimes_{\mathbb{Z}[a]} \mathbb{Z}$ is a weak equivalence, then $f$ is a weak equivalence.
    \item If $f$ is a weak equivalence, then $f \otimes_{\mathbb{Z}[a]} R$ is a weak equivalence.
\end{enumerate}
\end{lemma}

\begin{proof} Note first that if $M$ is a finitely generated $\mathbb{Z}[a]$-module with additional grading, and $M \otimes_{\mathbb{Z}[a]} \mathbb{Z}=0$, then $M=0$. This is because all elements of $M$ are divisible by $a$, so $M_i \neq 0$ implies $M_{i-1} \neq 0$, but finite generation means that $M_i=0$ for $i \ll 0$. For (1), we can equivalently show that $H_*(E)=0$, where $E$ is the mapping cone of $f$. Since $C$ and $D$ are degree-wise free, $E$ is also degree-wise free, so $E \xrightarrow{a \cdot-} E$ is injective, and we get a short exact sequence $0 \to E \xrightarrow{a \cdot-} E \to E \otimes_{\mathbb{Z}[a]} \mathbb{Z} \to 0$, hence a homology injection $H_*(E) \otimes_{\mathbb{Z}[a]} \mathbb{Z} = \mathrm{Coker}(H_*(E) \xrightarrow{a \cdot-} H_*(E)) \hookrightarrow H_*(E \otimes_{\mathbb{Z}[a]} \mathbb{Z})$, which is zero by assumption since $E \otimes_{\mathbb{Z}[a]} \mathbb{Z}$ is the mapping cone of $f \otimes_{\mathbb{Z}[a]} \mathbb{Z}$. Setting $M = H_*(E)$ gives the result ($E$ is degree-wise finitely generated because $C$ and $D$ are, so $H_*(E)$ is too because $\mathbb{Z}[a]$ is Noetherian).

For (2), we have by assumption that $E$ is acyclic. Since $E$ is also degree-wise free, we may write $H_*(E \otimes_{\mathbb{Z}[a]} R) = \mathrm{Tor}^{\mathbb{Z}[a]}_*(E,R)=0$, as required.  
\end{proof}

We now prove the main results. For any $(R,a)$ we have defined maps 
$$(T_R[x_1,x_3],d) \xrightarrow{\psi} (T_R[x,\hat{x},r,y],d) \xrightarrow{\varphi} (\CPL(4;R,a),d)$$
and we wish to show that both of these maps are equivalences. We first explain how to give everything an additional grading in the universal case $(\mathbb{Z}[a],a)$. In \cref{section: product and grading}, we defined an additional grading on $\CPL(4;\mathbb{Z},0)$ by counting loops, and we can define an additional grading on $\CPL(4;\mathbb{Z}[a],a)$ in the same way (this works since $\mathbb{Z}[a]$ is graded by setting $|a|=1$, so when the differential turns loops into factors of $a$ the grading is preserved, c.f.~\cite[Definition 2.11]{BBRWS}). To define compatible gradings on the models, recall that $\psi$ is defined by $\psi(x_1) =x$, $\psi(x_3) = y + 2 x r$, and that $\varphi$ is defined by sending the generators to certain sums of diagrams. Examining these diagrams, we see that we may set
$$|x_1|=1, \ |x_3|=2, \ |x|=|\hat{x}|=|r| =1, \textrm{ and }|y|=2,$$
to obtain the desired compatible gradings.

\begin{proof}[Proof of \cref{prop: model comparison}] We wish to show that $\psi$ is a weak equivalence. In the universal case $(R,a) =(\mathbb{Z}[a],a)$, we saw above that both domain and codomain can be given an additional grading. It then follows from \cref{lemma: reduction techniques} that it suffices to prove the result in the special case $(R,a)=(\mathbb{Z},0)$, where we work from now on.

We now claim that it further suffices to show that the restriction $\psi: (T_{\mathbb{Z}}[x_1],d) \to (T_{\mathbb{Z}}[x,\hat{x},r],d)$ is a weak equivalence. To see this, first change basis in the target, writing $T_R[x,\hat{x},r,y] \cong T_R[x,\hat{x},r,y']$ by setting $y'=y + 2 x r$, so that $\psi(x_3) = y'$. The differential can only reduce the number of instances of $x_3$ in a monomial word, and likewise for $y'$, so counting instances of these letters gives compatible increasing filtrations on the two algebras. In the filtration quotients, $x_3$ and $y'$ become cycles, so the map on the filtration quotients may be identified as a tensor power of the above restriction of $\psi$, which is a weak equivalence if and only if the restriction itself is a weak equivalence (by the Künneth theorem, using that both complexes are free over $\mathbb{Z}$). The spectral sequence associated to the filtration then gives the claim.

It remains to establish that the restriction $\psi \colon (T_{\mathbb{Z}}[x_1],d) \to (T_{\mathbb{Z}}[x,\hat{x},r],d)$ is a weak equivalence. To see this, note that both $\psi$ and the restricted differentials preserve word length, so we must verify for each $\ell \geq 0$ that $\psi$ is an equivalence on the subcomplex spanned by words of length $\ell$. In the domain, this subcomplex consists of the single word $x_1^{\ell}$. In the target, it is the $\ell$-fold tensor product of the complex obtained in the case $\ell=1$, which consists just of the three letters, with differential $d(x)=d(\hat{x})=0$, $d(r)=\hat{x}-x$ (since we are in the special case $a=0$). The homology of this $\ell=1$ complex is a single copy of $\mathbb{Z}$ in degree 0, generated by (e.g.) the class of $x$, so the result follows.
\end{proof}

\begin{proof}[Proof of \cref{thm: main}] The map $\varphi$ respects the involutions by construction, so by \cref{prop: model comparison}, we may equivalently show that $\varphi \circ \psi \colon (T_{R}[x_1,x_3],d) \to \CPL(4;R,a)$ is a weak equivalence. As before \cref{lemma: reduction techniques} allows us to reduce to the case $(\mathbb{Z},0)$.

In \cref{section: spectral sequence}, we filtered the complex $\CPL(4;\mathbb{Z},0)$ by number of \emph{non-dividers} $\myDA'(x)$ (\cref{def: filtration}) and studied the resulting spectral sequence. We will first get a compatible filtration on $(T_{R}[x_1,x_3],d)$. Recall that the number of non-dividers $\myDA'(x)$ of a graffito $x$ is defined to be $\myDA'(x)=\deg(x)-1-\myDA(x)$, where $\myDA(x)$ is the number of dividers, and that the filtration by non-dividers is multiplicative by \cref{cor:NonDividersOfProduct}. Consulting the definitions of the maps, we see that $\varphi \circ \psi$ sends $x_1$ (of degree 1) to a graffito with no dividers, and sends $x_3$ (of degree 3) to a sum of graffiti having at most one divider. Letting $x_1$ have filtration degree 0, letting $x_3$ have filtration degree 2, and letting the filtration degree of a monomial be the sum of the filtration degrees of its letters, defines a filtration on $(T_{R}[x_1,x_3],d)$, and $\varphi \circ \psi$ carries this filtration to the non-dividers filtration.

This map of filtered dgas induces a map between the associated spectral sequences. Each of the spectral sequences converges to $E^{\infty}$-page the associated graded of its homology, and it suffices to show that $\varphi \circ \psi$ induces an isomorphism on this associated graded. We will argue the stronger claim that it is actually an isomorphism on $E^2$-pages.

To do this, use \cref{cor: E2 as tensor} to see that the target $E^2$-page is the free nonunital associative algebra on $\varphi(x)$ and $\varphi(y)$. The differential on $(T_{R}[x_1,x_3],d)$ decreases filtration, so is zero on each filtration quotient, and $E^1 = E^0$. Since it further decreases filtration by at least two, the $d^1$-differential is also zero, and $E^2=E^0=T_{R}[x_1,x_3]$. By definition, $\varphi(x) = \varphi \circ \psi(x_1)$, and $\varphi(y) \equiv \varphi \circ \psi(x_3)$ modulo $\varphi(2xr)$, which is of lower filtration. The result follows.
\end{proof}
\printbibliography

\end{document}